\documentclass[a4paper,11pt, leqno]{article}
\usepackage[OT2,T1]{fontenc}
\DeclareSymbolFont{cyrletters}{OT2}{wncyr}{m}{n}
\usepackage[latin1]{inputenc}
\usepackage[intlimits]{amsmath}
\usepackage{amssymb}
\usepackage{amsfonts}
\usepackage{mathrsfs}
\usepackage{amsthm}
\usepackage{colonequals,comment}
\usepackage{mathrsfs}
\usepackage[all]{xy}
\RequirePackage{ifthen}
\RequirePackage{verbatim}

\newskip\procskipamount
\newskip\interskipamount
\newskip\refskipamount
\newcommand{\procskip}{\vskip\procskipamount}
\newcommand{\interskip}{\vskip\interskipamount}

\newcommand{\procbreak}{\par
   \ifdim\lastskip<\procskipamount\removelastskip
   \penalty-100
   \procskip\fi
   \noindent\ignorespaces}

\newcommand{\titlebreak}{\par%
\ifdim\lastskip<\interskipamount\removelastskip%
\penalty10000%
\interskip\fi%
\noindent}%

\newcommand{\interbreak}{\par%
\ifdim\lastskip<\interskipamount\removelastskip%
\penalty-100%
\interskip\fi%
\noindent\ignorespaces}%

\numberwithin{equation}{section}

\theoremstyle{plain}
\newtheorem{theorem}{Theorem}[section]
\newtheorem{lemma}[theorem]{Lemma}
\newtheorem{corollary}[theorem]{Corollary}
\newtheorem{proposition}[theorem]{Proposition}
\newtheorem{intro-theorem}{Theorem}
\newtheorem{intro-corollary}[intro-theorem]{Corollary}
\newtheorem{intro-proposition}[intro-theorem]{Proposition}

\theoremstyle{definition}

\newtheorem{definition}[theorem]{Definition}

\newtheorem{construction}[theorem]{Construction}
\newtheorem{remark}[theorem]{Remark}
\newtheorem{example}[theorem]{Example}

\newcommand{\marginrule}{\marginpar{\rule[-10.5mm]{1mm}{10mm}}}

\newcounter{listcounter}
\newcounter{deflistcounter}
\newcounter{equivcounter}

\newskip{\itemsepamount}
\newskip{\topsepamount}
\newenvironment{assertionlist}{%
  \begin{list}
    {\upshape (\arabic{listcounter})}
    {\setlength{\leftmargin}{18pt}
     \setlength{\rightmargin}{0pt}
     \setlength{\itemindent}{0pt}
     \setlength{\labelsep}{5pt}
     \setlength{\labelwidth}{13pt}
     \setlength{\listparindent}{\parindent}
     \setlength{\parsep}{0pt}
     \setlength{\itemsep}{\itemsepamount}
     \setlength{\topsep}{\topsepamount}
     \usecounter{listcounter}}}
  {\end{list}}

\newenvironment{definitionlist}{%
  \begin{list}
    {\upshape (\alph{deflistcounter})}
    {\setlength{\leftmargin}{18pt}
     \setlength{\rightmargin}{0pt}
     \setlength{\itemindent}{0pt}
     \setlength{\labelsep}{5pt}
     \setlength{\labelwidth}{13pt}
     \setlength{\listparindent}{\parindent}
     \setlength{\parsep}{0pt}
     \setlength{\itemsep}{\itemsepamount}
     \setlength{\topsep}{\topsepamount}
     \usecounter{deflistcounter}}}
  {\end{list}}

\newenvironment{equivlist}{%
  \begin{list}
    {\upshape (\roman{equivcounter})}
    {\setlength{\leftmargin}{18pt}
     \setlength{\rightmargin}{0pt}
     \setlength{\itemindent}{0pt}
     \setlength{\labelsep}{5pt}
     \setlength{\labelwidth}{13pt}
     \setlength{\listparindent}{\parindent}
     \setlength{\parsep}{0pt}
     \setlength{\itemsep}{\itemsepamount}
     \setlength{\topsep}{\topsepamount}
     \usecounter{equivcounter}}}
  {\end{list}}

\newenvironment{bulletlist}{%
  \begin{list}
    {\upshape \textbullet}
    {\setlength{\leftmargin}{18pt}
     \setlength{\rightmargin}{0pt}
     \setlength{\itemindent}{0pt}
     \setlength{\labelsep}{6pt}
     \setlength{\labelwidth}{12pt}
     \setlength{\listparindent}{\parindent}
     \setlength{\parsep}{0pt}
     \setlength{\itemsep}{\itemsepamount}
     \setlength{\topsep}{\topsepamount}}}
  {\end{list}}

% ...

\newenvironment{simplelist}{%
  \begin{list}{}
     {\setlength{\leftmargin}{10pt}
      \setlength{\rightmargin}{0pt}
      \setlength{\itemindent}{0pt}
      \setlength{\labelsep}{5pt}
      \setlength{\listparindent}{\parindent}
      \setlength{\parsep}{0pt}
      \setlength{\itemsep}{0pt}
      \setlength{\topsep}{0pt}}}
  {\end{list}}

\newenvironment{spacedlist}{%
  \begin{list}{}
     {\setlength{\leftmargin}{25pt}
      \setlength{\rightmargin}{0pt}
      \setlength{\itemindent}{0pt}
      \setlength{\labelsep}{5pt}
      \setlength{\listparindent}{\parindent}
      \setlength{\parsep}{0pt}
      \setlength{\itemsep}{6pt}
      \setlength{\topsep}{6pt}}}
  {\end{list}}

\renewcommand{\AA}{\mathbb{A}}

\newcommand{\CC}{\mathbb{C}}

\newcommand{\FF}{\mathbb{F}}
\newcommand{\GG}{\mathbb{G}}
\newcommand{\HH}{\mathbb{H}}
\newcommand{\II}{\mathbb{I}}

\newcommand{\OO}{\mathbb{O}}

\newcommand{\QQ}{\mathbb{Q}}
\newcommand{\RR}{\mathbb{R}}

\newcommand{\VV}{\mathbb{V}}
\newcommand{\WW}{\mathbb{W}}

\newcommand{\ZZ}{\mathbb{Z}}

\renewcommand{\hbar}{\bar{h}}

\newcommand{\kbar}{\bar{k}}

\newcommand{\sbar}{\bar{s}}

\newcommand{\Gbar}{\bar{G}}

\newcommand{\kgbar}{\bar{\kappa}}

\newcommand{\mgbar}{\bar{\mu}}

\newcommand{\QQbar}{\overline{\QQ}}

\newcommand{\Gbf}{{\bf G}}

\newcommand{\Ical}{{\mathcal I}}

\newcommand{\Ncal}{{\mathcal N}}
\newcommand{\Ocal}{{\mathcal O}}
\newcommand{\Pcal}{{\mathcal P}}
\newcommand{\Qcal}{{\mathcal Q}}

\newcommand{\Scal}{{\mathcal S}}
\newcommand{\Tcal}{{\mathcal T}}

\newcommand{\Vcal}{{\mathcal V}}

\newcommand{\gfr}{{\mathfrak g}}

\newcommand{\mfr}{{\mathfrak m}}

\newcommand{\Mline}{\underline{M}}

\newcommand{\Ascr}{{\mathscr A}}
\newcommand{\Bscr}{{\mathscr B}}

\newcommand{\Dscr}{{\mathscr D}}

\newcommand{\Gscr}{{\mathscr G}}
\newcommand{\Hscr}{{\mathscr H}}

\newcommand{\Lscr}{{\mathscr L}}
\newcommand{\Mscr}{{\mathscr M}}

\newcommand{\Oscr}{{\mathscr O}}
\newcommand{\Pscr}{{\mathscr P}}

\newcommand{\Tscr}{{\mathscr T}}

\newcommand{\Zscr}{{\mathscr Z}}

\newcommand{\Ascrtilde}{\tilde\Ascr}

\newcommand{\gtilde}{\tilde{g}}

\newcommand{\Ltilde}{\tilde{L}}
\newcommand{\Mtilde}{\tilde{M}}

\newcommand{\Ptilde}{\tilde{P}}

\newcommand{\Rtilde}{\tilde{R}}

\newcommand{\Wtilde}{\tilde{W}}
\newcommand{\Xtilde}{\tilde{X}}
\newcommand{\Ytilde}{\tilde{Y}}

\newcommand{\agtilde}{\tilde\alpha}

\newcommand{\zgtilde}{\tilde\zeta}

\newcommand{\eps}{\varepsilon}

\DeclareMathOperator{\Aut}{Aut}

\DeclareMathOperator{\Bor}{Bor}

\DeclareMathOperator{\codim}{codim}

\DeclareMathOperator{\der}{der}

\DeclareMathOperator{\DR}{DR}
\DeclareMathOperator{\End}{End}

\DeclareMathOperator{\Frac}{Frac}
\DeclareMathOperator{\Frob}{Frob}
\DeclareMathOperator{\Gal}{Gal}
\DeclareMathOperator{\GL}{GL}

\DeclareMathOperator{\GSp}{GSp}

\DeclareMathOperator{\Hom}{Hom}

\DeclareMathOperator{\id}{id}
\renewcommand{\Im}{\mathop{\rm Im}}

\DeclareMathOperator{\Isom}{Isom}
\newcommand{\Isomline}{\underline{\Isom}}

\DeclareMathOperator{\Ker}{Ker}

\DeclareMathOperator{\Lie}{Lie}

\renewcommand{\mod}[1]{\allowbreak \mkern5mu {\rm mod}\,\,#1}

\DeclareMathOperator{\Norm}{Norm}

\DeclareMathOperator{\Par}{Par}

\DeclareMathOperator{\relpos}{relpos}

\DeclareMathOperator{\Res}{Res}

\DeclareMathOperator{\sgn}{sgn}
\DeclareMathSymbol{\Sha}{\mathalpha}{cyrletters}{"58}

\DeclareMathOperator{\Sp}{Sp}
\DeclareMathOperator{\Spec}{Spec}

\newcommand{\ad}{\textnormal{ad}}

\newcommand{\cross}{{}^{\times}}
\newcommand{\dlbrack}{\mathopen{[\![}}

\newcommand{\drbrack}{\mathclose{]\!]}}

\newcommand{\limproj}{\mathop{\lim\limits_{\longleftarrow}}}
\newcommand{\lrangle}{{\langle\ ,\ \rangle}}
\renewcommand{\mod}[1]{\allowbreak \mkern5mu {\rm mod}\,\,#1}

\renewcommand{\star}{{}^*}

\newcommand{\vdual}{{}^{\vee}}

\newcommand\addots{\mathinner{\mkern1mu\raise0pt\vbox{\kern7pt\hbox{.}}\mkern2mu\raise3pt\hbox{.}\mkern2mu\raise6pt\hbox{.}\mkern1mu}}

\newcommand{\restricted}[1]{{}_{\vert}{}_{#1}}

\newcommand{\set}[2]{\{\,#1\ ;\  #2\,\}}

\newcommand\lto{\longrightarrow}
\newcommand\ltoover[1]{\mathrel{\smash{\overset{#1}{\lto}}}}
\newcommand\varto[1]{\mathrel{\hbox to #1pt{\rightarrowfill}}}
\newcommand\vartoover[2]{\mathrel{\smash{\overset{#2}{\varto{#1}}}}}

\newcommand{\bijective}{\leftrightarrow}

\newcommand{\epi}{\twoheadrightarrow}

\newcommand{\mono}{\hookrightarrow}

\newcommand{\sends}{\mapsto}

\newcommand{\iso}{\overset{\sim}{\to}}
\newcommand{\liso}{\overset{\sim}{\lto}}

\renewcommand{\implies}{\Rightarrow}
\renewcommand{\iff}{\Leftrightarrow}

\newcommand{\DZip}{\Dscr{-}\Zscr ip}
\newcommand{\APlus}{(A,\iota,\lambda,\eta)}
\newcommand{\APlusOne}{(A_1,\iota_1,\lambda_1,\eta_1)}
\newcommand{\APlusTwo}{(A_2,\iota_2,\lambda_2,\eta_2)}

\newcommand{\leftexp}[2]{{\vphantom{#2}}^{#1}{#2}}
\newcommand{\doubleexp}[3]{\leftexp{#1}{#2}^{#3}}
\newcommand{\dbrack}[1]{\dlbrack #1 \drbrack}

\begin{document}

\title{Ekedahl-Oort and Newton strata\\ for Shimura varieties of PEL type}

\author{Eva Viehmann \& Torsten Wedhorn}

\maketitle

\begin{abstract}
We study the Ekedahl-Oort stratification for good reductions of Shimura varieties of PEL type. These generalize the Ekedahl-Oort strata defined and studied by Oort for the moduli space of principally polarized abelian varieties (the ``Siegel case''). They are parameterized by certain elements $w$ in the Weyl group of the reductive group of the Shimura datum. We show that for every such $w$ the corresponding Ekedahl-Oort stratum is smooth, quasi-affine, and of dimension $\ell(w)$ (and in particular non-empty). Some of these results have previously been obtained by Moonen, Vasiu, and the second author using different methods. We determine the closure relations of the strata.

We give a group-theoretical definition of minimal Ekedahl-Oort strata generalizing Oort's definition in the Siegel case and study the question whether each Newton stratum contains a minimal Ekedahl-Oort stratum. As an interesting application we determine which Newton strata are non-empty. This criterion proves conjectures by Fargues and by Rapoport generalizing a conjecture by Manin for the Siegel case. We give a necessary criterion when a given Ekedahl-Oort stratum and a given Newton stratum meet. 
\end{abstract}


\section*{Introduction}

\subsection*{Starting point: The Siegel case}
Fix a prime $p > 0$. For $g \geq 1$ let $\Ascr_g$ be the moduli space of principally polarized abelian varieties of dimension $g$ in characteristic $p$. We consider two natural and well studied stratifications on $\Ascr_g$, the Newton and the Ekedahl-Oort stratification. The Newton stratification is the stratification corresponding to the isogeny class of the underlying $p$-divisible groups (with quasi-polarization) $(A,\lambda)[p^{\infty}]$ of the points $(A,\lambda)$ of $\Ascr_g$. By Dieudonn\'e theory the set of isogeny classes is parameterized by the attached symmetric (concave) Newton polygons. The Newton stratum in $\Ascr_g$ attached to a symmetric Newton polygon $\nu$ is denoted by $\Ncal_\nu$. The Ekedahl-Oort stratification is defined by the isomorphism class of the $p$-torsion $(A,\lambda)[p]$ or, equivalently, the isomorphism class of the first de Rham cohomology endowed with its structure of an $F$-zip (\cite{MW}). The set of strata is either parameterized by so-called elementary sequences (by Oort~\cite{Oo1}) or by certain elements in the Weyl group of the group $GSp_{2g}$ of symplectic similitudes (by Moonen~\cite{Mo1}, see also~\cite{MW}). By the pioneering work of Oort a lot is known about these two stratifications (see~\cite{Oo1} -- \cite{Oort_Simple}), for instance:
\begin{assertionlist}
\item
The Ekedahl-Oort strata are smooth, quasi-affine, and equi-dimensional and their dimension can be easily calculated in terms of the corresponding elementary sequence. The closure of an Ekedahl-Oort stratum is a union of Ekedahl-Oort strata -- although the question which Ekedahl-Oort strata appear in the closure of a given one remained open.
\item
Oort defines the notion of a minimal $p$-divisible group and shows that it is already determined by its $p$-torsion (up to isomorphism). The corresponding Ekedahl-Oort strata are called minimal as well. Due to his definition of minimality it is clear that every Newton stratum contains a unique minimal Ekedahl-Oort stratum.
\item
For every symmetric Newton polygon $\nu$ the corresponding Newton stratum $\Ncal_\nu$ is non-empty (``Manin conjecture''). One even has that every principally polarized $p$-divisible group is the $p$-divisible group of a principally polarized abelian variety.
\item
The closure of a Newton stratum $\Ncal_{\nu}$ is the union of those Newton strata $\Ncal_{\nu'}$ whose Newton polygon $\nu'$ lies below $\nu$, where we normalize Newton polygons such that they are concave (``Grothendieck conjecture'').
\end{assertionlist}
Note that this is not a complete list. For some of these results generalizations are known for arbitrary PEL-data using different methods. Others have been generalized to other special cases of PEL-data, see below for details.

\subsection*{Main results}
The goal of this article is to generalize and study the above notions for good reductions of general Shimura varieties of PEL type. 

Let $\Dscr$ denote a Shimura PEL-datum unramified at a prime $p > 0$, let $G$ be the attached reductive group scheme over $\ZZ_p$, and let $\mu$ be the conjugacy class of one-parameter subgroups defined by $\Dscr$ (see Section~\ref{sec2.1}). We assume that the fibers of $G$ are connected, i.e., we exclude the case~(D) in Kottwitz' terminology (recalled in Remark~\ref{CasesB}). Let $\Ascr_\Dscr$ be the corresponding moduli space of PEL type defined by Kottwitz. We denote its reduction modulo $p$ by $\Ascr_0$. It is defined over a finite field $\kappa$ determined by $\Dscr$ and it classifies tuples $\APlus$, where $A$ is an abelian variety, where $\iota$ is an action on $A$ of an order in a semi-simple $\QQ$-algebra $B$ given by $\Dscr$, where $\lambda$ is a similitude class of prime-to-$p$ polarizations on $A$, and where $\eta$ is a prime-to-$p$-level structure (see Section~\ref{ModSpaceAbVar} for details). For $B = \QQ$ one obtains the Siegel case, i.e., the moduli space of principally polarized abelian varieties.

The Ekedahl-Oort stratification is defined by attaching to every $\APlus$ as above the first de Rham homology $H_1^{\DR}(A)$ (which is by definition the dual of the first hypercohomology of the de Rham complex) endowed with its $F$-zip structure and the additional structure induced by $\iota$ and $\lambda$. For the sake of brevity we call such an object a $\Dscr$-zip (see Section~\ref{DZipToAV}). Over a perfect field, covariant Dieudonn\'e theory and a result of Oda show that it is equivalent to give the datum of the $\Dscr$-zip of $\APlus$ or the $p$-torsion of $A$ together with the structure induced by $\iota$ and $\lambda$ (Example~\ref{DieudonneZip}). The moduli space of all $\Dscr$-zips is a smooth Artin stack of finite type over the finite field $\kappa$, which we denote by $\DZip$. The above construction defines a morphism
\[
\zeta\colon \Ascr_0 \to \DZip.
\]
Let $\kgbar$ be an algebraic closure of $\kappa$. We consider the isomorphism class $w$ of a $\kgbar$-valued point of $\DZip$ (or, equivalently, the isomorphism class of a $\Dscr$-zip over $\kgbar$) as a point of the underlying topological space of $\DZip \otimes \kgbar$. The corresponding Ekedahl-Oort stratum $\Ascr_0^w$ is then defined as $\zeta^{-1}(w)$.

Isomorphism classes of $\Dscr$-zips can be described as follows. Let $(W,I)$ be the Weyl group of the reductive group $G$ together with its set of simple reflections. To the one-parameter subgroup $\mu$ we attach a subset $J \subseteq I$ of simple reflections (see last subsection of Section~\ref{sec2.1}). By the main result of \cite{MW}, the underlying topological space of $\DZip \otimes \kgbar$ is in bijection to $\leftexp{J}{W}$, where
\[
\leftexp{J}{W} := \set{w \in W}{\text{$\ell(sw) > \ell(w)$ for all $s \in J$}}
\]
is the set of representatives of minimal length of $W_J \backslash W$ (Appendix~\ref{RecallCoxeter} and~\ref{BruhatOrder}). Similar classification results have been obtained by B.~Moonen (\cite{Mo1}) and, more generally, by A.~Vasiu (\cite{Vasiu:ModPClassification}). 

By transport of structure we obtain a topology on $\leftexp{J}{W}$ which can be described combinatorially by results of Pink, Ziegler and the second author (\cite{PWZ}) via a partial order $\preceq$ on $\leftexp{J}{W}$ which was introduced by He~\cite{He_GStablePieces} in his study of the spherical completions of reductive groups (see Definition~\ref{SpecializationOrder}). 
This description can also be derived from results of the first author (\cite{trunc1}). In particular the closures and the codimension of points are known. These results are recalled in Section~\ref{GroupDZip}.

\bigskip

The first goal of this paper is the study of the Ekedahl-Oort stratification and the morphism $\zeta$. We show (Theorem~\ref{zetaflat}, Theorem~\ref{thmEOnonempty}):

\begin{intro-theorem}\label{IntroZetaFlat}
The morphism $\zeta$ is faithfully flat.
\end{intro-theorem}

For the proof of the flatness we use a valuative criterion for universal openness (Proposition~\ref{opencrit}). Our proof that $\zeta$ satisfies this valuative criterion relies on the theory of Dieudonn\'e displays developed by Zink (\cite{Zink_Dieudonne}) and Lau (\cite{Lau_Duality}, \cite{Lau_DieudonneDisplays}), recalled in Section~3.1.

The theorem shows in particular that $\zeta$ is open and surjective and we deduce from the known facts on the topology of $\leftexp{J}{W}$ the following result (Theorem~\ref{thmEOnonempty}, Corollary~\ref{DimOfEO}, Theorem~\ref{corclosure}):

\begin{intro-theorem}\label{IntroEO1}
\begin{assertionlist}
\item
The Ekedahl-Oort stratum $\Ascr^w_0$ is non-empty for all $w \in \leftexp{J}{W}$.
\item
$\Ascr^w_0$ is equi-dimensional of dimension $\ell(w)$.
\item
The closure of $\Ascr^w_0$ is the union of the Ekedahl-Oort strata $\Ascr^{w'}_0$ for those $w' \in \leftexp{J}{W}$ such that $w' \preceq w$.
\end{assertionlist}
\end{intro-theorem}
In particular there exists a unique Ekedahl-Oort stratum of dimension $0$ (corresponding to $w = 1$), which we call the superspecial Ekedahl-Oort stratum.

Assertion~(2) had already been proved by Moonen~\cite{Mo2} for $p > 2$ and under the assumption of Assertion~(1). Assertion~(3) is new even in the Siegel case. In fact for the Siegel case it was also claimed in the unpublished preprint \cite{Wd_SpecFZip} which however relied on the unpublished note \cite{Wd_FlatSiegel}, and this note contained a gap because it used some results of~\cite{Lau_Duality} which were published only four years later. We complete our picture on general Ekedahl-Oort strata by the following result (Proposition~\ref{EOSmooth}, Theorem~\ref{EOQuasiAffine}).

\begin{intro-theorem}\label{IntroEO2}
For all $w \in \leftexp{J}{W}$ the Ekedahl-Oort stratum $\Ascr^w_0$ is smooth and quasi-affine.
\end{intro-theorem}

The smoothness was already shown by Vasiu~\cite{Vasiu_Levelm} and we include only a quick alternative proof for $p > 2$.

\bigskip

Our second goal is the generalization and group-theoretic reformulation of Oort's notion of a minimal $p$-divisible group or a minimal Ekedahl-Oort stratum. For an abelian variety with  endomorphism, polarization and level structure $\APlus$ as above, $\iota$ and $\lambda$ induce additional structures on the $p$-divisible group $A[p^{\infty}]$ of $A$. We obtain the notion of a $p$-divisible group with $\Dscr$-structure (see Definition~\ref{DefDPdiv}). Let $k$ be an algebraically closed extension of the field of definition $\kappa$ of $\Ascr_0$. Let $L := \Frac W(k)$ be the field of fractions of the ring of Witt vectors endowed with the Frobenius, denoted by $\sigma$. Let $K=G(W(k))$. Covariant Dieudonn\'e theory yields a bijection
\begin{equation*}\label{IntroPDiv}
\left\{\begin{matrix}\text{isomorphism classes of $p$-divisible}\\
\text{groups with $\Dscr$-structure over $k$}\end{matrix}\right\}
\bijective C_k(G,\mu) :=
\left\{\begin{matrix}
\text{$K$-$\sigma$-conjugacy classes}\\
\text{in $K\mu(p)K$}
\end{matrix}\right\}
\end{equation*}
(see Section~\ref{CGBG} for details). We define minimality of a $p$-divisible group with $\Dscr$-structure $X$ by a group-theoretic criterion for the corresponding element in $C_k(G,\mu)$ (Definition~\ref{defminimal}) which is equivalent to the condition that $X$ is up to isomorphism already determined by its $p$-torsion (see Remark~\ref{ExplainMinimal}). To study the notion of minimality we also define the technical condition for $X$ to be fundamental (Definition~\ref{defpfund}) which is related to the definition of fundamental elements in the affine Weyl group $\Wtilde$ of $G$ given by Grtz, Haines, Kottwitz, and Reuman in~\cite{GHKR2} (although we slightly deviate from their definition). If $G$ is split (i.e., a product of groups of the form $GL_n$ and $GSp_{2g}$), we show the following result (Section~\ref{SecMinFund}).

\begin{intro-proposition}\label{IntroMinimal}
Assume that $G$ is split. Let $X$ be a $p$-divisible group with $\Dscr$-structure over an algebraically closed field.
\begin{assertionlist}
\item
If $X$ is minimal in the sense of Oort, then it is fundamental.
\item
If $X$ is fundamental, then it is minimal (in our sense).
\end{assertionlist}
\end{intro-proposition}

The proof of the second assertion relies essentially on a result of~\cite{GHKR2}. Both assertions together give in particular a new group-theoretic proof of the main result of~\cite{Oort2}. In fact, Theorem~0.2 of~\cite{Oort_Simple} can then be used to show that all three conditions are equivalent if $G = \GL_n$. For $G = \GSp_{2g}$ the same assertion follows by using uniqueness of polarizations (\cite{Oort1}~Corollary~(3.8)). Thus all three conditions are equivalent if $G$ is split.

\bigskip

The third goal of this paper is to relate Ekedahl-Oort strata and Newton strata. The group-theoretic definition of Newton strata is due to Rapoport and Richartz (\cite{RapoportRichartz}). Similarly as above it is well-known that there is an injection from the set of isogeny classes of $p$-divisible groups with $\Dscr$-structure over an algebraically closed field $k$ into the set $B(G)$ of $G(L)$-$\sigma$-conjugacy classes in $G(L)$. By Kottwitz' description of $B(G)$ this set is independent of $k$ and by attaching to $\APlus$ the isogeny class of the associated $p$-divisible group with $\Dscr$-structure we obtain a map ${\rm Nt}\colon \Ascr_0 \to B(G)$ whose image is known to be contained in a certain finite subset $B(G,\mu)$. The set $B(G)$ (and hence its subset $B(G,\mu)$) is endowed with a partial order which is the group-theoretic reformulation of the partial order on the set of (concave) Newton polygons of ``one lying above the other'' (see Section~\ref{CGBG} for details). Then the sets $\Ncal_b := {\rm Nt}^{-1}(b)$ for $b \in B(G,\mu)$ are called the Newton strata of $\Ascr_0$. They are locally closed in $\Ascr_0$. The stratum corresponding to the minimal element of $B(G,\mu)$ is called the \emph{basic stratum}. We show (Theorem~\ref{MinimalNewtonSplit}, Proposition~\ref{remminbasic}, Theorem \ref{thmexfundalc}):

\begin{intro-theorem}\label{IntroEONewton1}
\begin{assertionlist}
\item\label{1}
If $G$ is split, then every Newton stratum $\Ncal_b$ contains a unique fundamental Ekedahl-Oort stratum $\Ascr^{w(b)}_0$ (where an Ekedahl-Oort stratum $\Ascr_0^w$ is called fundamental if the $p$-divisible group with $\Dscr$-structure attached to one (or, equivalently, all) points $\APlus$ of $\Ascr_0^w$ is fundamental).
\item
In general, the basic Newton stratum contains the superspecial Ekedahl-Oort stratum, and this stratum is minimal.
Furthermore, for each Newton stratum $\Ncal_b$ there is a Newton stratum $\Ncal'_b$ associated with the same $b\in B(G)$ but possibly in a moduli space associated with a different $\mu$ which contains a minimal Ekedahl-Oort stratum.
\end{assertionlist}
\end{intro-theorem}

The uniqueness assertion made for split groups does not hold for general $G$. An example for Hilbert-Blumenthal varieties is given in \ref{exhb}. We do not know whether the existence statement of (\ref{1}) holds in general.

We use this theorem to prove ``Manin's conjecture'' (i.e., the non-emptiness of all Newton strata) for the PEL-case and also an integral version of it:

\begin{intro-theorem}\label{IntroManin}
\begin{assertionlist}
\item
For every $b \in B(G,\mu)$ the Newton stratum $\Ncal_b$ is non-empty.
\item
For every algebraically closed field extension $k$ of $\kappa$ and for every $p$-divisible group $X$ with $\Dscr$-structure over $k$ there exists a $k$-valued point of $\Ascr_0$ whose attached $p$-divisible group with $\Dscr$-structure is isomorphic to $X$.
\end{assertionlist}
\end{intro-theorem}

This was conjectured by Fargues (\cite{Fargues}~Conjecture~3.1.1), and by Rapoport (\cite{RapoportGuide}~Conjecture~7.1). ``Manin problems'' have also been considered by Vasiu in~\cite{Vasiu_Manin} using a different language.

We use Theorem~\ref{IntroEONewton1} to show in the split case two results about the intersection of Newton and Ekedahl-Oort strata. In the Siegel case the first one (Theorem~\ref{corEONP}) had been conjectured by Oort (\cite{Oort1}) and proved by Harashita (\cite{H3}):

\begin{intro-theorem}\label{IntroEONewton3}
Assume that $G$ is split. Then for all $w \in \leftexp{J}{W}$ with $\Ascr^w_0 \cap \Ncal_b \ne \emptyset$ one has $w(b) \preceq w$ where $w(b)$ is defined as in Theorem \ref{IntroEONewton1}.
\end{intro-theorem}

The second result is a generalization of a theorem of Harashita (\cite{H3}) for the Siegel case which describes the Newton polygon of the generic point of some irreducible component of a given Ekedahl-Oort stratum (see Proposition~\ref{cor812} for a precise statement).

Finally we deduce from Theorem~\ref{IntroEONewton3} the following result (Corollary~\ref{corclosleaf1}).

\begin{intro-proposition}\label{IntroEONewton2}
Let $G$ be split. For $b,b' \in B(G,\mu)$ with $b' \leq b$ the minimal Ekedahl-Oort stratum $\Ascr^{w(b')}_0$ is contained in the closure of $\Ascr^{w(b)}_0$.
\end{intro-proposition}

In the Siegel case this has been shown by Harashita (\cite{H2}, Corollary 3.2) and is used as a step in his proof of Theorem \ref{IntroEONewton3}.
This gives in particular a new proof of the ``weak Grothendieck conjecture'' (\cite{RapoportGuide}~Conjecture~(7.3)(1)) in the split case:

\begin{intro-corollary}\label{IntroGrothendieck}
Let $G$ be split. For $b,b' \in B(G,\mu)$ with $b' \leq b$ the Newton stratum $\Ncal_{b'}$ meets the closure of $\Ncal_b$.
\end{intro-corollary}

\bigskip

Beyond the Siegel case some of the results above were known for special cases of Shimura varieties of PEL type. The Ekedahl-Oort stratification for Hilbert-Blumenthal varieties (i.e.~with the notation of Section~\ref{sec2.1}, $(B,\star,V) = (F,\id,F^2)$, where $F$ is a totally real field extension of $\QQ$ of degree $g$) has been studied by Goren and Oort (\cite{GO}). In this case $W = (S_2)^g$ (where $S_2 = \{\id,\tau\}$ is the symmetric group of two elements), $J = \emptyset$, the order $\preceq$ equals the Bruhat order which is the $g$-fold product of the order on $S_2$ with $\id < \tau$. They show Theorem~\ref{IntroEO1}, Theorem~\ref{IntroEO2} and the analogue of Proposition~\ref{cor812} below in this case.

The ``linear case of signature $(1,n-1)$'' (in the notation of Section~\ref{sec2.1}: the derived group of $\Gbf_{\RR}$ is isomorphic to a product of ${\rm SU}(1,n-1)$ and compact groups, and the prime $p$ is chosen such that only the case (AL) in the classification in Remark~\ref{CasesOB} occurs) has been studied by Harris and Taylor (\cite{HT}) for their proof of the local Langlands conjecture. In this case the Ekedahl-Oort and the Newton stratification coincide, and all Ekedahl-Oort strata are minimal.

The ``unitary case of signature $(1,n-1)$'' (in the notation of Section~\ref{sec2.1}: the derived group of $\Gbf_{\RR}$ is isomorphic to ${\rm SU}(1,n-1)$, and the prime $p$ is chosen such that only the case (AU) in the classification in Remark~\ref{CasesOB} occurs) has been studied by Bltel and the second author (\cite{BW_Unitary}, see also~\cite{VW}). In this case $W = S_n$, $J = \{\tau_2,\dots,\tau_{n-1}\}$ (where $\tau_i$ denotes the transposition of $i$ and $i+1$), the partial order $\preceq$ on $\leftexp{J}{W}$ coincides with the Bruhat order and it is a total order. Every non-basic Newton stratum is a minimal Ekedahl-Oort stratum, and the basic Newton stratum consists of $[(n-1)/2]$ Ekedahl-Oort strata.

\tableofcontents

\bigskip

\noindent{\itshape Notation}. Throughout the paper we use the following notation. Let $G$ be a group, $X \subset G$ a subset and $g \in G$. Then we set $\leftexp{g}{X} = gXg^{-1}$.

For a commutative ring $R$, an $R$-scheme $X$, and a commutative $R$-algebra $R'$ we denote by $X_{R'} := X \otimes_R R' := X \times_{\Spec R} \Spec R'$ the base change to $R'$.

\bigskip

\noindent{\it Acknowledgements.} We thank E. Lau for providing a proof of Lemma \ref{ZinkHensel} which is much simpler than our original one. Further we thank F.~Oort for helpful discussions and D.~Wortmann and the referee for helpful remarks. The first author was partially supported by the SFB/TR45 ``Periods, Moduli Spaces and Arithmetic of Algebraic Varieties'' of the DFG and by ERC Starting Grant 277889 ``Moduli spaces of local $G$-shtukas''. The second author was partially supported by the SPP1388 ``Representation Theory''.

\bigskip


\section{Good Reductions of Shimura varieties}

\subsection{Shimura data of PEL type}\label{sec2.1}
In this section we recall the notion of Shimura-PEL-data and their attached moduli spaces. Our main reference is Kottwitz~\cite{Ko_ShFin}. 

\subsubsection*{Shimura datum}
Let $\Dscr = \bigl(B,\star,V,\lrangle, O_B, \Lambda, h\bigr)$ denote a Shimura-PEL-datum, integral and unramified at a prime~$p > 0$. Let $\Gbf$ be the associated reductive group over~$\QQ$, and denote by $[\mu]$ the associated conjugacy class of cocharacters of~$\Gbf$. By this we mean the following data.
\begin{bulletlist}
\item
$B$ is a finite-dimensional semi-simple $\QQ$-algebra, such that $B_{\QQ_p}$ is isomorphic to a product of matrix algebras over unramified extensions of~$\QQ_p$.
\item
$\star$ is a $\QQ$-linear positive involution on~$B$.
\item
$V$ is a finitely generated faithful left $B$-module.
\item
$\lrangle\colon V \times V \to \QQ$ is a symplectic form on~$V$ such that $\langle bv,w \rangle = \langle v,b^*w \rangle$ for all $v,w \in V$ and $b \in B$.
\item
$O_B$ is a $\star$-invariant $\ZZ_{(p)}$-order of~$B$ such that $O_B \otimes \ZZ_p$ is a maximal $\ZZ_p$-order of $B \otimes \QQ_p$.
\item
$\Lambda$ is an $O_B$-invariant $\ZZ_p$-lattice in~$V_{\QQ_p}$, such that $\lrangle$ induces a perfect pairing $\Lambda \times \Lambda \to \ZZ_p$. 
\item
$\Gbf$ is the $\QQ$-group of $B$-linear symplectic similitudes of~$(V,\lrangle)$, i.e., for
any $\QQ$-algebra~$R$ we have
$$
\Gbf(R) = \bigl\{g \in \GL_B(V \otimes R) \bigm| \langle gv,gw 
\rangle = c(g) \cdot \langle
v,w \rangle\ \hbox{for some}\ c(g) \in R\cross \bigr\}\, ;
$$
\item
$h\colon \Res_{\CC/\RR}(\GG_{m,\CC}) \to \Gbf_{\RR}$ is a homomorphism that defines a Hodge structure of type $(-1,0) + (0,-1)$ on $V \otimes \RR$ such that there exists a square root $\sqrt{-1}$ of $-1$ such that $2\pi \sqrt{-1}\lrangle$ is a polarization form; let $\mu_h\colon \GG_{m,\CC} \to \Gbf_{\CC}$ be the cocharacter such that $\mu_h(z)$ acts on $V^{(-1,0)}$ (resp.~$V^{(0,-1)}$) via $z$ (resp.~via $1$).
\item
$[\mu]$ is the $G(\CC)$-conjugacy class of the cocharacter~$\mu_h$ associated with~$h$. Then $V_{\CC}$ has only weights 0 and~1 with respect to any $\mu \in [\mu]$.
\end{bulletlist}

\medskip

\noindent Let $B = \prod_i B_i$ be a decomposition of $B$ into simple $\QQ$-algebras. Because of the positivity of the involution, $\star$ induces an involution on each simple factor $B_i$.

\begin{remark}\label{CasesB}
Let $B$ be simple. Let $F$ be the center of $B$ and set $F_0 := \set{a \in F}{a\star = a}$. Then $F_0$ is a totally real field extension of $\QQ$. Let $C := \End_B(V)$ and let $n := [F : F_0]\dim_F(C)^{1/2}/2$ (which is an integer by the existence of $h$). Then $C$ is a central simple $F$-algebra and $C \otimes_{\QQ} \RR$ is isomorphic to one of the following algebras
\begin{simplelist}
\item[(A)]
the product of $[F_0 : \QQ]$ copies of $M_n(\CC)$.
\item[(C)]
the product of $[F_0 : \QQ]$ copies of $M_{2n}(\RR)$.
\item[(D)]
the product of $[F_0 : \QQ]$ copies of $M_n(\HH)$, where $\HH$ denotes the Hamilton quaternions.
\end{simplelist}
\end{remark}

In this article we make the assumption that $\Gbf$ is connected. This is equivalent to the assumption that there is no simple factor of $B$ such that we are in case (D).

\begin{remark}\label{StructureGbf}
We recall some facts about the derived group $\Gbf^{\der}$ and the quotient $D := \Gbf/\Gbf^{\rm der}$ from~\cite{Ko_ShFin}~\S7. Let $\Gbf'$ be the kernel of the multiplicator homomorphism $c\colon \Gbf \to \GG_{m,\QQ}$. Then $\Gbf' = \Res_{F_0/\QQ}(\Gbf_0)$, where $\Res_{F_0/\QQ}(\ )$ denotes restriction of scalars from $F_0$ to $\QQ$ and where $\Gbf_0$ is the $F_0$-group of $B$-linear symplectic automorphisms of $(V,\lrangle)$. Let us again assume that $B$ is simple.

In case~(C), $\Gbf_0$ is an (automatically inner) form of a split symplectic group over $F_0$. In particular $\Gbf_0$ is simply connected. Hence $\Gbf' = \Gbf^{\der}$ is simply connected and $D \cong \GG_{m,\QQ}$.

In case~(A), $\Gbf_0$ is an inner form of the quasi-split unitary group over $F_0$ attached to an $(F/F_0)$-hermitian space $(V_0,(\ ,\ )_0)$. Thus $\Gbf_0^{\der}$ is simply connected and hence $\Gbf^{\rm der}$ is simply connected. Moreover, as passing from $\Gbf$ to an inner form does not change $D$, we see that $D$ is the subtorus of $\GG_{m,\QQ} \times \Res_{F/\QQ}\GG_{m,F}$ of elements $(t,x)$ such that $N_{F/F_0}(x) = t^n$, where $n := \dim_F(V_0)$.

If $n$ is even, say $n = 2k$, then we have an isomorphism
\begin{equation}\label{DescribeDEven}
D \cong \GG_{m,\QQ} \times \Res_{F_0/\QQ}D_0, \qquad (t,x) \sends (t,xt^{-k}),
\end{equation}
where $D_0 := \Ker(N_{F/F_0}\colon \Res_{F/F_0}\GG_{m,F} \to \GG_{m,F_0})$. If $n$ is odd, say $n = 2k+1$, then $(t,x) \sends t^{-k}x$ yields an isomorphism of $D$ with the $\QQ$-torus of elements $y$ in $\Res_{F/\QQ}\GG_{m,F}$ such that $N_{F/F_0}(y)$ belongs to $\GG_{m,\QQ}$. In particular we have an exact sequence
\begin{equation}\label{DescribeDOdd}
1 \to \Res_{F_0/\QQ}D_0 \to D \to \GG_{m,\QQ} \to 1.
\end{equation}
\end{remark}

\begin{remark}\label{CasesOB}
Our assumption that we exclude case (D) implies that $(O_B,\star) \otimes_{\ZZ_{(p)}} \ZZ_p$ is a product of $\ZZ_p$-algebras with involution $(O_i,\star)$, where $O_i$ is a matrix algebra over one the following rings $R$ (endowed with the involution induced by $\star$).
\begin{simplelist}
\item[(AL)] $R = O_K \times O_K$ where $K$ is some finite unramified extension of $\QQ_p$ and $(x,y)\star = (y,x)$ for $(x,y) \in O_K \times O_K$.
\item[(AU)] $R = O_K$ where $K$ is some finite unramified extension of $\QQ_p$ and $\star$ induces a $\QQ_p$-automorphism of order $2$ on $K$.
\item[(C)] $R = O_K$ where $K$ is some finite unramified extension of $\QQ_p$ and $\star$ is the identity on $O_K$.
\end{simplelist}
Note that the classifications given here and given in Remark~\ref{CasesB} match. Indeed, if $B$ is simple (and under our assumption that case (D) is excluded) we are in case (C) (as in Remark~\ref{CasesB}) if and only if $\star$ is trivial on $F$ which is again equivalent to $(O_B \otimes \ZZ_p,\star)$ being isomorphic to a product of $\ZZ_p$-algebras of type (C) in this classification.
\end{remark}

%
%
\subsubsection*{Reflex field and the determinant condition}
Let $E$ be the reflex field associated with $\Dscr$, i.e.~the field of definition of~$[\mu]$. It is a finite extension of~$\QQ$. We set $O_{E,(p)} := O_E \otimes_{\ZZ} \ZZ_{(p)}$. If we choose $\mu \in [\mu]$ and denote by $V_{\CC} = V_0 \oplus V_1$ the weight decomposition corresponding to $\mu$, then $E$ is the field of definition of the isomorphism class of the complex representation of $B$ on $V_1$.

We will now formulate Kottwitz's determinant condition following Rapoport and Zink (\cite{RZ}~3.23(a)). Choose a model $W$ of (the isomorphism class of) the $B$-represen-tation $V_1$ over $E$ and let $\Gamma$ be an $O_B$-invariant $O_{E,(p)}$-lattice in $W$. Let
\[
\OO_B := \VV(O\vdual_B \otimes_{\ZZ_{(p)}} O_{E,(p)})
\]
be the geometric vector bundle over $O_{E,(p)}$ corresponding to the free $\ZZ_{(p)}$-module $O_B$ and let $\delta_{\Dscr}\colon \OO_B \to \AA^1_{O_{E,(p)}}$ be the morphism of $O_{E,(p)}$-schemes given on $R$-valued points ($R$ any $O_{E,(p)}$-algebra) by
\begin{equation}\label{DefineDet}
\delta_{\Dscr}(R) \colon \OO_B(R) = O_B \otimes_{\ZZ_{(p)}} R \to R, \quad b \sends \det(b\mid \Gamma \otimes_{O_{E,(p)}} R).
\end{equation}
As $\delta_{\Dscr}$ is determined by its restriction to the generic fiber, it is independent of the choice of $\Gamma$.

Let $T$ be an $O_{E,(p)}$-scheme and let $L$ be a finite locally free $\Oscr_T$-module endowed with an $O_B$-action $O_B \to \End_{\Oscr_T}(L)$. As above we obtain a morphism of $T$-schemes
\[
\delta_L\colon \OO_B \times_{O_{E,(p)}} T \to \AA^1_T.
\]

\begin{definition}\label{DefDetCond}
We say that $L$ satisfies the \emph{determinant condition (with respect to $\Dscr$)} if the morphisms $\delta_L$ and $\delta_{\Dscr} \otimes \id_T$ of $T$-schemes are equal.
\end{definition}

If $L$ satisfies the determinant condition, then its rank is equal to $\dim_{\CC}(V_1) = \dim_{\QQ}(V)/2$.

\begin{remark}\label{CheckDetCond}
As morphisms of schemes can be glued for the fpqc topology, it suffices to check the determinant condition locally for the fpqc topology.

Moreover, as $O_B$ is unramified over $\ZZ_p$, the determinant condition on $L$ can be translated (possibly after some faithfully flat quasi-compact base change) into a condition on the rank of certain direct summands (see~\cite{RZ}~3.23(b) for details). This shows that the determinant condition on $L$ is open and closed in the unramified case. More precisely, there exists an open and closed subscheme $T_0$ of $T$ such that a morphism $f\colon T' \to T$ factors through $T_0$ if and only if $f^*L$ satisfies the determinant condition.
\end{remark}

\subsubsection*{The reductive group scheme $G$ and the set of simple reflections $J$}
As before let $\Dscr=(B,\star,V,\lrangle,\Gbf, O_B,\Lambda,h,[\mu])$ denote a Shimura PEL-datum (see Section \ref{sec2.1}). Let $G$ be the $\ZZ_p$-group scheme of $O_B$-linear symplectic similitudes of $\Lambda$. This is a reductive group scheme over $\ZZ_p$ whose generic fiber is $\Gbf_{\QQ_p}$. It is quasi-split (Appendix~\ref{Parabolic}). We denote the special fiber of $G$ by $\Gbar$.

We denote by $(W,I)$ the Weyl group of $G$ (or of $\Gbar$) together with its set of simple reflections (Appendix~\ref{Weylgroups}). The Frobenius endomorphism of $\Gbar$ induces an automorphism $\bar\varphi$ of $(W,I)$. Let $J \subset I$ be the type of the conjugacy class of one-parameter subgroups $[\mu^{-1}]$ (Appendix~\ref{one-ps}). It is defined over $\kappa$ (i.e.~$\bar\varphi^n(J)=J$ where $n=[\kappa:\mathbb{F}_p]$).

\begin{remark}\label{StructureG}
The reductive group scheme $G$ sits in an exact sequence
\begin{equation}\label{RelateGGprime}
1 \to G' \lto G \ltoover{c} \GG_{m,\ZZ_p} \to 1,
\end{equation}
where $c$ is the multiplicator homomorphism and $G'$ is the reductive $\ZZ_p$-group scheme of $O_B$-linear symplectic automorphisms of $\Lambda$. Every decomposition $(O_B,\star) \otimes_{\ZZ_{(p)}} \ZZ_p = (O_{B_1},\star) \times (O_{B_2},\star)$ yields a corresponding decomposition $G' = G'_1 \times G'_2$. Thus to describe the structure of $G'$ (and of $G$), we may assume that $(B,\star)$ is simple and thus that we are in one of the cases of Remark~\ref{CasesOB}. Then $G'$ has the following form (where $O_K$ is the ring of integers of a finite unramified extension of $\QQ_p$).
\begin{simplelist}
\item[(AL)]
$G' \cong \Res_{O_K/\ZZ_p}\GL_{n,O_K}$ ($n \geq 1$). In this case~\eqref{RelateGGprime} is split.
\item[(AU)]
$G' \cong \Res_{O_K/\ZZ_p}G_0$, where $G_0$ is the unitary group over $K$ attached to a perfect $(O_L/O_K)$-hermitian module $(\Vcal,(\ ,\ ))$, where $L$ is quadratic imaginary field extension of $K$.
\item[(C)]
$G' \cong \Res_{O_K/\ZZ_p}\Sp_{2g,O_K}$ ($g \geq 1$).
\end{simplelist}
\end{remark}

\begin{remark}\label{Gsplit}
Some of our results only hold if $G$ is split. From the explicit description of $G_{\QQ_p}$ (e.g.~in~\cite{We1}) it follows that $G$ (or, equivalently, $G_{\QQ_p}$) is split if and only if only the case (AL) and (C) with $K = \QQ_p$ occur (notations of Remark~\ref{CasesOB} or Remark~\ref{StructureG}).
\end{remark}

\subsection{Moduli spaces of abelian varieties with $\Dscr$-structure}\label{ModSpaceAbVar}

\subsubsection*{The integral model of the Shimura variety}
Let $\AA^p_f$ be the ring of finite adeles of $\QQ$ with trivial $p$-th component and let $C^p \subset \Gbf(\AA^p_f)$ be a compact open subgroup. We denote by $\Ascr = \Ascr_{\Dscr,C^p}$ the moduli space defined by Kottwitz~\cite{Ko_ShFin} \S5. More precisely, $\Ascr$ is the category fibered in groupoids over the category of $O_{E,(p)}$-schemes whose fiber category over an $O_{E,(p)}$-scheme $S$ is the category of tuples $\APlus$ satisfying the following properties.
\begin{bulletlist}
\item
$A$ is an abelian scheme over $S$.
\item
$\iota\colon O_B \to \End(A) \otimes_{\ZZ} \ZZ_{(p)}$ is a $\ZZ_{(p)}$-algebra homomorphism; this induces an $O_B$-action on the dual abelian scheme $A\vdual$ by $b \sends \iota(b\star)\vdual$.
\item
$\lambda$ is a $\ZZ_{(p)}^{\times}$-equivalence class of an $O_B$-linear polarization of order prime to $p$ (this means $\lambda$ is considered equivalent to $\alpha\lambda$, where $\alpha\colon S \to \ZZ_{(p)}^{\times}$ is a locally constant map).
\item
$\eta$ is a level structure of type $C^p$ on $A$ (in the sense of loc.~cit.).
\end{bulletlist}
Then $\iota$ induces by functoriality a homomorphism $\iota\colon O_B \to \End_{\Oscr_S}(\Lie A)$. We require that the $\Oscr_S$-module $\Lie(A)$ with this $O_B$-action satisfies the determinant condition (Definition~\ref{DefDetCond}).

A morphism of two tuples $\APlusOne$ and $\APlusTwo$ in the fiber category over $S$ is an $O_B$-linear quasi-isogeny $f\colon A_1 \to A_2$ of degree prime to $p$ such that $f^*\lambda_2 = \lambda_1$ and such that $f^*\eta_2 = \eta_1$.

Then $\Ascr$ is an algebraic Deligne-Mumford stack which is smooth over $O_{E,(p)}$. If $C^p$ is sufficiently small, $\Ascr$ is representable by a smooth quasi-projective scheme over $O_{E,(p)}$ (see loc.~cit.\ or \cite{Lan}~1.4.1.11 and~1.4.1.13).

Fix an embedding of the algebraic closure~$\QQbar$ of $\QQ$ in~$\CC$ into some fixed algebraic closure $\QQbar_p$ of~$\QQ_p$. Via this embedding we can consider~$[\mu]$ as a $G(\QQbar_p)$-conjugacy class of cocharacters. Denote by $v\vert p$ the place of~$E$ given by the chosen embedding $\QQbar \mono \QQbar_p$ and write $E_v$ for the $v$-adic completion of~$E$. Let $\kappa = \kappa(v)$ be its residue class field. Finally let $\kgbar$ be the residue field of the ring of integers of $\QQbar_p$. This is an algebraic closure of $\kappa$.

We denote by
\begin{equation}\label{a0}
\Ascr_0 := \Ascr_{\Dscr,C^p,0} := \Ascr_{\Dscr,C^p} \otimes_{O_{E,(p)}} \kappa
\end{equation}
the special fiber of $\Ascr_{\Dscr,C^p}$ at~$v$.

\subsubsection*{Modules and $p$-divisible groups with $\Dscr$-structure}
For a $T$-valued point $\APlus$ of $\Ascr_{\Dscr,C^p}$ the $O_B$-action and the polarization induce additional structures on the $p$-divisible group and the cohomology of $A$. Let us make this more precise.

We call two perfect pairings $\lrangle_1$ and $\lrangle_2$ on a finite locally free $\Oscr_T$-module $M$ \emph{similar} if there exists an open affine covering $T = \bigcup_j V_j$ and for all $j$ a unit $c_j \in \Gamma(V_j,\Oscr_{V_j}^{\times})$ such that $\langle m,m' \rangle_2 = c_j\langle m,m' \rangle_1$ for all $m,m' \in \Gamma(V_j,M)$.

\begin{definition}\label{DefDModule}
Let $T$ be a $\ZZ_p$-scheme. A \emph{$\Dscr$-module over $T$} is a locally free $\Oscr_T$-module $M$ of rank $\dim_{\QQ}(V)$ endowed with an $O_B$-action and the similitude class of a symplectic form $\lrangle$ such that $\langle bm,m' \rangle = \langle m,b^*m'\rangle$ for all $b \in O_B$ and local sections $m,m'$ of $M$.

An \emph{isomorphism $M_1 \iso M_2$ of $\Dscr$-modules} over $T$ is an $\Oscr_T \otimes O_B$-linear symplectic similitude.
\end{definition}

For a morphism $f\colon T' \to T$ of $\ZZ_p$-schemes and a $\Dscr$-module $M$ there is the obvious notion of a pull back $f^*M = M_{T'}$ of $M$ to a $\Dscr$-module on $T'$.

The $\ZZ_p$-module $\Lambda$ together with the induced $O_B$-action and the induced pairing is a $\Dscr$-module over $\ZZ_p$. The following lemma shows in particular that all $\Dscr$-modules are \'etale locally isomorphic to the $\Dscr$-module $\Lambda_T$.

\begin{lemma}\label{LocalIsom}
Let $T$ be a $\ZZ_p$-scheme and let $M_1$ and $M_2$ be two $\Dscr$-modules over $T$. Then locally for the \'etale topology on $T$, $M_1$ and $M_2$ are isomorphic as $\Dscr$-modules. Moreover the scheme of isomorphisms $\Isomline_{\Dscr}(M_1,M_2)$ of $\Dscr$-modules is a smooth affine $T$-scheme.
\end{lemma}

\begin{proof}
This is a special case of \cite{RZ}~Theorem~3.16.
\end{proof}

\begin{example}\label{DeRhamDModule}
Let $T$ be an $O_{E_v}$-scheme and let $\APlus \in \Ascr_{\Dscr,C^p}(T)$. We denote by $H_1^{\DR}(A/T)$ the first de Rham homology, i.e.~$H_1^{\DR}(A/T)$ is the $\Oscr_T$-linear dual of the first de Rham cohomology $R^1f_*(\Omega^{\bullet}_{A/S})$. This is a locally free $\Oscr_T$-module of rank $\dim_{\QQ}(V)$. As $A \sends H_1^{\DR}(A/T)$ is a covariant functor, we obtain an $O_B$-action on $H_1^{\DR}(A/T)$. The $\ZZ_{(p)}^{\times}$-equivalence class $\lambda$ yields a similitude class of perfect alternating forms $\lrangle$ on $H_1^{\DR}(A/T)$. Thus $H_1^{\DR}(A/T)$ is a $\Dscr$-module over $T$.
\end{example}

\begin{definition}\label{DefDPdiv}
Let $T$ be a $\ZZ_p$-scheme. A \emph{$p$-divisible group with $\Dscr$-structure} over $T$ is a triple $(X,\iota,\lambda)$ consisting of a $p$-divisible group $X$ over $T$ of height $\dim_{\QQ}(V)$, an $O_B$-action $\iota\colon O_B \to \End(X)$ and a $\ZZ_p^{\times}$-equivalence class of an $O_B$-linear isomorphism $\lambda\colon X \to X\vdual$ such that $\lambda\vdual = -\lambda$. Furthermore $\Lie(X)$ with the induced $O_B$-action is assumed to satisfy the determinant condition (Definition~\ref{DefDetCond}).
\end{definition}

Here we denote by $X\vdual$ the dual $p$-divisible group. If $X$ is endowed with an $O_B$-action $\iota\colon O_B \to \End(X)$, we endow $X\vdual$ with the $O_B$ action $b \sends \iota(b\star)\vdual$.

\begin{example}\label{DStructurePDiv}
Let $T$ be an $O_{E_v}$-scheme and let $\APlus \in \Ascr_{\Dscr,C^p}(T)$. Via functoriality $\iota$ and $\lambda$ induce  the structure of a $p$-divisible group with $\Dscr$-structure on $A[p^{\infty}]$.
\end{example}

Let $(X,\iota,\lambda)$ and $(X',\iota',\lambda')$ be $p$-divisible groups with $\Dscr$-structure. An \emph{isomorphism} (resp.~an \emph{isogeny}) \emph{of $p$-divisible groups with $\Dscr$-structure} is an $O_B$-linear isomorphism (resp.~an isogeny) $f\colon X \to X'$ such that $f^*\lambda' = \lambda$ (resp.~$f^*\lambda' = p^e\lambda$ for some integer $e \geq 0$), where we set $f^*\lambda' := f\vdual \circ \lambda' \circ f$.

\subsubsection*{Reformulation of the determinant condition}
\begin{proposition}\label{ReformDetCond}
Let $T$ be an $O_{E_v}$-scheme and let $M$ be a $\Dscr$-module over $T$. Let $H \subset M$ be a totally isotropic $O_B \otimes \Oscr_T$-submodule which is locally on $T$ a direct summand as an $\Oscr_T$-module. Then the following assertions are equivalent:
\begin{equivlist}
\item
$M/H$ satisfies the determinant condition.
\item
Locally for the \'etale topology on $T$ there exists an isomorphism $\alpha\colon M \to \Lambda_T$ of $\Dscr$-modules such that the stabilizer of $\alpha(H)$ in $G_T$ is a parabolic of type $J$.
\end{equivlist}
\end{proposition}

\begin{proof}
Both conditions are local for the \'etale topology. Therefore we may assume that $T$ is the spectrum of a strictly henselian ring $A$. By Lemma~\ref{LocalIsom} there exists an isomorphism $\alpha\colon M \iso \Lambda_T := \Lambda \otimes_{\ZZ_p} \Oscr_T$ of symplectic $\Oscr_T \otimes O_B$-modules. Thus we may assume that $M = \Lambda_T$. The stabilizer $P$ of $H$ is a smooth subgroup scheme of $G_T$. The properties that $P$ is a parabolic subgroup of $G$ and that $H$ satisfies the determinant condition can be checked on the geometric fibers (Remark~\ref{CheckDetCond}). Thus we may assume that $T$ is the spectrum of an algebraically closed field. Then going through the cases in Remark~\ref{CasesOB} the claim follows from the explicit description of the determinant condition in \cite{We1} 4.
\end{proof}

\subsection{The Siegel case as an example}\label{SiegelModSpace}
As an example, we will consider the case that $B = \QQ$ and $O_B = \ZZ_{(p)}$, the Siegel case. Then $\Gbf$ is the group of symplectic similitudes of a symplectic $\QQ$-vector space $V$ of dimension $2g$ and $G$ is the group of symplectic similitudes of a symplectic $\ZZ_p$-module $\Lambda$ of rank $2g$. This is a split reductive group. The conjugacy class $[\mu]$ is the unique one such that the decomposition $V_{\CC} = V_0 \oplus V_1$ is a decomposition into totally isotropic subspaces. It is defined already over $\QQ$, i.e.~$E = \QQ$.

The field $\kappa$ equals $\FF_p$ and $\Gbar$ is the group of symplectic similitudes of a symplectic $\FF_p$-vector space. Its Weyl group $W$ is described in Appendix~\ref{SymplecticExample}. With the notations introduced there the subset $J$ of the set of simple elements is just $\{s_1,\dots,s_{g-1}\}$.

Every element $b \in B = \QQ$ acts on the $g$-dimensional vector space $V_1$ as a scalar and therefore
\[
\delta_{\Dscr}\colon \OO_B = \AA^1_{\ZZ_{(p)}} \to \AA^1_{\ZZ_{(p)}}, \qquad b \sends b^g.
\]
Therefore the condition that for a $\ZZ_{(p)}$-scheme $S$ a finite locally free $\Oscr_S$-module $\Lscr$ satisfies the determinant condition just means that the rank of this module is equal to $g$. 

Fix an integer $N \geq 1$ prime to $p$ and set $C^p := \set{g \in \Gbf(\AA^p_f)}{g \equiv 1 \mod N}$. Then $\Ascr_{\Dscr,C^p}$ is the moduli space $\Ascr_{g,N}$ over $\ZZ_{(p)}$ of $g$-dimensional abelian varieties endowed with a principal polarization and a symplectic full level-$N$ structure. 

In this case, a $\Dscr$-module over a $\ZZ_p$-scheme $T$ is a locally free $\Oscr_T$-module $M$ of rank $2g$ endowed with a similitude class of symplectic forms. An $\Oscr_S$-submodule $H$ satisfies the conditions of Proposition~\ref{ReformDetCond} if and only if $H$ is Lagrangian (Appendix~\ref{SymplecticExample}). A $p$-divisible group with $\Dscr$-structure is a $p$-divisible group $X$ of height $2g$ together with a $\ZZ_p^{\times}$-equivalence class of isomorphisms $\lambda\colon X \iso X\vdual$ such that $\lambda\vdual = -\lambda$.


\section{Definition of the Ekedahl-Oort stratification}\label{DefEOStrata}

\subsection{$\Dscr$-zips attached to $S$-valued points of $\Ascr_0$}\label{DZipToAV}
Let $S$ be a $\kappa$-scheme and let $\APlus$ be an $S$-valued point of~$\Ascr_0$. We set $M := H_1^{\DR}(A/S)$ endowed with its structure of a $\Dscr$-module (Example~\ref{DeRhamDModule}). The de Rham cohomology $H^1_{\DR}(A/S)$ is endowed with two locally direct summands, namely $f_*\Omega^1_{A/S}$ given by the Hodge spectral sequence and $R^1f_*(\Hscr^0(\Omega^{\bullet}_{A/S}))$ given by the conjugate spectral sequence. We obtain locally direct summands of $M$ by defining
\[
C := (f_*\Omega^1_{A/S})^{\perp}, \qquad D := (R^1f_*(\Hscr^0(\Omega^{\bullet}_{A/S})))^{\perp}.
\]
The formation of the de Rham cohomology endowed with these two spectral sequences is functorial in $A$ in a contravariant way. The $\Oscr_S$-linear submodules $C$ and~$D$ of $M$ are then functorial, $O_B$-invariant and locally on $S$ direct summands of rank equal to $\dim(A/S)$.
We also recall that there is an isomorphism $f_*\Omega^1_{A/S} \iso \Lie(A)\vdual$ which is functorial in $A$. Thus we obtain a functorial identification
\begin{equation}\label{IdentifyC}
C = \Ker(H_1^{\DR}(A/S) \to \Lie(A)).
\end{equation}
Moreover, $D$ and~$C$ are totally isotropic with respect to $\lrangle$ (e.g., see \cite{DelPap}~1.6).

For every $\Oscr_S$-module $X$ we set $X^{(p)} := \Frob_S^*X$, where $\Frob_S\colon S \to S$ is the absolute Frobenius. Dualizing the construction in \cite{MW}~Section~7.5 shows that the Cartier isomorphism induces isomorphisms
\[
\varphi_0\colon (M/C)^{(p)} \iso D, \qquad \varphi_1\colon C^{(p)}
\iso M/D,
\]
The functoriality in $A$ of the Cartier isomorphism shows that $\varphi_0$ and $\varphi_1$ are $O_B$-linear and the tuple $(M, C, D, \varphi_0, \varphi_1)$ is a $\Dscr$-zip over $S$ in the following sense.

\begin{definition}
Let $S$ be a $\kappa$-scheme. A $\Dscr$-zip is a tuple $\Mline = (M, C, D, \varphi_0, \varphi_1)$ satisfying the following conditions.
\begin{definitionlist}
\item
$M$ is a $\Dscr$-module over $S$.
\item
$C$ and $D$ are $O_B$-invariant totally isotropic $\Oscr_S$-submodules of $M$ which are locally on $S$ direct summands such that $M/C$ satisfies the determinant condition. This implies that $C$ and $D$ both have rank $(1/2)\dim_{\CC}(V)$ and $C^{\perp} = C$ and $D^{\perp} = D$.
\item
$\varphi_0\colon (M/C)^{(p)} \iso D$ and $\varphi_1\colon C^{(p)} \iso M/D$ are $O_B \otimes \Oscr_S$-linear isomorphisms such that for some representative $\lrangle$ in the similitude class of symplectic forms on $M$ the following diagram commutes
\[\xymatrix{
(M/C)^{(p)} \ar[d] \ar[r]^{\varphi_0} & D \ar[d] \\
(C\vdual)^{(p)} & \ar[l]_{\varphi\vdual_1} (M/D)\vdual, 
}\]
where the vertical maps are the isomorphisms induced by the isomorphism $M \iso M\vdual$ given on local sections by $m \sends \langle - , m \rangle$.
\end{definitionlist}
\end{definition}
An \emph{isomorphism of $\Dscr$-zips} over $S$ is an isomorphism of $\Dscr$-modules preserving the submodules $C$ and $D$ and commuting with $\varphi_0$ and $\varphi_1$. We obtain the category of $\Dscr$-zips over $S$ (where every morphism is an isomorphism). Moreover we have the obvious notion of a pullback $f^*$ of a $\Dscr$-zip for a morphism $f\colon S' \to S$ of $\kappa$-schemes. Thus we obtain a category $\DZip$ fibered in groupoids over the category of $\kappa$-schemes. Clearly this is a stack for the fpqc-topology. Moreover it is easy to see that $\DZip$ is an algebraic stack in the sense of Artin.

\begin{example}\label{DieudonneZip}
Let $S = \Spec k$, where $k$ is a perfect field, and let $\APlus \in \Ascr_0(k)$. Let $\Mtilde := H^1_{\rm cris}(A/W(k))\vdual$ the covariant Dieudonn\'e module of the $p$-divisible group of $A$. It is endowed with a $\sigma$-linear endomorphism $F$ such that $V := F^{-1}p$ exists on $\Mtilde$. Thus $M = M(A) := \Mtilde/p\Mtilde$ is a $k$-vector space endowed with a $\Frob_k$-linear map $F\colon M \to M$ and a $\Frob^{-1}_k$-linear map $V\colon M \to M$ such that $\Ker(F) = {\rm Im}(V)$ and $\Ker(V) = {\rm Im}(F)$. The formation of $(M,F,V)$ is functorial in $A$ and we obtain an $O_B$-action and a similitude class of symplectic forms on $M$ induced by $\iota$ and $\lambda$, respectively. We set $C := \Ker(F)$ and $D := \Ker(V)$ and denote by $\varphi_0\colon (M/C)^{(p)} \iso D$ the isomorphism induced by $F$ and by $\varphi_1\colon C^{(p)} \to M/D$ the isomorphism induced by $V^{-1}$. Then $(M, C, D, \varphi_0, \varphi_1)$ is a $\Dscr$-zip. By \cite{Oda}~Corollary~5.11 there is a functorial isomorphism $M(A) \iso H_1^{\DR}(A/k)$ which is an isomorphism of $\Dscr$-zips.
\end{example}

Attaching to $\APlus$ the tuple $(M, \lrangle, C, D, \varphi_0, \varphi_1)$ defines a morphism of algebraic stacks
\begin{equation}\label{DefZeta}
\zeta\colon \Ascr_0 \to \DZip.
\end{equation}
We will see in Section~\ref{GroupDZip} that there are only finitely many isomorphism classes of $\Dscr$-zips over any fixed algebraically closed extension $k$ of $\kappa$ and that the set of isomorphism classes does not depend on $k$. 

\begin{definition}\label{FirstDefEO}
Let $[\Mline]$ be the isomorphism class of a $\Dscr$-zip over the algebraic closure $\bar\kappa$ of $\kappa$. We consider $[\Mline]$ as a (locally closed) point of the underlying topological space of $\DZip \otimes_{\kappa} \kgbar$. We call $\Ascr^{[\Mline]}_0 := \zeta^{-1}([\Mline]) \subseteq \Ascr_0 \otimes \kgbar$ the \emph{Ekedahl-Oort stratum} attached to $[\Mline]$.
\end{definition}

$\Ascr^{[\Mline]}_0$ is locally closed in the underlying topological space of $\Ascr_0 \otimes_{\kappa} \bar\kappa$ and we endow it with its reduced structure. A $\kgbar$-valued point $(A,\iota,\lambda,\eta)$ of $\Ascr_0$ is in $\Ascr^{[\Mline]}_0(\kgbar)$ if and only if the attached $\Dscr$-zip is isomorphic to $\Mline$.

In Section \ref{AllToGroup} we recall the classification of $\Dscr$-zips which yields a different index set for the Ekedahl-Oort strata.

\subsection{Ekedahl-Oort strata in the Siegel case I}\label{EOSiegel1}
We continue the example of the Siegel case of Section~\ref{SiegelModSpace}. Via the construction in Example~\ref{DieudonneZip} a $\Dscr$-zip over $\kgbar$ corresponds to a tuple $\Mline = (M,F,V,\lrangle)$, where $M$ is a $\kgbar$-vector space of dimension $2g$ together with a $\Frob_{\kgbar}$-linear map $F\colon M \to M$, a $\Frob_{\kgbar}^{-1}$-linear map $V\colon M \to M$, and a similitude class of symplectic pairings $\lrangle$ such that $\Ker(F) = \Im(V)$, $\Ker(V) = \Im(F)$ and such that
\[
\langle Fx, y \rangle = \langle x,Vy \rangle^p, \qquad x, y \in M.
\]
We call such a tuple $\Mline$ a \emph{symplectic Dieudonn\'e space}.

If $\Ascr_0 = \Ascr_{g,N}$ is the moduli space of principally polarized abelian varieties $(A,\lambda)$ of dimension $g$ with full level $N$ structure $\eta$, then $\Ascr_0^{[\Mline]}$ is the reduced locally closed substack of $\Ascr_0 = \Ascr_{g,N}$ such that $\Ascr_0^{[\Mline]}(\kgbar)$ consists of those $\kgbar$-valued points $(A,\lambda,\eta)$ of $\Ascr_{g,N}$ such that the covariant Dieudonn\'e module $(M',F',V')$ of $A[p]$ together with the similitude class of symplectic pairings $\lrangle'$ induced by $\lambda$ is isomorphic to $[\Mline]$. Thus we obtain in this case the Ekedahl-Oort stratification as defined in~\cite{Oo1}.

\subsection{Smooth coverings of $\Ascr_0$}\label{smoothcover}

We define two smooth coverings $\Ascr^{\#}_0$ and $\Ascrtilde_0$ of~$\Ascr_0$ as follows:

\smallskip

For every $\kappa$-scheme~$S$ the $S$-valued points of $\Ascr^{\#}_0$ are given by tuples $(A,\iota,\lambda,\eta,\alpha)$ where $(A,\iota,\lambda,\eta) \in \Ascr_0(S)$ and where $\alpha$ is an $O_B/pO_B$-linear symplectic similitude $H_1^{\DR}(A/S) \iso \Lambda \otimes_{\ZZ_p} \Oscr_S$.

Therefore $\Ascr^{\#}_0$ is a $\Gbar$-torsor over~$\Ascr_0$ for the \'etale topology. Here $\Gbar$ is the smooth group scheme obtained as the special fiber of $G$.

\smallskip

The $S$-valued points of $\Ascrtilde_0$ are given by tuples
$(A,\lambda,\iota,\eta, \alpha, C',D')$ with
$(A,\lambda,\iota,\eta, \alpha) \in \Ascr^{\#}_0$ and where
$C'$ and~$D'$ are $O_B/pO_B$-invariant totally isotropic
complements of $\alpha(C)$ and of~$\alpha(D)$ in $\Lambda \otimes_{\ZZ_p}
\Oscr_S$, respectively.


\section{Dieudonn\'e displays with additional structure}\label{sec3}

\subsection{Dieudonn\'e displays (after Zink and Lau)}\label{DieudonneDisplay}
In this section we always denote by $(R,\mfr)$ a complete local noetherian ring
with perfect residue field of characteristic $p$. We also assume that $p \geq 3$ or $pR = 0$.

In this section we will endow Dieudonn\'e displays in the sense of
Zink~\cite{Zink_Dieudonne} with additional structure. We start by recalling some facts on Dieudonn\'e displays.

\subsubsection*{The Zink ring}
We use the notations and terminology of Lau's paper~\cite{Lau_DieudonneDisplays}. In the sense of~\cite{Lau_DieudonneDisplays} Definition~1.2 the ring $R$ is admissible topological. In particular we have the Zink ring $\WW(R)$ which is endowed with a Frobenius $\sigma$ and a Verschiebung $\tau$ (denoted by $f$ and $v$ in loc.~cit.). By loc.~cit.\ it has the following properties.
\begin{assertionlist}
\item
The Zink ring $\WW(R)$ is a subring of the Witt ring $W(R)$ stable under $\sigma$ and $\tau$. Under the projection $W(R) \to W(k)$ it maps surjectively onto $W(k)$. In particular $\WW(k) = W(k)$.
\item
The kernel of $\WW(R) \to W(k)$ is $\WW(\mfr)$, the set of elements $x = (x_0,x_1,\dots) \in W(\mfr)$ such that the sequence $(x_i)_i$ converges to zero for the $\mfr$-adic topology.
\item
There exists a unique ring homomorphism $s\colon W(k) \to \WW(R)$ which is a section of the projection $\WW(R) \to W(k)$. Thus $\WW(R) = s(W(k)) \oplus \WW(\mfr)$.
\item
$\WW(R) = \limproj_n \WW(R/\mfr^n)$ is a local $p$-adically complete ring with residue field $k$.
\item
If $R$ is Artinian, then $\WW(\mfr)$ consists only of nilpotent elements. In particular $\WW(R)_{\rm red} = W(k)$.
\end{assertionlist}
The first two properties characterize $\WW(R)$. Let $\II_R$ be the kernel of $w_0\colon \WW(R) \to R$. Then $\II_R$ is the image of $\tau$. We denote by $\sigma_1\colon \II_R \to \WW(R)$ the inverse of $\tau$.

For every $\WW(R)$-module $M$ we set $M^{\sigma} =  \WW(R) \otimes_{\sigma,\WW(R)} M$. For two $\WW(R)$-modules $M$ and $N$ we identify $\sigma$-linear maps $M \to N$ with linear maps $M^{\sigma} \to N$. If $M$ is of the form $M = \Lambda \otimes_{\ZZ_p} \WW(R)$ for some $\ZZ_p$-module $\Lambda$ we have a canonical isomorphism $M^\sigma \cong M$ which we use to identify these two $\WW(R)$-modules.

Let $X$ be any $\WW(R)$-scheme. Then the ring endomorphism $\sigma$ of $\WW(R)$ induces a map $\sigma\colon X(\WW(R)) \to X(\WW(R))$. We will use this notation in particular for $X = G
\otimes_{\ZZ_p} \WW(R)$ and for the scheme of parabolics of $G \otimes_{\ZZ_p} \WW(R)$.

We will use the following result.

\begin{lemma}\label{ZinkHensel}
Let $R$ be a ring as above and let $k$ be its (perfect) residue field. For any smooth $W(k)$-scheme $X$ the canonical map
\[
c_X\colon X(\WW(R)) \to X(R) \times_{X(k)} X(W(k))
\]
is surjective.
\end{lemma}

\begin{proof}
Let $\WW_n(R)=R\times_{R/\mathfrak{m}^n}\WW(R/\mathfrak{m}^n)$. Then $\WW(R) = \limproj_n \WW_n(R)$. The transition maps $\WW_n{R}\rightarrow \WW_{n-1}(R)$ have nilpotent kernel. The formal smoothness of $X$ implies that $X(\WW(R))=\limproj X(\WW_n(R))$ maps surjectively to $X(\WW_1(R))=X(R\times_k W(k))=X(R)\times_{X(k)}X(W(k))$.
\end{proof}

\subsubsection*{Definition of Dieudonn\'e displays}
Recall (\cite{Lau_DieudonneDisplays}~Definition~2.6 and Section~2.8) that a \emph{Dieudonn\'e display} over $R$ is a tuple $(\Pcal,\Qcal,F,F_1)$ where $\Pcal$ is a finitely generated projective $\WW(R)$-module, $\Qcal$ is a submodule such that there exists a decomposition $\Pcal = \Scal \oplus \Tcal$ of $\WW(R)$-modules with $\Qcal = \Scal \oplus \II_R\Tcal$ (called a \emph{normal decomposition}), and where $F\colon \Pcal \to \Pcal$ and $F_1\colon \Qcal \to \Pcal$ are $\sigma$-linear maps of $\WW(R)$-modules such that
\begin{equation}\label{Eq:CompF1}
F_1(ax) = \sigma_1(a)F(x),\qquad \text{for all $a \in \II_R$, $x \in \Pcal$}
\end{equation}
and such that $F_1(\Qcal)$ generates the $\WW(R)$-module $\Pcal$.

These axioms imply
\begin{equation}\label{Eq:CompF2}
F(x) = pF_1(x),\qquad \text{for all $x \in \Qcal$.}
\end{equation}

Furthermore, as $\WW(R)$ is local, $\Pcal$ is in fact a free $\WW(R)$-module.

\begin{remark}\label{PsiDisplay}
Let $\Scal$ and $\Tcal$ be finitely generated projective $\WW(R)$-modules. Set $\Pcal := \Scal \oplus \Tcal$ and $\Qcal := \Scal \oplus \II_R\Tcal$. Using~\eqref{Eq:CompF1} and~\eqref{Eq:CompF2} it follows that there is a bijection between the set of pairs $(F,F_1)$ such that $(\Pcal,\Qcal,F,F_1)$ is a Dieudonn\'e display and the set of isomorphisms $\Psi\colon \Pcal^{\sigma} \iso \Pcal$ of $\WW(R)$-modules given by
\[
(F\colon \Pcal^{\sigma} \to \Pcal,F_1\colon \Qcal^{\sigma} \to \Pcal) \sends \Psi := F_1|_{\Scal^{\sigma}} \oplus F|_{\Tcal^{\sigma}}\colon (\Scal \oplus \Tcal)^{\sigma} \to \Pcal.
\]
We call $(\Scal,\Tcal,\Psi)$ a \emph{split Dieudonn\'e display} over $R$.
\end{remark}

\subsubsection*{Duality for Dieudonn\'e displays}
We recall Lau's duality for Dieudonn\'e displays (\cite{Lau_Duality}). A bilinear form between Dieudonn\'e displays $\Pscr = (\Pcal,\Qcal,F,F_1)$ and $\Pscr' = (\Pcal',\Qcal',F',F'_1)$ is a $\WW(R)$-bilinear map $\beta\colon \Pcal \times \Pcal' \to \WW(R)$ such that $\beta(\Qcal,\Qcal') \subseteq \II_R$ and such that
\begin{equation}\label{BilinarDisplay}
\beta(F_1x,F'_1x') = \sigma_1\beta(x,x')
\end{equation}
for $x \in \Qcal$ and $x' \in \Qcal'$. There exists a Dieudonn\'e display $\Pscr\vdual = (\Pcal\vdual,\Qcal\vdual,F\vdual,F\vdual_1)$, the \emph{dual Dieudonn\'e display} of $\Pscr$, such that for every Dieudonn\'e display $\Pscr'$ there exists a functorial isomorphism
\[
\Hom(\Pscr',\Pscr\vdual) = \{\text{Bilinear forms between $\Pscr$ and $\Pscr'$}\}.
\]
The tautological pairing $\gamma$ between $\Pscr$ and $\Pscr\vdual$ is perfect and $\Pcal\vdual$ is the $\WW(R)$-linear dual of $\Pcal$. Given a normal decomposition $\Pcal = \Scal \oplus \Tcal$ for $\Pscr$ there is a unique normal decomposition $\Pcal\vdual = \Scal\vdual \oplus \Tcal\vdual$ for $\Pscr\vdual$ such that $\gamma(\Scal,\Scal\vdual) = \gamma(\Tcal,\Tcal\vdual) = 0$. The attached linear operators $\Psi$ and $\Psi\vdual$ (Remark~\ref{PsiDisplay}) satisfy
\begin{equation}\label{PsiBilinear}
\gamma(\Psi(x),\Psi\vdual(x')) = \gamma(x,x')
\end{equation}
for $x \in \Pcal^{\sigma}$ and $x' \in \Pcal^{\prime\sigma}$.

\subsubsection*{Dieudonn\'e displays and $p$-divisible groups}
By~\cite{Zink_Dieudonne} and~\cite{Lau_Duality} (see also~\cite{Lau_DieudonneDisplays} Theorem~3.24 and Corollary~3.26) we have:

\begin{theorem}\label{DieudonneEquiv}
There exists an equivalence $\Phi_R$ between the category of $p$-divisible groups over $R$ and the category of Dieudonn\'e displays over $R$. This equivalence is compatible with base change by continuous ring homomorphisms $R \to R'$ and with duality.
\end{theorem}

Moreover, let $A$ be an abelian scheme over $R$ and let $\Pscr = (\Pcal,\Qcal,F,F_1)$ be the Dieudonn\'e display of the $p$-divisible group of $A$. Then by the construction in \cite{Lau_DieudonneDisplays}~\S3 (relating the Dieudonn\'e display and the (covariant) Dieudonn\'e crystal) and by~\cite{BBM} Chap.~5 we have a functorial isomorphism of exact sequences
\begin{equation}\label{HodgeFil}\xymatrix{
0 \ar[r] & (R^1f_*\Oscr_A)\vdual \ar[r] \ar[d] & H_1^{\DR}(A/R) \ar[r] \ar[d] & \Lie(A) \ar[r] \ar[d] & 0 \\
0 \ar[r] & \Qcal/\II_R\Pcal \ar[r] & \Pcal/\II_R\Pcal \ar[r] & \Pcal/\Qcal \ar[r] & 0.
}\end{equation}
Note that $(R^1f_*\Oscr_A)\vdual \cong (\Lie A\vdual)\vdual$.


\subsection{$\Dscr$-Dieudonn\'e displays}
We continue to assume that $R$ is a complete local noetherian ring with perfect residue field of characteristic $p$ such that $p \geq 3$ or $pR = 0$. We will now endow Dieudonn\'e displays with a ``$\Dscr$-structure''. For our purpose it suffices to do this only for Dieudonn\'e displays with a fixed normal decomposition and whose underlying $\WW(R)$-module is $\Lambda \otimes_{\ZZ_p} \WW(R)$. The general definition is similar, but rather tedious and thus omitted.

\begin{definition}\label{DefDDisplay}
Assume that $R$ is in addition an $O_{E_v}$-algebra. We set $M := \Lambda \otimes_{\ZZ_p} \WW(R)$. Then $M$ carries a symplectic form $\lrangle$ and an $O_B$-action.
A \emph{split $\Dscr$-Dieudonn\'e display over $R$} consists of a tuple $(\Scal,\Tcal,\Psi)$ where $\Scal$ and $\Tcal$ are totally isotropic $O_B \otimes \WW(R)$-submodules of $M$ such that $\Scal \oplus \Tcal = M$ and where $\Psi\colon M^{\sigma} \to M$ is a symplectic $\WW(R) \otimes O_B$-linear similitude. We assume that $M/\Scal$ satisfies the determinant condition.
\end{definition}

Assume in addition that $R$ is a $\kappa$-algebra. We will now attach to an $R$-valued point $(A,\lambda,\iota,\eta,\alpha)$ of $\Ascr^\#_0$ (Section~\ref{smoothcover}) a split $\Dscr$-Dieudonn\'e display. Let $(\Pcal,\Qcal,F,F_1)$ be the Dieudonn\'e display associated with the $p$-divisible group of $A$ (Theorem~\ref{DieudonneEquiv}). Then $\lambda$ induces by functoriality a similitude class of perfect alternating forms $\lrangle$ on the free $\WW(R)$-module $\Pcal$. The $O_B$-action by $\iota$ induces an $O_B$-action on $\Pcal$ and we obtain a $\Dscr$-module $\Pcal$ over $\WW(R)$. Moreover $O_B$ acts by homomorphisms of Dieudonn\'e displays and $\lrangle$ is a bilinear form of Dieudonn\'e displays. Therefore $\Qcal$ is a $\WW(R) \otimes O_B$-submodule of $\Pcal$, and $F$ and $F_1$ are $O_B$-linear.

As $(\WW(R),\II_R)$ is henselian, Lemma~\ref{LocalIsom} implies that there is a $\WW(R) \otimes O_B$-linear symplectic similitude
\[
\agtilde\colon \Pcal \iso \Lambda \otimes \WW(R)
\]
whose reduction modulo $\II(R)$ is equal to $\alpha$. We use $\agtilde$ to identify $\Pcal$ and $\Lambda \otimes \WW(R)$.

By~\eqref{HodgeFil} we have a split exact sequence of free $R$-modules
\[
0 \to C \to \Pcal/\II_R \Pcal \to \Pcal/\Qcal \to 0
\]
where $C$ is as in Section \ref{DZipToAV}. In the category of $\WW(R) \otimes O_B$-modules we choose a totally isotropic direct summand $\Scal$ of $M$ such that its reduction modulo $\II_R$ is equal to $C$ and a totally isotropic complement $\Tcal$ of $\Scal$ in $M$. Let $\Psi\colon \Pcal^{\sigma} \iso \Pcal$ be the isomorphism of $\Dscr$-modules attached to $(F,F_1)$ (Remark~\ref{PsiDisplay}). Then $(\Scal,\Tcal,\Psi)$ is a split $\Dscr$-Dieudonn\'e display.


\section{Group-theoretic reformulation}\label{AllToGroup}

\subsection{Group-theoretic description of $\Dscr$-displays}
We are now going to give a group-theoretic reformulation of the split Dieudonn\'e displays with additional structure of Definition~\ref{DefDDisplay}. We continue to assume that $R$ is a complete local noetherian $O_{E_v}$-algebra with perfect residue field $k$ of characteristic $p$ and that $p \geq 3$ or $pR = 0$. Recall that $G$ denotes the quasi-split reductive $\ZZ_p$-group scheme of $O_B$-linear symplectic similitudes of $\Lambda$ and that $(W,I)$ denotes the Weyl group of $G$ together with its set of simple reflections (Appendix~\ref{Weylgroups}).

\begin{definition}\label{DefYtilde}
For a subset $J\subseteq I$ defined over $\kappa$ let $\Ytilde_J$ be the $O_{E_v}$-scheme such that for each $O_{E_v}$-scheme $X$ the $X$-valued points of $\Ytilde_J$ are triples $(P,L,g)$, where $P$ is a parabolic subgroup of $G_X$ of type $J$, where $L$ is a Levi subgroup of $P$, and where $g \in G(X)$. By Appendix~\ref{Parabolic}, $\Ytilde_J$ is a quasi-projective smooth scheme.
\end{definition}

\begin{construction}\label{DieudonneToGroup}
Let $(\Scal,\Tcal,\Psi)$ be a split $\Dscr$-Dieudonn\'e display over $R$ (Definition~\ref{DefDDisplay}). We construct an associated element $(\Ptilde,\Ltilde,\gtilde)$ in $\Ytilde_J(\WW(R))$ as follows. We define $\Ptilde$ as the
stabilizer of the flag $0 \subset \Scal \subset M := \Lambda \otimes \WW(R)$ in $G_{\WW(R)}$. Then $\Ptilde$ is a parabolic of the associated type $J$ by Proposition~\ref{ReformDetCond}. Furthermore, $\Ltilde$ is defined as the stabilizer of the decomposition $M = \Scal \oplus \Tcal$ in $G_{\WW(R)}$. Then $\Ltilde$ is a Levi subgroup of $\Ptilde$. Finally let $\gtilde$ be the composition
\[
M \liso M^{\sigma} \ltoover{\Psi} M.
\]
\end{construction}

\begin{lemma}\label{DescribeSplit}
The construction above defines a bijection between the set of all split $\Dscr$-Dieudonn\'e displays over $R$ and the set $\Ytilde_J(\WW(R))$.
\end{lemma}

\begin{proof}
We construct an inverse. Let $(\Ptilde,\Ltilde,\gtilde)$ be in $\Ytilde_J(\WW(R))$. We let $\Scal$ be the unique $O_B$-invariant totally isotropic direct summand of $M$ whose stabilizer is equal to $\Ptilde$. Then $M/\Scal$ satisfies the determinant condition. Let $\Tcal$ be the unique $O_B$-invariant complement of $\Scal$ in $M$ such that the stabilizer of the decomposition $\Scal \oplus \Tcal$ is equal to $\Ltilde$. Further let $\Psi$ be the $\sigma$-linear map attached to $\gtilde$. Then $(\Scal,\Tcal,\Psi)$ is a split $\Dscr$-Dieudonn\'e display. This construction defines an inverse.
\end{proof}

\begin{remark}\label{RecoverF}
Let $(\Scal,\Tcal,\Psi)$ be a split $\Dscr$-Dieudonn\'e display and $(\Ptilde,\Ltilde,\gtilde) \in \Ytilde_J(\WW(R))$ the corresponding triple. Let $(\Pcal,\Qcal,F,F_1)$ be the Dieudonn\'e display associated with $(\Scal,\Tcal,\Psi)$. By definition we have $\Pcal= \Scal \oplus \Tcal = \Lambda \otimes \WW(R)$ and $\Qcal = \Scal \oplus I_R\Tcal$. By~\eqref{Eq:CompF2} and Remark~\ref{PsiDisplay} the $\sigma$-linear map $F\colon \Pcal \to \Pcal$ is given by $\gtilde \circ (p\id_{\Scal} \oplus \id_{\Tcal}) \circ (\id_\Lambda \otimes \sigma)$.
\end{remark}

\subsection{Group-theoretic description of $\Dscr$-zips}\label{GroupDZip}
We recall (a special case of) the definition of the schemes studied in~\cite{MW} and, more generally, in~\cite{PWZ}. Recall that $\overline{G}$ denotes the fiber of $G$ over $\mathbb{F}_p$. Denote by $F$ the Frobenius on $\overline{G}$ given by the $p$-th power. Then $F$ induces an automorphism $\bar\varphi\colon (W,I) \iso (W,I)$ of Coxeter systems. We denote by $\Gamma_\kappa = \Gal(\kgbar/\kappa)$ the Galois group of $\kappa$. It acts on $(W,I)$ via the group generated by $\bar\varphi^{n}$ where $n = [\kappa:\mathbb{F}_p]$.

Let $w_0$ be the element of maximal length in $W$, set $K := {}^{w_0}{\bar\varphi(J)}$ and let $x \in \doubleexp{K}{W}{\bar\varphi(J)}$ be the element of minimal length in $W_Kw_0W_{\bar\varphi(J)}$. Thus
\begin{equation}\label{Shapex}
x = w_{0,K}w_0 = w_0w_{0,\bar\varphi(J)},
\end{equation}
where $w_{0,K}$ and $w_{0,\bar\varphi(J)}$ denote the longest elements in $W_K$ and in $W_{\bar\varphi(J)}$, respectively. Then $x$ is the unique element of maximal length in $\doubleexp{K}{W}{\bar\varphi(J)}$.

\subsubsection*{The $\overline{G}$-scheme $X_J$}
We denote by $X_J = X_{J,F,x}$ the functor on $\kappa$-schemes which is the Zariski-sheafification of the functor $X^{\flat}_J$ which associates with a $\kappa$-scheme $S$ the set of triples $(P,Q,U_QgU_{F(P)})$ where $P \subset G_S$ is a parabolic of type $J$, where $Q \subset G_S$ is a parabolic of type $K$ and where $g \in G(S)$ is an element such that $\relpos(Q,{}^gF(P)) = x$ (here $U_R$ denotes the unipotent radical of a parabolic subgroup $R$ and $\relpos$ the relative position, see~\ref{Weylgroups}). By \cite{MW}~Corollary~4.3 this functor is representable by a $\kappa$-scheme. Note that strictly speaking, in \cite{MW} the functor, where $F$ is the $\kappa$-Frobenius, is considered. However, the representability of the functor considered here follows from the same proof. For any affine scheme $S$ we have $X_J(S) = X^{\flat}_J(S)$ (use $H^1(S,U) = 0$ if $U$ is the unipotent radical of a parabolic group of a reductive group scheme and if $S$ is affine, e.g., see \cite{SGA3}~Exp.~XXVI, Corollaire~2.2).

\begin{remark}\label{PQOpposition}
As $K := {}^{x}\bar\varphi(J)$ and $x$ is the longest element in $\doubleexp{K}{W}{\bar\varphi(J)}$, we have $\relpos(Q,{}^gF(P)) = x$ if and only if $Q \cap {}^gF(P)$ is a common Levi subgroup of $Q$ and ${}^gF(P)$, i.e.~$Q$ and ${}^gF(P)$ are opposite parabolic subgroups (\cite{Lu_ParII} \S8).
\end{remark}

\begin{remark}\label{ZJasTorsor}
The forgetful morphism which is defined on points by $(P,Q,[g]) \mapsto (P,Q)$ is a morphism
$X_J \to \Par_J \times \Par_K$ which is a torsor under the base change to $\Par_J \times \Par_K$ of the Levi quotient of the universal parabolic subgroup over $\Par_J$ (\cite{MW}~Lemma~4.2). This reductive group scheme over $\Par_J \times \Par_K$ is of relative dimension $\dim(P/U_P)$ for any parabolic of $G$ of type $J$ . In particular $X_J$ is a geometrically connected smooth $\kappa$-scheme whose dimension equals $\dim(\Par_J) + \dim(\Par_K) + \dim(P/U_P) = \dim(G)$.
\end{remark}

The $\kappa$-group scheme $G_\kappa = G \otimes_{\ZZ_p} \kappa$ acts on $X_J$ by
\[
h\cdot(P,Q,[g]) = ({}^hP,{}^hQ,[hgF(h)^{-1}]).
\]
The arguments in \cite{MW}~Section~6 show:

\begin{proposition}\label{GroupDZips}
The algebraic stack $\DZip$ and the algebraic quotient stack $[G_\kappa \backslash X_J]$ are isomorphic. In particular, there is a bijection between isomorphism classes of $\Dscr$-zips over $\kgbar$ and $G(\kgbar)$-orbits on $X_J(\kgbar)$.
\end{proposition}

Theorem~12.17 of \cite{PWZ} yields a bijection of the set of $G(\kgbar)$-orbits of $X_J(\kgbar)$ with the set $\leftexp{J}{W}$. In particular, there are only finitely many orbits and hence only finitely many isomorphism classes of $\Dscr$-zips over $\kgbar$. The bijection is given by
\begin{equation}\label{DescribeXJOrbit}
\leftexp{J}{W} \ni w \sends \text{$G(\kgbar)$-orbit of $(P_J, \leftexp{\dot w}{P_K}, [\dot w\dot x])$}.
\end{equation}
Here we have chosen a Borel $B$ of $\overline{G}$, a maximal torus $T$ of $B$ (yielding an identification of $W$ with the Weyl group of $T$), and for all $w \in W$ a representative $\dot w$ in $\Norm_G(T)(\kgbar)$. Moreover $P_J$ (resp.~$P_K$) denotes the unique parabolic subgroup $G_{\kgbar}$ of type $J$ (resp.~$K$) containing $B_{\kgbar}$.

\begin{definition}\label{SpecializationOrder}
We endow $\leftexp{J}{W}$ with a relation $\preceq$, where we define $w' \preceq w$ if there exists $y \in W_J$ with $yw'x\bar\varphi(y^{-1})x^{-1}\leq w$ with respect to the Bruhat order. Here $x = w_0w_{0,\bar\varphi(J)}$ is the element defined above.
\end{definition}

\begin{lemma}\label{SpecIsOrder}
The relation $\preceq$ is a partial order on $\leftexp{J}{W}$ compatible with the Galois action by $\Gamma_\kappa$ (i.e., $w' \preceq w \iff \gamma(w') \preceq \gamma(w)$ for all $\gamma \in \Gamma_\kappa$).
\end{lemma}

\begin{proof}
It is shown in \cite{PWZ}~Corollary~6.3 that $\preceq$ is a partial order. Let $n=[\kappa:\mathbb{F}_p]$. As the group of automorphisms of $(W,I)$ induced by the Galois action is generated by $\bar\varphi^n$, it remains to show that $w' \preceq w$ implies $\bar\varphi^n(w') \preceq \bar\varphi^n(w)$. But this follows from $\bar\varphi^n(J) = J$ which implies $\bar\varphi^n(W_J) = W_J$ and $\bar\varphi^n(x)=x$.
\end{proof}

\begin{proposition}\label{DescribeZJ}
The underlying topological space of $[G_\kappa \backslash X_J] \otimes_{\kappa} \kgbar$ is homeo\-morphic to the topological space $(\leftexp{J}{W})_{\rm top}$ attached to the partially ordered set \mbox{$(\leftexp{J}{W}, \preceq)$}.
\end{proposition}

The description of the underlying topological space of a quotient stack for an action with finitely many orbits and the notion of the topological space attached to a partially ordered set are recalled in Section~\ref{QuotStacks}.

\begin{proof}
For $w,w' \in \leftexp{J}{W}$ let $O$ and $O'$ be the $G(\kgbar)$-orbits of $X_J(\kgbar)$ corresponding to $w$ and $w'$, respectively. As explained in Section~\ref{QuotStacks}, it suffices to show that the closure of $O'$ contains $O$ if and only if $w \preceq w'$. This is shown in \cite{PWZ}~Theorem~12.15.
\end{proof}

In particular we obtain for every algebraically closed extension $k$ of $\kappa$ a bijection
\begin{equation}\label{IsoClassesDZip}
\{\text{isomorphism classes of $\Dscr$-zips over $k$}\} \bijective \leftexp{J}{W}.
\end{equation}

\begin{remark}\label{DescribeZJRational}
Proposition~\ref{DescribeZJ} shows that the underlying topological space of $[G_\kappa  \backslash X_J]$ is homeomorphic to the set of $\Gamma_{\kappa}$-orbits on $\leftexp{J}{W}$ endowed with the quotient topology (Section~\ref{QuotStacks}). In other words, it is homeomorphic to the topological space attached to the partially ordered set $\Gamma_{\kappa}\backslash \leftexp{J}{W}$ of $\Gamma_{\kappa}$-orbits on $\leftexp{J}{W}$, where we set $\Gamma_{\kappa}w' \preceq \Gamma_{\kappa}w$ if and only if for one (or, equivalently by Lemma~\ref{SpecIsOrder}, for all) $v' \in \Gamma_{\kappa}w'$ there exists $v \in \Gamma_{\kappa}w$ such that $v' \preceq v$.
\end{remark}

\begin{remark}\label{RelPosPQ}
Let $k$ be an algebraically closed extension of $\kappa$, $w \in \leftexp{J}{W}$, and let $O_w \subset X_J(k)$ be the corresponding $G(k)$-orbit. Let $\bar w$ be the image of $w$ under the canonical map $\leftexp{J}{W} = W_J \backslash W \epi W_J \backslash W/W_K = \doubleexp JWK$. Then for all $(P,Q,[g]) \in O_w$ one has $\relpos(P,Q) = \bar w$ (\cite{MW}~Section~4.6).
\end{remark}

\subsubsection*{The smooth cover $\Xtilde_J$ of $X_J$}
We denote by $\Xtilde_J$ the scheme which represents the functor on $\kappa$-schemes which associates with a $\kappa$-scheme $S$ the set of triples $(P,Q,g)$ where $P \subset G_S$ is a parabolic of type $J$, where $Q \subset G_S$ is a parabolic of type $K$ and where $g \in G(S)$ is an element such that $\relpos(Q,{}^gF(P)) = x$.

The group $G_\kappa$ acts on $\Xtilde_J$ by the rule
\[
h\cdot(P,Q,g) = ({}^hP,{}^hQ,hgF(h)^{-1}).
\]
The canonical surjective morphism $\Xtilde_J \to X_J$ is then $G_\kappa$-equivariant.


\subsection{Group-theoretic definition of $\zeta$ and the Ekedahl-Oort stratification}

We relate the reductions of Shimura varieties and the schemes $\Ytilde_J$ (Definition~\ref{DefYtilde}), $\Xtilde_J$, and $X_J$. We first construct a morphism
\begin{equation}\label{YtildeZtilde}
\psi\colon \Ytilde_J \otimes \kappa \to \Xtilde_J
\end{equation}
as follows. For every $S$-valued point $(P,L,g)$ of $\Ytilde_J$ we define $Q$ as the unique opposite parabolic such that ${}^{g^{-1}}Q \cap F(P) = F(L)$ (\cite{SGA3}~Exp.~XXVI,~4.3.). Then $Q$ is of type $K = \leftexp{w_0}{\bar\varphi(J)}$, and $(P,Q,g) \in \Xtilde_J(S)$ by Remark~\ref{PQOpposition}.

Now define a morphism
\[
\tilde\theta\colon \Ascrtilde_0 \to \Ytilde_J \otimes \kappa
\]
as follows: Let $R$ be a $\kappa$-algebra and let $(A,\iota,\lambda,\eta,\alpha,C',D')$ be an $R$-valued point of $\Ascrtilde_0$. Let $(M,C,D,\varphi_0,\varphi_1)$ be the $\Dscr$-zip attached to $(A,\iota,\lambda)$ (Section~\ref{DZipToAV}). Then $\alpha$ is by definition an isomorphism of $\Dscr$-modules $M \iso \Lambda \otimes R$. Let $P$ be the stabilizer of $\alpha(C)$ in $G_R$, let $L$ be the stabilizer of the decomposition $\alpha(C) \oplus C' = \Lambda_R$, and let $g$ be the composition
\[
\Lambda_R \iso \Lambda_R^{(p)} = \alpha(C)^{(p)} \oplus
C^{\prime(p)} \vartoover{30}{\varphi'_1 \oplus \varphi'_0} \alpha(D)
\oplus D' = \Lambda_R
\]
where $\varphi'_1$ and $\varphi'_0$ are the morphisms induced by $\varphi_1$ and $\varphi_0$ via the isomorphism $\alpha$. We obtain an $R$-valued point $(P,L,g)$ of $\Ytilde$. This construction is functorial in $R$ and defines the morphism $\tilde\theta$.

The composition $\psi \circ \tilde\theta$ is a morphism $\zgtilde\colon \Ascrtilde_0 \to \Xtilde_J$. It induces morphisms
\begin{equation}\label{DefineZeta}
\begin{aligned}
\zeta^{\#}&\colon \Ascr^{\#}_0 \to X_J,\\
\zeta&\colon \Ascr_0 \to [G_{\kappa} \backslash X_J].
\end{aligned}
\end{equation}
Via the isomorphism $\DZip \cong [G \backslash X_J]$ (Proposition~\ref{GroupDZips}) the morphism $\zeta$ defined here is identified with the morphism $\zeta$ in~\eqref{DefZeta}.

By Proposition~\ref{DescribeZJ} we know that the points of the underlying topological space of $[G_{\kappa}\backslash X_J] \otimes \kgbar$ are in natural bijection with $\leftexp{J}{W}$. For $w \in \leftexp{J}{W}$ set
\[
\Ascr^w_0 := \zeta^{-1}(\{w\}).
\]
This is a locally closed subset of the underlying topological space of $\Ascr_0 \otimes \kgbar$.

\begin{definition}\label{DefEO}
We call the corresponding reduced locally closed substack of \mbox{$\Ascr_0 \otimes \kgbar$} the \emph{Ekedahl-Oort stratum} corresponding to $w \in \leftexp{J}{W}$. It is again denoted by $\Ascr^w_0$.
\end{definition}

\begin{remark}\label{SchemeEO}
The residue gerbe (\cite{LM}~(11.1)) $\Gscr(w)$ of the point $w$ in the underlying topological space of $[G_{\kappa}\backslash X_J] \otimes \kgbar$ is a locally closed smooth algebraic substack. Then the locally closed substack $\Ascr_0(w) := \zeta^{-1}(\Gscr(w))$ of $\Ascr_0$ has the same underlying topological space as $\Ascr^w_0$. It can be shown that $\Ascr_0(w) = \Ascr^w_0$ (at least for $p > 2$). As we do not need this in the sequel, we decided to define the structure of a substack just by taking the reduced substack.

Let $\Spec \kgbar \to [G_{\kappa}\backslash X_J]$ be a representative of the point $w$ and let $\Mline$ be the corresponding $\Dscr$-zip over $\kgbar$. Then $\Ascr_0(w) = \Ascr_0^{[\Mline]}$ with the notation in Definition~\ref{FirstDefEO}. Let $\APlus$ be the universal family restricted to $\Ascr_0(w)$ and let $\underline{\Mscr}$ be the attached $\Dscr$-zip on $\Ascr_0(w)$. Then by the definition of the residue gerbe, a morphism of $\kgbar$-schemes $f\colon T \to \Ascr_0 \otimes \kgbar$ factors through $\Ascr_0(w)$ if and only if $f^*\underline{\Mscr}$ is fppf-locally on $T$ isomorphic to the pullback $\Mline_T$ of $\Mline$ to $T$.
\end{remark}

\begin{remark}\label{EOFieldOfDef}
Instead of working with points in $[G_{\kappa} \backslash X_J] \otimes \kgbar$ to define Ekedahl-Oort strata one could also work with points in $[G_{\kappa} \backslash X_J]$ (considered as locally closed substacks of $[G_{\kappa} \backslash X_J]$ defined over $\kappa$) to define $\kappa$-rational Ekedahl-Oort strata indexed by $\Gamma_{\kappa}$-orbits $\Gamma_{\kappa}w$ on $\leftexp{J}{W}$ (Remark~\ref{DescribeZJRational}). We denote them by $\Ascr^{\Gamma_{\kappa}w}_0$. Then
\[
\Ascr^{\Gamma_{\kappa}w}_0 \otimes \kgbar = \bigcup_{w' \in \Gamma_{\kappa}w} \Ascr^{w'}_0.
\]
Conversely, $\Ascr^w_0$ for $w \in \leftexp{J}{W}$ is already defined over the finite extension $\kappa(w)$ of $\kappa$ in $\kgbar$ such that $\Gal(\kgbar/\kappa(w)) = \set{\gamma \in \Gamma_\kappa}{\gamma(w) = w}$.
\end{remark}

\begin{example}\label{MuOrdSS}
The partially ordered set $(\leftexp{J}{W},\preceq)$ has a unique minimal element, namely $1$, and a unique maximal element, $\leftexp{\mu}{w} := w_{0,J}w_0$, where $w_0$ is the longest element in $W$ and $w_{0,J}$ is the longest element in $W_J$. We call the corresponding Ekedahl-Oort strata $\Ascr_0^1$ and $\Ascr_0^{\leftexp{\mu}{w}}$ the \emph{superspecial} and the \emph{$\mu$-ordinary Ekedahl-Oort stratum}, respectively. As the $\Gamma_\kappa$-action preserves the order, it fixes $1$ and $\leftexp{\mu}{w}$. Therefore both strata are defined over $\kappa$.
\end{example}


\subsection{Ekedahl-Oort strata in the Siegel case II}\label{EOSiegel2}
We continue the example of the Siegel case from Sections~\ref{SiegelModSpace} and~\ref{EOSiegel1}. As explained in~\ref{SymplecticExample}, $\leftexp{J}{W}$ can be identified with $\{0,1\}^g$. Moreover, the Frobenius acts trivially on $W$, i.e. $\bar\varphi = \id$, and $K = J$ because $w_0 = -1$ is central. As the Galois group of $\kappa = \FF_p$ acts trivially on $(W,I)$, all Ekedahl-Oort strata are already defined over $\FF_p$ (Remark~\ref{EOFieldOfDef}).

By Section~\ref{EOSiegel1} and~\eqref{IsoClassesDZip} we have for every algebraically closed extension $k$ of $\FF_p$ a bijection
\begin{equation}\label{ClassSympDieud}
\{\text{isomorphism classes of symplectic Dieudonn\'e spaces over $k$}\} \bijective \{0,1\}^g.
\end{equation}
For $w = (\epsilon_i)_i \in \{0,1\}^g$ let $\Mline = (M,F,V,\lrangle)$ be a symplectic Dieudonn\'e space over $k$ in the isomorphism class corresponding to $w$. Combining~\eqref{InterpretePos} and Remark~\ref{RelPosPQ} we obtain
\begin{equation}\label{RelPosFMVM}
a(\Mline) := \dim M/(FM + VM) = g - \#\set{i}{\epsilon_i = 1}.
\end{equation}
Therefore the longest element $\leftexp{\mu}{w} = (1,1,\dots,1)$ in $\leftexp{J}{W}$ corresponds to the unique isomorphism class of $\Mline$ with $a(\Mline) = 0$, which means that the associated truncated $p$-divisible group is ordinary. Thus in the Siegel case the $\mu$-ordinary Ekedahl-Oort stratum is equal to the stratum of ordinary principally polarized abelian varieties in $\Ascr_0$.

The shortest element $1 = (0,\dots,0)$ in $\leftexp{J}{W}$ corresponds to the isomorphism class of $\Mline$ with $a(\Mline) = g$, which means that the associated truncated $p$-divisible group is superspecial. Thus in the Siegel case the superspecial Ekedahl-Oort stratum is equal to the stratum of principally polarized abelian varieties in $\Ascr_0$ such that the underlying abelian variety of a geometric point is isomorphic to a product of supersingular elliptic curves.


\section{Flatness of $\zeta$}\label{secflat}

\begin{theorem}\label{zetaflat}
The morphism $\zeta$ is flat.
\end{theorem}

By definition the following diagram is cartesian
\[\xymatrix{
\Ascr^{\#}_0 \ar[r] \ar[d]_{\zeta^{\#}} & \Ascr_0 \ar[d]_{\zeta} \\
X_J \ar[r] & [G \backslash X_J],
}\]
where the horizontal morphisms are $G$-torsors and in particular smooth and surjective. We will show that $\zeta^{\#}$ is universally open. As $X_J$ and $\Ascr^{\#}_0$ are both regular, it then follows from \cite{EGA}~IV,~(15.4.2), that $\zeta^{\#}$ is flat. By faithfully flat descent this shows that $\zeta$ is flat.

To show that $\zeta^{\#}$ is universally open, we use the following criterion.

\begin{proposition}\label{opencrit}
Let $Y$ be a locally noetherian scheme, let $X$ be a scheme, and let $f\colon X \to Y$ be a morphism of finite type. Assume that for every commutative diagram
\begin{equation}\label{valdiagram}
\begin{gathered}\xymatrix{
\Spec(k) \ar[r]^h \ar[d] & X \ar[d]^f \\
\Spec(R) \ar[r]^g & Y
}\end{gathered}
\end{equation}
where $R$ is a complete discrete valuation ring with algebraically closed residue field $k$, there exists a surjective morphism $\Spec(\Rtilde) \to \Spec(R)$, where $\Rtilde$ is a local integral domain and a morphism $\gtilde\colon \Spec(\Rtilde) \to X$ such that the following diagram commutes
\[\xymatrix{
\Spec(\Rtilde \otimes_R k) \ar[r] \ar[d] & \Spec(k) \ar[r]^h & X \ar[d]^f\\
\Spec(\Rtilde) \ar[rru]^{\gtilde} \ar[r] & \Spec(R) \ar[r]^g & Y
}\]
Then $f$ is universally open.
\end{proposition}

It is easy to see that such a property already characterizes universally open morphisms. We will not need this remark in the sequel.

\begin{proof}
By \cite{EGA}~IV~(8.10.2) it suffices to show that the base change $f_{\AA^n_Y} \colon \AA^n_X \to \AA^n_Y$ is open for all $n \geq 0$. The affine space $\AA^n_Y$ is again locally noetherian and the hypothesis for $f$ implies the same property for $f_{\AA^n_Y}$. Thus we may replace $Y$ by $\AA^n_Y$ and $X$ by $\AA^n_X$. Therefore it suffices to show that $f$ is open. The question is local on $Y$ and we may thus assume that $Y$ is noetherian.

Let $U \subset X$ be an open subset. By Chevalley's theorem $f(U)$ is constructible. Therefore it suffices to show that $f(U)$ is stable under generization (e.g., \cite{GW}~Lemma~10.17). Let $x_0 \in U$ and $y_0 = f(x_0) \in f(U)$ and let $y_1 \in Y$ be a generization with $y_1 \not= y_0$. By Lemma~\ref{existval} below, there exists a diagram like in~\eqref{valdiagram} such that $g(s) = y_0$, $g(\eta) = y_1$ (where $s$ (resp.~$\eta$) is the special (resp.~generic) point of $\Spec R$) and such that the image of $h$ is $x_0$. We apply the hypothesis and find a morphism $\gtilde\colon \Spec(\Rtilde) \to X$ such that $f \circ \gtilde$ is the composition
\[
\Spec(\Rtilde) \to \Spec(R) \to Y.
\]
The image $x_1$ of the generic point of $\Spec(\Rtilde)$ under $\gtilde$ is a
generization of $x_0$ and hence lies in $U$ as $U$ is open. Therefore $y_1
= f(x_1) \in f(U)$. 
\end{proof}

\begin{lemma}\label{existval}
Let $Y$ be a locally noetherian scheme and let $f\colon X \to Y$ be a morphism of schemes. Let $x_0 \in X$, $y_0 := f(x_0)$ and let $y_1 \not= y_0$ be a generization of $y_0$. Then there exists a commutative diagram
\[\xymatrix{
\Spec(\kappa) \ar[r]^{h} \ar[d]_i & X \ar[d]^f \\
\Spec(R) \ar[r]^g & Y \\
}\]
where $R$ is a complete discrete valuation ring with algebraically closed residue field $\kappa$, and where $i\colon \Spec(\kappa) \mono \Spec(R)$ is the inclusion such that the image of the generic (resp.\ special) point of $\Spec(R)$ under $g$ is $y_1$ (resp.\ $y_0$) and such that the image of $h$ is $x_0$.
\end{lemma}

\begin{proof}
There exists a morphism $g'\colon \Spec(R') \to Y$ where $R'$ is a discrete valuation ring such that $g'(s') = y_0$ and $g'(\eta') = y_1$ where $s'$ (resp.\ $\eta'$) is the closed
(resp.\ generic) point of $\Spec(R')$.

Let $\mfr'$ be the maximal ideal of $R'$ and let $\kappa$ be an algebraically closed field extension of $\kappa(y_0)$ such that there exist $\kappa(y_0)$-embeddings $\kappa(x_0) \mono \kappa$ and $\kappa(s') \mono \kappa$ and let $R' \to R$ be a flat local homomorphism of $R'$ into a complete discrete valuation ring $R$ with residue field $\kappa$ such that $\mfr'R$ is the maximal ideal of $R$ (this exists by \cite{EGA} ${\bf 0}_I$ 6.8.3). We set $g$ as the composition
\[
\Spec(R) \lto \Spec(R') \ltoover{g'} Y
\]
and $h$ as the composition
\[
\Spec(\kappa) \lto \Spec(\kappa(x_0)) \lto X.\qedhere
\]
\end{proof}

\begin{proof}[Proof of the universal openness of $\zeta^{\#}$]
Every equi-characteristic complete discrete valuation ring $R$ with residue field $k$ is isomorphic to the ring of formal power series $k\dlbrack \eps \drbrack$. Therefore by Proposition~\ref{opencrit} it suffices to show the following lemma.
\end{proof}

\begin{lemma}
Let $k$ be an algebraically closed field of characteristic $p$, let $R
= k\dlbrack \eps \drbrack$ be the ring of formal power series in one
variable $\eps$ and set $R_1 = k\dlbrack \eps^{1/p} \drbrack$. We denote by
\[
\pi\colon \Spec(R_1) \to \Spec(R)
\]
the natural morphism.

Let $x = (A,\iota,\lambda,\eta,\alpha)$ be a $k$-valued point of $\Ascr^\#_0$. Let $(P,Q,[g]) \in X_J(k)$ be the image of $x$ under $\zeta^{\#}$. Denote by $(P_\eps,Q_\eps,[g_\eps]) \in X_J(R)$ any deformation of $(P,Q,[g])$ to $R$ (which exists because $X_J$ is smooth). Then there exists a deformation $x_1 = (A_1,\iota_1,\lambda_1,\eta_1,\alpha_1) \in \Ascr^\#_0(R_1)$ of $x$ such that $\zeta^{\#}(x_1) = \pi^*(P_\eps,Q_\eps,[g_\eps])$.
\end{lemma}
Note that the smoothness of $X_J$ implies that a deformation $(P_\eps,Q_\eps,[g_\eps])$ as in the lemma always exists.
\begin{proof}
Let $\Pscr = (\Pcal,\Qcal,F,F_1)$ be the Dieudonn\'e display of the $p$-divisible group of the abelian variety $A$. The free $W(k)$-module $\Pcal$ is equipped with a perfect alternating form via $\lambda$ and an $O_B$-action. Moreover, we can fix an $O_B$-linear symplectic similitude $\agtilde$ of $\Pcal$ with $\Lambda \otimes W(k)$ which lifts the isomorphism $\alpha$ (Lemma~\ref{LocalIsom}). Set $\Pcal_\eps = \Lambda \otimes \WW(R)$ and $\Pcal_{\eps,1} = \Lambda \otimes \WW(R_1)$.

We choose a normal decomposition $\Pcal = \Scal \oplus \Tcal$ as in Definition~\ref{DefDDisplay} and obtain a split $\Dscr$-display $(\Scal,\Tcal,\Psi)$. Let $(\Ptilde,\Ltilde,\gtilde)$ be the associated element in $\Ytilde_J(W(k))$ (Construction~\ref{DieudonneToGroup}).

Then the reduction of $\gtilde$ modulo $p$ is an element $g \in [g]$ by the definition of $\zeta^{\#}$. Let $g_\eps \in [g_\eps]$ be an element such that the reduction of $g_\eps$ modulo $\eps$ is equal to $g$. Set $P_{\eps,1} := P_\eps \otimes_R R_1$, $Q_{\eps,1} := P_\eps \otimes_R R_1$ and let $g_{\eps,1}$ be the
element $g_\eps$ considered as an $R_1$-valued point of $G$.

Let $L_{\eps,1}$ be the Levi subgroup of $P_{\eps,1}$ such that
\[
F(L_{\eps,1}) = {}^{(g^{-1}_{\eps,1})}Q_{\eps,1} \cap F(P_{\eps,1}).
\]
We now apply Lemma~\ref{ZinkHensel} to the smooth scheme $\Ytilde_J$ and the ring $R_1$. We see that there exists an element
\[
(\Ptilde_{\eps,1},\Ltilde_{\eps,1},\gtilde_{\eps,1}) \in
 \Ytilde_J(\WW(R_1))
\]
whose reduction to $R_1$ equals $(P_{\eps,1},L_{\eps,1},g_{\eps,1})$ and whose reduction to $W(k)$ equals $(\Ptilde,\Ltilde,\gtilde)$.

Let $(\Scal_{\eps,1},\Tcal_{\eps,1},\Psi_{\eps,1})$ be the split $\Dscr$-display associated with $(\Ptilde_{\eps,1},\Ltilde_{\eps,1},\gtilde_{\eps,1})$ (Construction~\ref{DieudonneToGroup}) and $(X_1,\iota_1,\lambda_1)$ be the $p$-divisible group with $\Dscr$-structure corresponding to $(\Scal_{\eps,1},\Tcal_{\eps,1},\Psi_{\eps,1})$ (Theorem~\ref{DieudonneEquiv}). By Serre-Tate theory we obtain the desired point $x_1 \in \Ascr^\#_0(R_1)$.
\end{proof}


\section{Closures of Ekedahl-Oort strata}\label{sec6}

From now on we fix a maximal torus $T$ of $G$ and a Borel group $B$ containing $T$, both defined over $\ZZ_p$. This is possible by~\ref{Parabolic}. Let $(X^*(T),\Phi,X_*(T),\Phi\vdual,\Delta)$ be the corresponding based root datum. Its Weyl system is $(W,I)$. The based root datum and its Weyl system are endowed with an action by $\Gamma = \pi_1(\Spec \ZZ_p, \Spec \kgbar) = \Gal(\kgbar/\FF_p)$. We consider the conjugacy class $[\mu]$ defined by the Shimura datum (Section~\ref{sec2.1}) as a $W$-orbit in $X_*(T)$. The representative in $X_*(T)$ that is dominant with respect to $B$ is denoted by $\mu$. It is defined over $W(\kappa)$ or, if we consider only the reduction modulo $p$ of $(G,B,T)$, over $\kappa$.

\begin{theorem}\label{corclosure}
The closure of an Ekedahl-Oort stratum $\Ascr_0^w$ ($w \in \leftexp{J}{W}$) is a union of Ekedahl-Oort strata. More precisely,
\[
\overline{\Ascr_0^w} = \bigcup_{w' \preceq w} \Ascr_0^{w'},
\]
where $\preceq$ is the partial order defined in Definition~\ref{SpecializationOrder}.
\end{theorem}

For $p > 2$ the first assertion has been shown in \cite{We2}~(6.8). For the Siegel case, Oort has given a different parametrization of the Ekedahl-Oort strata (in terms of elementary sequences, see Section~\ref{EOSiegel2}) and also proved the first assertion (\cite{Oo1}). The description which strata appear in the closure of a given stratum is new even in the Siegel case.

\begin{proof}
By Proposition~\ref{DescribeZJ} the closure of $\{w\}$ in the underlying topological space of $[\Gbar \backslash X_J] \otimes \kgbar$ is $\set{w' \in \leftexp{J}{W}}{w' \preceq w}$. As $\zeta$ is universally open, we have $\overline{\Ascr_0^w} = \zeta^{-1}(\overline{\{w\}})$.
\end{proof}

\subsubsection*{Comparison to loop groups}
For a field $k$ and a linear algebraic group $H$ we denote by $LH$ the loop group. It is the group ind-scheme over $k$ representing the sheaf for the fpqc-topology
\[
(\text{$k$-algebras}) \rightarrow (\text{groups}), \qquad R \sends LH(R) := H(R(\!(z)\!)),
\]
see \cite{Faltings}~Definition~1. We denote by $LG$ the loop group of $G_{\mathbb{F}_p}$.

Recall that $\mu \in X_*(T)$ is the dominant coweight determined by the PEL Shimura datum $\Dscr$, defined over $\kappa$. Let $k$ be an algebraically closed extension of $\kappa$, set $K := G(k\dlbrack z\drbrack)$, and let $K_1$ be the kernel of the reduction modulo $z$ map $K\rightarrow G(k)$. We denote by $\mu(z)\in LG(k)$ the image under $\mu\colon \mathbb{G}_m\rightarrow T$ of $z\in \mathbb{G}_m(k(\! ( z)\! ) )$. For all $w\in W$ we have a chosen representative $\dot w\in G(k)\hookrightarrow G(k(\! ( z)\! ) )$. As in \cite{trunc1} we consider for each $w \in \leftexp{J}{W}$ the corresponding truncation stratum in the loop group of $LG$ which is defined as the locally closed reduced subscheme of $LG \otimes_{\FF_p} \kappa(w)$ (where $\kappa(w)$ is the field of definition of $w$, see Remark~\ref{EOFieldOfDef}) with 
\[
\mathcal{S}_{w,\mu}(k) = \set{k^{-1}\dot w\dot x_{\mu}\mu(z)k_1\sigma(k)}{k\in K, k_1\in K_1},
\]
where $x_\mu := w_0w_{0,\bar\varphi(J)}=x$ is the element already considered in~\eqref{Shapex}. Remark~\ref{DZipToPDiv} below explains why this can be considered as an ``Ekedahl-Oort stratum for loop groups''.

\begin{corollary}\label{corloopeo}
A stratum $\Ascr^{w'}_0$ is contained in the closure of $\Ascr_0^w$ if and only if $\mathcal{S}_{w',\mu}$ is contained in the closure of $\mathcal{S}_{w,\mu}$ in $LG$.
\end{corollary}

\begin{proof}
This follows from the explicit descriptions of the closure relations in Theorem~\ref{corclosure} and in \cite{trunc1}, Corollary 4.7.
\end{proof}


\section{The Newton stratification}\label{sec7}
\subsection{Affine Weyl groups}\label{AffWeyl}
To be able to compare the moduli space $\Ascr_0$ with the loop group $LG$ we make the following definitions. We consider two cases in which $O_L$ denotes either $W(k)$ or $k\dlbrack z\drbrack$, where $k$ is an algebraically closed field of characteristic $p$. Let $L$ be the field of fractions of $O_L$. By $\epsilon$ we denote the uniformizer $p$ or $z$.

By $\Ical$ we denote the inverse of $B$ under the projection $G(O_L) \rightarrow G(k)$. 

Let $\widetilde{W}=N_T(L)/T(O_L)\cong W\ltimes X_*(T)$ denote the extended affine Weyl group of $G$. It has a decomposition $\widetilde W\cong \Omega\ltimes W_{\rm aff}$. Here $\Omega$ is the subset of elements of $\widetilde{W}$ which stabilize the chosen Iwahori subgroup $\Ical$. The second factor $W_{\rm aff}$ is the affine Weyl group of $G$, an infinite Coxeter group generated by the simple reflections together with the simple affine reflection defined by $\Ical$. If $M$ is a Levi subgroup of $G$ containing $T$, let $\Ical_M=\Ical\cap M(O_L)$. Let $W_M$ and $\widetilde{W}_M$ be the Weyl group and extended affine Weyl group for $M$. They are canonically subgroups of $W$ and $\widetilde{W}$. Let $\Omega_M$ be the subgroup of elements of $\widetilde{W}_M$ which stabilize $\Ical_M$.

For $\lambda\in X_*(T)$ we have the representative $\lambda(\epsilon)\in T(L)$. Together with the chosen representatives of $W$ we have representatives $\dot x$ in $G(L)$ for all elements $x \in \widetilde{W}$. By the Bruhat-Tits decomposition (see \cite{Tits}) each element of $G(L)$ is contained in a double coset $\Ical\dot x\Ical$ for a unique $x\in \widetilde{W}$.

Let $\mathfrak{a}=X_*(T)\otimes_{\mathbb{Z}}\mathbb{R}$ and let $\mathfrak{a}_M$ be the subset of elements which are centralized by $M$. Finally we denote by $\overline{\mathbf a}$ the unique alcove in the antidominant Weyl chamber of the standard apartment of the Bruhat-Tits building of $G$ whose closure contains the origin. Then $\Ical$ is the Iwahori subgroup fixing $\overline{\mathbf a}$. We call $\overline{\mathbf a}$ the opposite base alcove.

\subsection{Crystals and isocrystals with $\Dscr$-structure attached to points in $\Ascr_0$}\label{CGBG}

\subsubsection*{The map $\Upsilon$}
Let $k$ be an algebraically closed field extension of $\kappa$. By (covariant) Dieudonn\'e theory, isomorphism classes of $p$-divisible groups $(X,\iota,\lambda)$ with $\Dscr$-structure over $k$ are in bijection with isomorphism classes of their Dieudonn\'e modules $(\Pcal,F)$. Here $\Pcal$ is a $\Dscr$-module over $W(k)$ and $F\colon \Pcal^{\sigma} \to \Pcal$ is an injective $W(k) \otimes O_B$-linear map which preserves the symplectic form up to the scalar $p \in W(k)$. Thus after tensoring with $L := \Frac W(k)$, $F$ is an isomorphism of $\Dscr$-modules over $L$.

By Lemma~\ref{LocalIsom} we can choose an isomorphism $\alpha\colon \Pcal \iso \Lambda \otimes_{\ZZ_p} W(k)$ of $\Dscr$-modules. By transport of structure with $\alpha$ we obtain $F = (\id_{\Lambda} \otimes \sigma)b$ for some $b \in G(L)$. Let $K=G(W(k))$. Then for a different trivialization $\alpha'$ we have $\alpha' = \sigma(g) \circ \alpha$ for some $g \in K$ and $b$ is replaced by $g^{-1}b\sigma(g)$.

If we denote by $C(G) = C_k(G)$ the set of $K$-$\sigma$-conjugacy classes
\[
\dlbrack b \drbrack := \set{g^{-1}b\sigma(g)}{g\in K}
\]
of elements $b\in G(L)$, we therefore obtain an injective map
\begin{equation}\label{PDivGroup}
\left\{\begin{matrix}
\text{isomorphism classes of $p$-divisible groups}\\
\text{with $\Dscr$-structure over $k$}
\end{matrix}\right\} \mono C_k(G).
\end{equation}
The image consists of the subset $C(G,\mu) = C_k(G,\mu)$ of $K$-$\sigma$-conjugacy classes of elements $g \in K\mu(p)K$.

Let $\APlus$ be a $k$-valued point of $\Ascr_0$. Then the $p$-divisible group $A[p^{\infty}]$ carries a $\Dscr$-structure and thus we can attach an element $\dlbrack b \drbrack \in C(G)$ which we denote by $\Upsilon\APlus$. We obtain a map
\begin{equation}\label{DefUps}
\Upsilon\colon \Ascr_0(k) \to C(G,\mu).
\end{equation}

\begin{remark}\label{DZipToPDiv}
For a $p$-divisible group with $\Dscr$-structure the (covariant) Dieudonn\'e module of its $p$-torsion carries the structure of a $\Dscr$-zip (cf.~Example~\ref{DieudonneZip}). This induces a map on isomorphism classes, i.e., by~\eqref{PDivGroup} and~\eqref{IsoClassesDZip} a map
\begin{equation}\label{Truncation}
{\rm tc}\colon C(G,\mu) \lto \leftexp{J}{W},
\end{equation}
which we call \emph{truncation map (at level $1$)}.

This map is surjective: It follows from Remark~\ref{RecoverF} and~\eqref{DescribeXJOrbit} that for $w \in \leftexp{J}{W}$ a pre-image is given by $\dlbrack \dot w \dot x \mu(p) \drbrack$ where $x$ is as in (\ref{Shapex}).

As explained in the introduction of~\cite{trunc1}, ${\rm tc}$ maps $\dlbrack g \drbrack$ and $\dlbrack g' \drbrack$ of $C(G,\mu)$ to the same element if and only if
\begin{equation}\label{SameTruncation}
\exists\ k_1,k'_1 \in K_1 := \Ker(K \to G(k)): \quad \dlbrack k_1g'k'_1 \drbrack = \dlbrack g \drbrack.
\end{equation}
As $K_1 \subset K$, this condition is also equivalent to the existence of $k_1 \in K_1$ such that $\dlbrack g'k_1 \drbrack = \dlbrack g \drbrack$.
\end{remark}

\subsubsection*{The sets $B(G)$ and $B(G,\mu)$}
Again let $k$ be an algebraically closed extension of $\kappa$. We now come to the group theoretical classification of isocrystals over $k$. The map~\eqref{PDivGroup} induces a bijective map
\begin{equation}\label{IsocGroup}
\left\{\begin{matrix}
\text{isogeny classes of $p$-divisible groups}\\
\text{with $\Dscr$-structure over $k$}
\end{matrix}\right\} \liso B(G,\mu) \subseteq B(G),
\end{equation}
where $B(G)$ is the set of $G(\Frac W(k))$-$\sigma$-conjugacy classes in $G(\Frac W(k))$, and where we define $B(G,\mu)$ to be the image of $C(G,\mu)$ in $B(G)$ under the canonical surjection $C(G) \to B(G)$. 

We recall the classification of $B(G)$ and $B(G,\mu)$ (following~\cite{RapoportGuide}~\S4 supplemented by some more recent results). In order to compare with results from the case of loop groups, we work in a more general situation: We consider the two cases $L = \Frac W(k)$ and $L=k(\!(z)\!)$. Let $O_L$ be its ring of integers and $\epsilon$ the uniformizer $p$ or $z$. Let $B(G)$ be the set of $G(L)$-$\sigma$-conjugacy classes $\set{g^{-1}b_0\sigma(g)}{g\in G(L)}$ of elements $b_0 \in G(L)$.

The elements of $B(G)$ are classified (in much greater generality) by Kottwitz \cite{Kottwitz1}. In \cite{Kottwitz1} only the case $L=\Frac W(k)$ is considered. However, the same arguments show the analogous classification for the other case. Let $(X_*(T) \otimes \QQ)_{\rm dom}$ be the cone of rational cocharacters of $T$ that are dominant with respect our chosen Borel subgroup. As the Borel subgroup is chosen over $\ZZ_p$, $(X_*(T) \otimes \QQ)_{\rm dom}$ is preserved by the action of the Galois group $\Gamma = \Gal(\kgbar/\FF_p)$. There is a map
\[
\nu\colon B(G) \to (X_*(T) \otimes \QQ)_{\rm dom}^{\Gamma}.
\]
We call $\nu(b)$ the Newton polygon of $b$. For $G=\GL_n$ it coincides with the usual Newton polygon of $F$-isocrystals.

Furthermore Kottwitz \cite{Kottwitz1} defines a map $\kappa_G\colon G(L)\rightarrow \pi_1(G)_{\Gamma}$ (to be precise, the reformulation using $\pi_1(G)_{\Gamma}$ is due to Rapoport and Richartz, \cite{RapoportRichartz}, Section 1). Here $\pi_1(G)$ is the quotient of $X_*(T)$ by the coroot lattice and $\pi_1(G)_{\Gamma}$ denotes the coinvariants under the Galois group $\Gamma$. In our situation this map has the following simplified description: Let $\tilde b \in G(L)$ be a representative of $b$ and let $\lambda \in X_*(T)$ be the unique dominant cocharacter with $\tilde b \in G(O_L)\lambda(\epsilon)G(O_L)$. Then $\kappa(b)$ is the image $\lambda^{\flat}$ of $\lambda$ under the projection $X_*(T)\rightarrow \pi_1(G)_{\Gamma}$. Indeed, the maps $\kappa_G$ considered by Kottwitz are invariant under $\sigma$-conjugation, they are group homomorphisms, and are natural transformations in $G$. For $G=\mathbb{G}_m$ we have $\kappa_G(b)=v_p(b)$. In our case the properties of $\kappa_G$ mentioned above show that $\kappa_G$ is trivial on $G(O_L)$ as each such element is $\sigma$-conjugate to $1$ (by a version of Lang's theorem), and that the torus element $\lambda(\epsilon)$ is mapped to its image in $\pi_1(G)_{\Gamma}$. Together we obtain the explicit description of $\kappa_G$ given above.

On $X_*(T) \otimes \QQ$ one considers the standard order which is given by $\nu' \leq \nu$ if and only if $\nu-\nu'$ is a non-negative linear combination of positive coroots. It induces a partial ordering on $B(G)$ by setting $b\leq b'$ if and only if $\nu(b)\leq\nu(b')$ and $\kappa(b)=\kappa(b')$. 

An element $b \in B(G)$ is determined by $\nu(b)$ and $\kappa(b)$. Note that $(X_*(T) \otimes \QQ)_{\rm dom}^{\Gamma} \times \pi_1(G)_{\Gamma}$ is the same for both cases of $L$. Furthermore, the explicit description of the image of the maps $\nu$ and $\kappa$ (as for example in \cite{RapoportRichartz}, Remark 1.18) shows that it also does not depend on these cases and that it is independent of the choice of the algebraically closed field $k$. In this way we obtain a bijection of partially ordered sets between $B(G)$ for $L=\Frac W(k)$ and $L=k(\!(z)\!)$.

We now give a description of $B(G,\mu)$. Recall that we denoted by $\mu \in X_*(T)_{\rm dom}$ the dominant representative of the conjugacy class $[\mu]$ from the Shimura datum $\Dscr$. Let $\Gamma_{\mu}$ be the stabilizer of $\mu$ in $\Gamma$ and set
\[
\mgbar := [\Gamma : \Gamma_{\mu}]^{-1} \sum_{\tau \in \Gamma/\Gamma_{\mu}} \tau(\mu) \in (X_*(T) \otimes \QQ)^{\Gamma}_{\rm dom}.
\]
Then by the work of Rapoport and Richartz \cite{RapoportRichartz}, Kottwitz and Rapoport \cite{KR}, Lucarelli \cite{Lucarelli} and Gashi \cite{Gashi}, we have
\begin{equation}\label{DefBGmu}
B(G,\mu) = \set{b \in B(G)}{\nu(b) \leq \mgbar, \kappa(b) = \mu^{\flat}}.
\end{equation}
This is a finite set. The same arguments show the analogous assertion for $L=k((z))$. Note that this condition for non-emptiness of the intersection is a purely group-theoretic condition in terms of $b\in B(G)$ and $\mu\in X_*(T)$, in particular, we obtain a canonical bijection between the partially ordered sets $B(G,\mu)$ for the two cases.

\subsubsection*{The Newton stratification of $\Ascr_0$}
The composition of $\Upsilon$~\eqref{DefUps} with $C(G) \to B(G)$ is given by the map that attaches to a $k$-valued point $\APlus$ of $\Ascr_0$ the isogeny class of its $p$-divisible group with $\Dscr$-structure. We call its image in $B(G)$ the Newton point of $\APlus$. It lies in $B(G,\mu)$. For $s \in \Ascr_0$ (i.e., $s$ is a point of the underlying topological space of $\Ascr_0$) let $\Oscr_{\Ascr_0,s}$ be the local ring of $\Ascr_0$ as Deligne-Mumford stack (if $\Ascr_0$ is a scheme, then $\Oscr_{\Ascr_0,s}$ is a strict henselization of the usual local ring at the point $s$). Choose an algebraically closed extension $k$ of the residue field $\kappa(s)$ of $\Oscr_{\Ascr_0,s}$. We denote the Newton point of the composition $\Spec k \to \Spec \kappa(s) \to \Ascr_0$ by ${\rm Nt}(s)$. The independence of $B(G)$ of $k$ shows that ${\rm Nt}(s)$ does not depend on the choice of $k$. Thus we obtain a map
\[
{\rm Nt}\colon \Ascr_0 \to B(G,\mu).
\]
For $b \in B(G,\mu)$ we denote by $\mathcal{N}_b$ the set of points $s \in \Ascr_0$ such that ${\rm Nt}(s) = b$. By~\cite{RapoportRichartz}, $\Ncal_b$ is a locally closed subset of the underlying topological space of $\Ascr_0$ and there is an inclusion of closed subsets
\begin{equation}\label{NewtonSpec}
\overline{\Ncal_b} \subseteq \bigcup_{b' \leq b} \Ncal_{b'}.
\end{equation}
This is a group-theoretic formulation of Grothendieck's specialization theorem. We endow $\Ncal_b$ with its reduced structure of a locally closed substack (a reduced subscheme if $\Ascr_0$ is a scheme).

\begin{definition}\label{DefineNewton}
$\Ncal_b$ is called the \emph{Newton stratum attached to $b \in B(G,\mu)$}.
\end{definition}

\begin{example}\label{MuOrdBasic}
There is a unique maximal element $b_{\mu}$ and a unique minimal element $b_{\rm basic}$ in $B(G,\mu)$ characterized by $\nu(b_{\mu}) = \mgbar$ and by $\nu(b_{\rm basic}) \in X_*(Z)_{\QQ}$, where $Z \subset T$ is the center of $G$. The corresponding Newton strata $\Ncal_{b_{\mu}}$ and $\Ncal_{b_{\rm basic}}$ are called the \emph{$\mu$-ordinary Newton stratum} and the \emph{basic Newton stratum}, respectively.

By~\eqref{NewtonSpec} $\Ncal_{b_{\rm basic}}$ is closed in $\Ascr_0$. Moreover, \eqref{NewtonSpec} shows that $\Ncal_{b_{\mu}}$ is the complement of the (finite) union of the closures of the non-$\mu$-ordinary Newton strata. Therefore $\Ncal_{b_{\mu}}$ is open in $\Ascr_0$.
\end{example}


\section{Minimal Ekedahl-Oort strata}

\subsection{Minimal Ekedahl-Oort strata and fundamental elements}\label{SecMinFund}
For $p$-divisible groups over an algebraically closed field of positive characteristic Oort defines within each isogeny class of $p$-divisible groups a single isomorphism class which he calls minimal (\cite{Oort2}). His definition is recalled in Remark~\ref{exminpdg} below. It is equivalent to the condition that the ring of endomorphisms of the Dieudonn\'e module of a minimal $p$-divisible group is a maximal order in the endomorphisms of its isocrystal. Oort proves (\cite{Oort2} and \cite{Oort_Simple}) that a $p$-divisible group $X$ is minimal if and only if for any $p$-divisible group $X'$ one has
\begin{equation}\label{OortMinimal}
X[p] \cong X'[p] \implies X \cong X'.
\end{equation}
Similar results for $p$-divisible groups with a principal polarization $\lambda\colon X \iso X\vdual$ follow by using uniqueness of polarizations (\cite{Oort1}~Corollary~(3.8)). For general $p$-divisible groups with $\Dscr$-structure we take the group-theoretic analogue of~\eqref{OortMinimal} as a definition for minimality.

\begin{remark}\label{ExplainMinimal}
Let $k$ be an algebraically closed extension of $\kappa$, $L := \Frac W(k)$, and recall that $K=G(W(k))$ and $K_1 = \Ker(K \to G(k))$. Let $(X,\iota,\lambda)$ be a $p$-divisible group with $\Dscr$-structure over $k$ and let $x \in G(L)$ such that the isomorphism class of $(X,\iota,\lambda)$ is given by the $K$-$\sigma$-conjugacy class $\dbrack{x} \in C_k(G,\mu)$ via~\eqref{PDivGroup}. Then the following conditions are equivalent.
\begin{equivlist}
\item
For all $y \in K_1xK_1$ there exists a $g \in K$ such that $gy\sigma(g)^{-1} = x$.
\item
For all $x' \in \dbrack{x}$ and for all $y \in K_1x'K_1$ there exists a $g \in K$ such that $gy\sigma(g)^{-1} = x'$.
\item
Any $p$-divisible group with $\Dscr$-structure $(X',\iota',\lambda')$ with $(X',\iota',\lambda')[p] \cong (X,\iota,\lambda)[p]$ is isomorphic to $(X,\iota,\lambda)$.
\end{equivlist}
The equivalence of~(i) and~(ii) is immediate because $K_1$ is a normal subgroup of $K$ and the equivalence of~(ii) and~(iii) follows from~\eqref{SameTruncation}.
Clearly the equivalent conditions above depend (for a fixed field $k$) only on the isomorphism class of $(X,\iota,\lambda)[p]$, i.e., on ${\rm tc}(\dbrack{x}) \in \leftexp{J}{W}$~\eqref{Truncation}. They can be expressed by the condition that ${\rm tc}^{-1}({\rm tc}(\dbrack{x}))$ consists of a single element in $C_k(G,\mu)$.
\end{remark}

\begin{definition}\label{defminimal}
\begin{assertionlist}
\item
Let $k$, $L$ and $K_1$ be as in Remark~\ref{ExplainMinimal}. An element $x\in G(L)$ is called \emph{minimal} if the equivalent conditions of Remark \ref{ExplainMinimal} hold. Then each element of $\dlbrack x \drbrack \in C_k(G)$ is also minimal and we call $\dlbrack x \drbrack $ \emph{minimal}. A $p$-divisible group with $\Dscr$-structure over $k$ is called \emph{minimal} if its class in $C_k(G,\mu)$ is minimal.
\item
An element $w \in \leftexp{J}{W}$ (or its corresponding Ekedahl-Oort stratum $\Ascr^w_0$) is called \emph{minimal} if there exists a minimal $\dbrack x \in C_{\kgbar}(G,\mu)$ such that ${\rm tc}(\dbrack x) = w$.
\end{assertionlist}
\end{definition}

\begin{remark}\label{GaloisMinimal}
If $x \in G(L)$ is minimal, then $\sigma(x)$ is also minimal. This shows that if $w \in \leftexp{J}{W}$ is minimal, then each element in the same $\Gamma_k$-orbit of $w$ (Remark~\ref{EOFieldOfDef}) is also minimal.
\end{remark}

\begin{example}\label{MuOrdMinimal}
Moonen (\cite{Mo_SerreTate}, Theorem in Section~0.3) has shown that the $\mu$-ordinary Ekedahl-Oort stratum is minimal in the sense of Definition~\ref{defminimal}.
\end{example}

Our next goal is to study the question whether each Newton stratum contains a minimal Ekedahl-Oort stratum. We first translate this question into group-theoretic language. We obtain our results using a variant of the notion of fundamental elements introduced by G\"ortz, Haines, Kottwitz, and Reuman in~\cite{GHKR2}, and generalized by Viehmann to unramified groups in \cite{trunc1}.

To construct abelian varieties with additional structures we will use the following principle. Let $\Dscr=\bigl(B,\star,V,\lrangle, O_B, \Lambda, h\bigr)$ and $\Dscr'=\bigl(B,\star,V,\lrangle', O_B, \Lambda, h'\bigr)$ be two unramified Shimura-PEL-data which agree except for the data $\lrangle$ and $h$ and such that the $(B,\star)$-skew hermitian spaces $(V,\lrangle)$ and $(V,\lrangle')$ are isomorphic over $\QQ_p$ (after choosing such an isomorphism it thus makes sense to consider the same $\Lambda$). Let $\Gbf$ and $\Gbf'$ be the reductive $\QQ$-groups associated to $\Dscr$ and $\Dscr'$, respectively. Choose compact open subgroups $C \subset \Gbf(\AA^p_f)$ and $C' \subset \Gbf'(\AA^p_f)$ and let $\Ascr_0$ and $\Ascr'_0$ the special fibers of $\Ascr_{\Dscr,C}$ and $\Ascr_{\Dscr',C'}$, defined over finite fields $\kappa$ and $\kappa'$.

\begin{lemma}\label{IsogAV}
Let $k$ be an algebraically closed extension of $\kappa$ and of $\kappa'$, let $(A',\iota',\lambda',\eta') \in \Ascr'_0(k)$, and let $X' = (X',\iota',\lambda')$ be the $p$-divisible group with $\Dscr'$-structure of $(A',\iota',\lambda',\eta')$. Let $X = (X,\iota,\lambda)$ be a $p$-divisible group with $\Dscr$-structure and let $\rho\colon X' \rightarrow X$ be an $O_B$-linear isogeny of $p$-divisible groups compatible with the polarizations on $X'$ and $X$ up to a power of $p$. Then there exist a $k$-valued point $\APlus$ of the moduli space $\Ascr_{0}$ associated with $\Dscr$ and an $O_B$-linear isogeny $f\colon A' \to A$ with $f^*\lambda = p^e\lambda$ for some $e \geq 0$ such that $\APlus[p^{\infty}] = (X,\iota,\lambda)$ (up to isomorphism of $p$-divisible groups with $\Dscr'$-structure) and such that $f[p^{\infty}] = \rho$.
\end{lemma}

It is not claimed that $f$ satisfies any compatibility with the level structures.

\begin{proof}
Dividing $A'$ by $C := \Ker(\rho)$ yields an isogeny $f\colon A' \to A := A'/C$ between $A'$ and an abelian variety $A$ with $A[p^{\infty}] = X$. Furthermore, the $\Dscr'$-structure (resp.\ the level structure) on $A'$ induces an $O_B$-action and a polarization $(\iota,\lambda)$ on $A$ (resp.~a level structure $\eta$ on $A$) whose restrictions to $X$ are the given ones. The determinant condition for $\Dscr$ holds for $A$ because it holds for $X$. 
\end{proof}

\begin{lemma}\label{lemFundAlc}
Let $\APlus$ be a $k$-valued point of $\Ascr_0$ and let $b$ be its image in $B(G,\mu)$. There exists a minimal $k$-valued point $\APlusOne$ which is isogenous to $\APlus$ (i.e., there exists an $O_B$-linear isogeny $f\colon A \to A_1$ such that $f^*\lambda_1 = p^e\lambda$ for some $e \geq 0$ and $f^*\eta_1 = \eta$) if and only if the following condition holds.
\begin{equation}\label{eqstar}
\text{$K\mu(p)K\cap b \subseteq G(L)$ contains a minimal element}.
\end{equation}
\end{lemma}

\begin{proof}
The condition is clearly necessary. Conversely, let $g \in K\mu(p)K \cap b$ be a minimal element. Let $X = (X,\iota,\lambda)$ be the $p$-divisible group with $\Dscr$-structure of $\APlus$ and let $b_0\in G(L)$ be a representative of the corresponding class in $C(G,\mu)$. Let $h\in G(L)$ with $g=h^{-1}b_0\sigma(h)$. Let $X_1$ be a $p$-divisible group with $\Dscr$-structure whose corresponding class is $\dlbrack g\drbrack\in C(G,\mu)$. Then $X_1$ is minimal and $h$ induces an isogeny $\rho\colon X \rightarrow X_1$ of $p$-divisible groups with $\Dscr$-structure. Thus the lemma follows from Lemma~\ref{IsogAV} for $\Dscr=\Dscr'$.
\end{proof}

\begin{definition}\label{defpfund}
\begin{assertionlist}
\item
Let $F$ be a finite unramified extension of $\QQ_p$ and let $P$ be a semistandard parabolic subgroup of $G_{O_F}$, i.e.~a parabolic subgroup containing $T_{O_F}$ but not necessarily $B_{O_F}$. Let $N$ be its unipotent radical and let $M$ be the Levi factor containing $T_{O_F}$. Let $\overline{N}$ be the unipotent radical of the opposite parabolic. Let $\Ical_M = \Ical \cap M(O_L)$ and similarly $\Ical_N = \Ical \cap N(O_L)$ and $\Ical_{\overline{N}} = \Ical \cap \overline{N}(O_L)$. Then an element $x \in \widetilde{W}$ is called \emph{$P$-fundamental} if $\sigma({\dot x}\Ical_M{\dot x}^{-1}) = \Ical_M$, $\sigma({\dot x}\Ical_N{\dot x}^{-1}) \subseteq \Ical_N$, and $\sigma({\dot x}\Ical_{\overline N}{\dot x}^{-1})\supseteq \Ical_{\overline N}$.
\item
We call $x \in \Wtilde$ \emph{fundamental} if it is $P$-fundamental for some finite unramified extension $F$ and some semistandard parabolic subgroup $P$ of $G_{O_F}$.
\item
We call a class $c \in C(G)$ \emph{fundamental} if there exists a fundamental element $x \in \Wtilde$ and a representative $\dot x$ such that $c = \dbrack{\dot x}$. If $c \in C(G,\mu)$ is fundamental, we also call the corresponding isomorphism class of $p$-divisible groups with $\Dscr$-structure \emph{fundamental}.
\end{assertionlist}
\end{definition}

This definition generalizes the notion of $P$-fundamental elements in \cite{GHKR2} from split groups to unramified groups, see also \cite{trunc1}, Definition 6.1. In \cite{GHKR2}, fundamental elements in $\widetilde{W}$ are defined by a condition which is a priori weaker than ours (i.e.~as elements for which the first assertion of Remark \ref{remFundAlc} holds). We did not consider the question whether elements in $\widetilde{W}$ which are fundamental in their sense are also fundamental in our sense.

\begin{remark}\label{remFundAlc}
For a fundamental element $x \in \widetilde{W}$ and in the equi-characteristic case, G\"{o}rtz, Haines, Kottwitz, and Reuman show in \cite{GHKR2}~Proposition~6.3.1 that every element of $\Ical\dot x\Ical$ is $\Ical$-$\sigma$-conjugate to $\dot x$. The same argument still shows this property in our more general situation.

This implies in particular that every element of $K_1\dot xK_1$ (where $K_1 := \Ker(K\to G(k)) \subset \Ical$) is $K$-$\sigma$-conjugate to $\dot x$ (where $K = G(W(k)) \supset \Ical$). Hence if $x \in \Wtilde$ is a fundamental element, then every representative $\dot x$ of $x$ is minimal. In particular we see that fundamental $p$-divisible groups with $\Dscr$-structure are minimal.
\end{remark}

\subsection{The split case}
If $G$ is split, we use results of~\cite{GHKR2} to show the following proposition:

\begin{proposition}\label{FundAlcSplit}
Assume that $G$ is split and that $\mu$ is minuscule. Then each element $b$ of $B(G)$ contains the representative of a fundamental element of $\Wtilde$. Any two such fundamental elements in $\Wtilde$ are conjugate under $W$. In particular, $b$ contains a unique fundamental element $c_b$ of $C(G)$. If $b \in B(G,\mu)$, then $c_b \in C(G,\mu)$.
\end{proposition}

\begin{proof}
By \cite{GHKR2}~Corollary~13.2.4 each element $b$ of $B(G)$ contains a fundamental element $x_b$, and \cite{trunc1}, Lemma 5.3 and 6.11 show that any two fundamental elements in $b$ are conjugate under $W$. To be precise, \cite{GHKR2} consider the equi-characteristic case, but the same proof shows the analogous result in our situation. It remains to show the last assertion. Thus it is enough to show that for each $b \in B(G,\mu)$ there exists a fundamental element $x \in W\mu W$ with $\dot x \in b$. We have already seen that there is a fundamental element $x_b$ in $b$. Let $\mu' \in X_*(T)$ be dominant with $x_b\in W\mu'W$. We want to show that $\mu' = \mu$. As $\mu$ is minuscule, it is enough to show that $\mu'\leq\mu$. Note that this is a statement purely in terms of the affine Weyl group and of the given element of $B(G,\mu) \hookrightarrow X_*(T)_{\mathbb{Q}}\times \pi_1(G)$. In particular, we can prove this assertion using equi-characteristic methods, i.e.~replacing $W(k)$ by $k\dlbrack z\drbrack$. As $b\in B(G,\mu)$ there is an element $g\in G(k((z)))$ with $g\in G(k\dlbrack z\drbrack)\mu(z)G(k\dlbrack z\drbrack)$ and in the $\sigma$-conjugacy class in $G(k((z)))$ associated with $b$. By \cite{trunc1}, Proposition 5.5 $\dot x_b$ lies in the closure of the double coset $G(k\dlbrack z\drbrack)\mu(z)G(k\dlbrack z\drbrack)$. In particular, $\mu'\leq\mu$.
\end{proof}

The preceding result and its proof are valid for an arbitrary split reductive group scheme. From now on we assume again that $G$ is associated with a PEL-Shimura-datum.

\begin{theorem}\label{MinimalNewtonSplit}
Assume that $G$ is split. Then for every $b \in B(G,\mu)$ there exists a unique Ekedahl-Oort stratum $\Ascr^{w(b)}_0$ that corresponds to (a representative of) a fundamental element and such that $\Ascr_0^{w(b)} \subseteq \Ncal_b$.
\end{theorem}

By Remark~\ref{remFundAlc}, $\Ascr^{w(b)}_0$ is a minimal Ekedahl-Oort stratum.

\begin{proof}
By Proposition~\ref{FundAlcSplit} there exists a fundamental element $x \in \Wtilde$ such that $\dbrack{\dot x} \subseteq b$ and $\dbrack{\dot x} \in C(G,\mu)$. Moreover $x$ is unique up to $W$-conjugation (which is the same as $W$-$\sigma$-conjugation because $G$ is split) in $\Wtilde$. This shows that $\dbrack{\dot x}$ depends only on $b$. We call its image under the truncation map~\eqref{Truncation} $w(b)$. By definition, the corresponding Ekedahl-Oort stratum $\Ascr^{w(b)}_0$ meets the Newton stratum $\Ncal_b$. But by Remark~\ref{remFundAlc} $\Ascr^{w(b)}_0$ is minimal and thus is contained in $\Ncal_b$.
\end{proof}

By Lemma~\ref{lemFundAlc} we obtain

\begin{corollary}
Let $\Ascr_0$ be a moduli space of abelian varieties with $\Dscr$-structure such that the associated group $G$ is split. Let $\APlus$ be a $k$-valued point of $\Ascr_0$. Then there exists a minimal $k$-valued point $\APlusOne$ which is isogenous to $\APlus$.
\end{corollary}

\begin{remark}\label{exminpdg}
We compare fundamental elements for $G = GL_h$ and minimal $p$-divisible groups in the sense of Oort (\cite{Oort2}) (for short we will call them Oort-minimal). Let us first recall the definition of Oort-minimal $p$-divisible groups over an algebraically closed field $k$ of positive characteristic. By definition there is a unique isomorphism class of Oort-minimal $p$-divisible groups in each isogeny class of $p$-divisible groups over $k$. It is given as follows. Let $\mathbf N$ be the isocrystal corresponding to the isogeny class. If $X$ is minimal in the given class, its Dieudonn\'e module is isomorphic to a Dieudonn\'e module of the following form. There is a $\QQ_p$-rational decomposition $\mathbf N = \bigoplus_{i=1}^l \mathbf N_i$ of $\mathbf N$ into simple isocrystals $\mathbf N_i$ such that $\mathbf M = \bigoplus_{i=1}^l \mathbf M\cap \mathbf N_i$ and such that $\mathbf M\cap \mathbf N_i$ has a basis $e^i_1,\dotsc,e^i_{h_i}$ with $F(e^i_j)=e^i_{j+n_i}$. Here $\lambda_i=n_i/h_i$ with $(n_i,h_i)=1$ is the slope of $\mathbf N_i$ and we use the notation $e^i_{j+h_i}=pe^i_j$. 

Let now $f^i_j = e^i_{h_i+1-j}$. Then $f^1_1,\dotsc,f^1_{h_1},f^2_1,\dotsc$ is a basis of $\mathbf N$. Let $h := \dim \mathbf N$. One easily checks that if we write $F=b\sigma$ for $b\in GL_{h}(L)$ with respect to the given basis, then $b$ is contained in the Levi subgroup $M$ given by the decomposition $\mathbf N=\bigoplus_{i=1}^l\mathbf N_i$. We fix the diagonal torus $T$ in $G = GL_h$ and the Borel subgroup of upper triangular matrices containing it. If $\mu$ denotes the $M$-dominant Hodge polygon of $b$ (with respect to these choices), then $\mu\in\{0,1\}^h \subset \ZZ^h = X_*(T)$ is minuscule and $b$ satisfies $b\Ical_Mb^{-1}=\Ical_M$. By \cite{trunc1}, Lemma 6.11 this implies that the Dieudonn\'e module also has a trivialization $(\mathbf M,F) \cong (W(k)^n,b'\sigma)$ such that $b' \in G(L)$ is a representative of a fundamental element in $\widetilde W$. Thus an Oort-minimal $p$-divisible group of height $h$ (or its corresponding class in $C(G)$) corresponds to a representative of a fundamental element.

The same argument works for $G = GSp_{2g}$ by working with isocrystals where simple isocystals of slope $\lambda$ occur with the same multiplicity as isocrystals of slope $1-\lambda$.
\end{remark}

Thus we see that Oort-minimal $p$-divisible groups (without any additional structure) are fundamental, and fundamental $p$-divisible groups are minimal by Remark~\ref{remFundAlc}. We obtain a new proof of the main result of~\cite{Oort2}:

\begin{corollary}\label{OortMinimalMinimal}
Oort-minimal $p$-divisible groups are minimal in the sense of Definition~\ref{defminimal}.
\end{corollary}

If $G$ is split (i.e. isomorphic to a product of reductive groups of the form $GL_n$ or $\GSp_{2g}$), then -- as explained above -- Oort has shown in a series of papers that each Newton stratum contains a unique minimal Ekedahl-Oort stratum. This then shows that the notion of Oort-minimal, fundamental and minimal $p$-divisible groups coincide in this case.

In general the existence of minimal Ekedahl-Oort stratum within a given Newton stratum is not clear. Lemma~\ref{lemFundAlc} shows that for $b \in B(G,\mu)$ the corresponding Newton stratum $\Ncal_b$ contains a minimal Ekedahl-Oort stratum if the following group-theoretical condition holds.
\begin{equation}\label{ExistFundamental}
\text{There is a fundamental element in $W\mu W$ whose image in $B(G)$ is equal to $b$.}
\end{equation}
\begin{example}\label{exhb}
The uniqueness of minimal Ekedahl-Oort strata within a given Newton stratum does not hold in general. One example of this phenomenon is the ``unramified inert Hilbert-Blumenthal case'', i.e., the case where we have an exact sequence
\[
1 \to \Res_{K/\QQ_p}SL_{2,K} \to G_{\QQ_p} \to \GG_m \to 1,
\]
where $K$ is an unramified extension of $\QQ_p$ of some degree $g$. We give an explicit example that for $g \geq 6$ there exist minimal Ekedahl-Oort strata which are not in the same Galois orbit, but which are in the same Newton stratum. We consider $G_{\QQ_p}$ as a subgroup of $\Res_{K/\QQ_p}GL_{2,K}$. Let $T$ and $B$ be the maximal torus and the Borel subgroup of $G$ consisting of diagonal matrices and of upper triangular matrices, respectively. We identify $\Delta=\Gal(K/\QQ_p)$ with $\ZZ/g\ZZ$. The Weyl group $W$ is isomorphic to $S_2^{\Delta}$ where $S_2=\{1,s\}$ is the symmetric group of two elements. For $\tau\in\Delta$ we denote the non-trivial element in the $\tau$th factor by $s_\tau$. We have
\[
X_*(T)\cong \set{((g_{\tau,1},g_{\tau,2}))_{\tau}\in(\mathbb{Z}^2)^{\Delta}}{\text{$\exists\,c\in\mathbb{Z}: g_{\tau,1}+g_{\tau,2}=c$ for all $\tau$}}.
\]
We have $\mu=((1,0),\ldots,(1,0))$. We consider two particular elements $\phi_1=(1,0)$ and $\phi_2=(0,1)s$ in $\widetilde W_{\GL_2}\cong S_2\ltimes \mathbb{Z}^2$. Let $x,x'\in \widetilde W$ with decomposition $x=\prod x_\tau$ and $x'=\prod x'_\tau$. Let $x_{5}=x_6=x'_3=x'_6=\phi_1$ and all other $x_\tau$ and $x'_\tau$ equal to $\phi_2$ (viewed as an element of the corresponding component of $\widetilde W$). Then $x$ and $x'$ are not in the same Galois orbit in $\widetilde W$. We have $x=y^{-1}x'\sigma(y)$ where $y=s_3s_4(-1,1)_4$, and thus $x,x'$ are in the same Newton stratum. 

We define two parabolic subgroups $P,P'$ of $G_{\QQ_p}$ as intersections of $G_{\QQ_p}$ with the following parabolic subgroups $\tilde P,\tilde P'$ of $\Res_{K/\QQ_p}GL_{2,K}$. Let $\tilde P_2,\tilde P_4, \tilde P'_2,\tilde P'_5$ be the subgroup of lower triangular matrices, and let all other $\tilde P_\tau$ and $\tilde P'_{\tau}$ be the subgroups of $\GL_2$ of upper triangular matrices. Let $\tilde P_K=\prod_\tau \tilde P_\tau$ and similarly for $\tilde P'$. Then an explicit calculation shows that $x$ is $P$-fundamental and $x'$ is $P'$-fundamental. Thus they induce two minimal Ekedahl-Oort strata contained in the same Newton stratum.
\end{example}

\subsection{The basic case}
Although we do not know whether in general an isogeny class of $p$-divisible groups with $\Dscr$-structure contains a minimal or even a fundamental $p$-divisible group with $\Dscr$-structure, we will now explain that this is true in the $\mu$-ordinary and in the basic case.

\begin{remark}\label{MuOrdEONewton}
Moonen (\cite{Mo_SerreTate}, Theorem in Section~0.3) has shown that the $\mu$-ordinary Newton stratum is equal to the $\mu$-ordinary Ekedahl-Oort stratum and that this is minimal.
\end{remark}

\begin{proposition}\label{FundAlcBasic}
All basic classes $b \in B(G)$ contain a fundamental element of $C(G)$. If $b \in B(G,\mu)$ is basic, then it contains a fundamental element of $C(G,\mu)$.
\end{proposition}

\begin{proof}
We use the notation from Subsection~\ref{AffWeyl}. Let $x \in \Omega$ be such that its image under the surjection $\Omega\rightarrow \pi_1(G)\rightarrow \pi_1(G)_{\Gamma}$ is equal to $\kappa(b)$. Then $[\dot x] \in B(G)$ is by definition basic and has the same image in $\pi_1(G)_{\Gamma}$ as $b$. Thus $[\dot x] = b$. Furthermore $x \in \Omega$ implies that $x$ is $P$-fundamental for $P=G$.

Let $b \in B(G,\mu)$ for some $\mu$ correspond to the isogeny class of a $p$-divisible group with $\Dscr$-structure. Then one can choose $x \in \Omega$ to be the unique element whose image in $\pi_1(G)$ is equal to the image of $\mu \in X_*(T)$ under the projection to $\pi_1(G)$. Then $\dot x \in K\mu(p)K$.
\end{proof}

Note that this a purely group-theoretic result and its proof works for any reductive group scheme $G$ over $\ZZ_p$. For basic Newton strata one has the following explicit description of a corresponding minimal Ekedahl-Oort stratum.

\begin{proposition}\label{remminbasic}
The Ekedahl-Oort stratum $\Ascr_0^1$ for $w=1$ is minimal and non-empty. If $b \in B(G,\mu)$ is the basic element, then $\Ascr_0^1\subseteq\mathcal{N}_b$. 
\end{proposition}

\begin{proof}
By \cite{Fargues}~3.1.8 or \cite{Ko_ShFin}~\S18, the basic Newton stratum is non-empty. Let $\APlus$ be a $k$-valued point of $\mathcal{N}_b$. Let $(X,\iota,\lambda)$ be a $p$-divisible group with $\Dscr$-structure whose associated $\Dscr$-zip is of type $w=1$. We want to show that $(X,\iota,\lambda)$ is minimal and isogenous to the $p$-divisible group with $\Dscr$-structure of $\APlus$. Indeed, then the non-emptiness of the stratum $\Ascr_0^1$ follows as in the proof of Lemma \ref{IsogAV} and the inclusion $\Ascr_0^1\subseteq\mathcal{N}_b$ follows from the definition of minimality. Thus by Remark~\ref{DZipToPDiv} and Remark \ref{remFundAlc} it is enough to show that $x: = w_0w_{0,\mu}\cdot\mu\in W\ltimes X_*(T)\cong\widetilde{W}$ is fundamental and that its image in $B(G,\mu)$ is basic. We show that $x\Ical x^{-1}=\Ical$, where $\Ical$ is the chosen Iwahori subgroup. Then $x$ is $P$-fundamental for $P=M=G$ and $x\in\Omega$ and hence its image in $B(G,\mu)$ is basic. To prove $x\Ical x^{-1}=\Ical$ we use the decomposition of $\Ical$ into $T\cap \Ical=T(O_L)$ and $U_{\alpha}\cap \Ical$ for all root subgroups $U_\alpha$. We have $x(T\cap \Ical)x^{-1}=T\cap \Ical$. Furthermore $$U_{\alpha}\cap \Ical=\begin{cases}
U_{\alpha}(O_L)&\text{if }\alpha>0\\
\{g\in U_{\alpha}(O_L)\mid g\equiv 1\pmod{p}\}&\text{if }\alpha<0.
\end{cases}$$ We have to show 
\begin{equation}\label{eqassert}
w_0w_{0,\mu}\mu(p)(U_{\alpha}(L)\cap \Ical)(w_0w_{0,\mu}\mu(p))^{-1}=w_0w_{0,\mu}\mu(p)U_{\alpha}(L)(w_0w_{0,\mu}\mu(p))^{-1}\cap \Ical
\end{equation}
for every $\alpha$. Let $P_{\mu}$ be the standard parabolic subgroup associated with $\mu$ and let $M_{\mu}$ and $N_{\mu}$ be its Levi factor containing $T$ and its unipotent radical. Let $\overline{N_{\mu}}$ be the unipotent radical of the opposite parabolic. For roots $\alpha$ of $T$ in $M_\mu$ conjugation by $\mu$ acts trivially on $U_{\alpha}$, and conjugation by $w_0w_{0,\mu}$ maps positive roots in $M_{\mu}$ to positive roots (not necessarily in $M_{\mu}$) and negative roots in $M_{\mu}$ to negative roots. Thus (\ref{eqassert}) holds for roots of $T$ in $M$. For roots $\alpha$ of $T$ in $N_{\mu}$ we have $\langle \mu,\alpha\rangle=1$, hence conjugation by $\mu$ maps $U_{\alpha}\cap \Ical= U_{\alpha}(O_L)$ to $\{g\in U_{\alpha}(O_L)\mid g\equiv 1\pmod{p}\}$. Furthermore $w_0w_{0,\mu}$ maps these roots to $\overline{N_{\mu}}$, hence to negative roots. Together with a similar consideration for the roots in $\overline{N_{\mu}}$ we obtain that $w_0w_{0,\mu}\mu(p)(U_{\alpha}\cap \Ical)(w_0w_{0,\mu}\mu(p))^{-1}=w_0w_{0,\mu}\mu(p)U_{\alpha}(w_0w_{0,\mu}\mu(p))^{-1}\cap \Ical$ for every $\alpha$. 
\end{proof}

\subsection{The general case}

In general we do not know whether an isogeny class of $p$-divisible groups with $\Dscr$-structure contains a fundamental $p$-divisible group with $\Dscr$-structure. However, we show in this subsection the slightly weaker statement that each isogeny class of $p$-divisible groups with $O_B$-action and polarization (but without requiring a determinant condition) contains a fundamental $p$-divisible group with this additional structure. For non-split groups, fundamental $p$-divisible groups in a given isogeny class are in general not unique, compare Example \ref{exhb}.

\begin{theorem}\label{thmexfundalc}
Let $G$ and $\mu$ be the reductive $\ZZ_p$-group and the minuscule cocharacter associated with a Shimura-PEL-datum and let $b\in B(G_{\QQ_p}, \mu)$. Then there exists an $x\in \widetilde W$ such that $x\in b$ is fundamental and such that $x\in W\mu'$ for some minuscule coweight $\mu'$.
\end{theorem}

In particular, $x$ satisfies condition (\ref{eqstar}). The theorem and its proof are a refinement of a corresponding but slightly weaker statement for split groups, cf. \cite{trunc1}, Theorem 6.5.

For the proof of the theorem we may assume that we are in one of the cases (AL), (AU), or (C) of Remark~\ref{CasesOB}. As a preparation of the proof we construct an explicit representative of $b$. Let $\nu\in X_*(T)_{\mathbb{Q},\rm dom}$ be the dominant Newton point of $b$ and let $M_b$ be its centralizer. There is a standard parabolic subgroup of $G_{\QQ_p}$ whose Levi subgroup $M$ containing $T$ is minimal among those Levi subgroups of $G_{\QQ_p}$ that are defined over $\mathbb{Q}_p$ and contain an element $b'\in M(L)\cap b$ whose $M$-dominant Newton point is equal to $\nu$ (compare \cite{ckv}, Section 3.1). In other words, $b$ is superbasic in $M$. By \cite{ckv}, Lemma 3.1.2 the adjoint group $M^{\ad}$ is then isomorphic to a product of groups of the form $\Res_{F_i/\mathbb{Q}_p}PGL_{n_i}$ for some finite unramified extension $F_i$ of $\mathbb{Q}_p$ and some $n_i$. Using this together with the explicit form of the groups $G$ associated with Shimura PEL-data (Remark~\ref{StructureG}) one obtains that $M$ is of the form
\[
M \cong M'_0 \times \prod_{i=1}^t M_i,
\]
where $M_i = \Res_{F_i/\mathbb{Q}_p}\GL_{n_i}$ for some unramified extension $F_i$ of $\QQ_p$ and some $n_i \geq 1$ and where $M'_0$ is either isomorphic to $\GG_{m,\QQ_p}$ (in Case~(AL) or if in the isocrystal corresponding to $\rho(b)$ ($\rho\colon G_{\QQ_p} \mono \GL(V_{\QQ_p})$ the standard representation) the multiplicity of the simple isocrystal with slope $1/2$ is even) or where $M'_0$ sits in a non-split exact sequence
\begin{equation}\label{eqformm'}
0\rightarrow\Res_{F_0/\mathbb{Q}_p}SL_{2}\rightarrow M_0'\rightarrow \mathbb{G}_m\rightarrow 0,
\end{equation}
where $F_0$ is again unramified over $\QQ_p$.

The image $b'_i$ of $b'$ in each factor $M_0'$ or $M_i=\Res_{F_i/\mathbb{Q}_p}GL_{n_i}$ is basic and the valuation $q_i$ of its determinant is coprime to $n_i$. If $M_0'$ is as in \eqref{eqformm'} then $q_0=[F_0:\mathbb{Q}_p]$. If $M_0'\cong \mathbb{G}_m$ then $q_0=1$.  If $M_0'$ is as in \eqref{eqformm'} let $M_0 := \Res_{F_0/\mathbb{Q}_p}\GL_{2}$ and $n_0 = 2$, otherwise let $M_0 := M_0'$, $F_0 := \mathbb{Q}_p$, and $n_0 = 1$. For $i = 0,\dots,t$ we let $I_i$ be the set of $\QQ_p$-embeddings $F_i \to \QQbar_p$. Fixing one such embedding we identify $I_i$ with $\ZZ/d_i\ZZ$ or $\{1,\dots,d_i\}$, where $d_i := [F_i : \QQ_p]$.

Over $L$, the group $M_i$ decomposes into a product $\prod_{I_i} \GL_{n_i,L}$. For each $\tau\in I_i$ let $(e_{\tau,l})_{1\leq l\leq n_i}$ denote the standard basis of $L^{n_i}$. For $l\in \mathbb{Z}$ we define $e_{\tau,l}$ by $e_{\tau,l+n_i}=pe_{\tau,l}$. We want to define an element $x$ of $M(L)$ by defining its projection $x_i$ to each factor $M_0'(L)$ and $M_i(L) = \prod_{I_i}\GL_{n_i,L}$. For $i = 0,\dots,t$ define $x_i \in M_i(L)$ by
\[
x_i(e_{\tau,l}) := e_{\tau,l+\lfloor \frac{\tau q_i}{d_i}\rfloor-\lfloor \frac{(\tau-1) q_i}{d_i}\rfloor}.
\]
Then $\det x_i = \pm p^{q_i}$. Thus if $i=0$ and $M_0'$ is as in \eqref{eqformm'}, then $\det x_i \in \QQ_p^{\times}$ and we obtain that indeed $x_0\in M_0'(L)$. Furthermore, $(x_i\sigma)^{d_in_i}=p^{q_i}$, hence $x_i$ is basic and thus $x$ is contained in the given $\sigma$-conjugacy class $b$. Using the explicit form of $x$ one can see that we obtained a representative $x$ of $b$ with the following properties.
\begin{assertionlist}
\item $x$ is contained in the image of the embedding of $\widetilde W_M$ into $G(L)$.
\item Let $\nu_x=\nu$ and $\mu_x$ be the $M$-dominant Newton- and Hodge-polygon of $x$. Then for every root $\alpha$ of $T$ in $G$ we have $$|\langle \alpha,\mu_x-\nu_x\rangle|<2.$$ 
\item $\mu_x$ is minuscule in $G$.
\item Let $N$ be the unipotent radical of $P=MB$, in particular $\langle \alpha,\nu\rangle\geq 0$ for each root $\alpha$ of $T$ with $U_{\alpha}\subseteq N$. Then the previous property implies that $\langle \alpha,\mu_x\rangle$ (being an integer) is at least $-1$.
\item Properties (1), (2) and (4) also hold for $x\sigma(x)\dotsm\sigma^m(x)$ for all $m\geq 0$.
\end{assertionlist}

\begin{proof}[Proof of Theorem \ref{thmexfundalc}]
Let $M,N,P,$ and $x$ be as above. We denote by $\phi_x$ the morphism $g\mapsto \sigma(x^{-1}gx)$ on $G(L)$.\\

\noindent{\it Claim.} There is an Iwahori subgroup $\tilde \Ical\subset K=G(W(k))$ with $\tilde \Ical\cap M=\Ical\cap M$ and $\phi_x(\tilde \Ical\cap N)\subseteq \tilde \Ical\cap N$.\\

We first show that this claim implies the theorem. As $x\in \Omega_{M}$ we have $\phi_x(\Ical_M)=\sigma(\Ical_{M})=\Ical_{M}$. By \cite{trunc1}, Lemma 6.6 we have $\phi_{x}(\tilde \Ical\cap \overline{N})\supseteq \tilde \Ical\cap \overline N$. As $\tilde \Ical\subseteq K$ there is a $y\in W$ with $y^{-1}\tilde \Ical y=\Ical$. Let $M'=y^{-1}My$ and $P'=y^{-1}Py$ with unipotent radical $N'$ and opposite $\overline N'$. Then $y^{-1}\tilde \Ical_{M}y=\Ical_{M'}$. Let $x'=\sigma^{-1}(y)^{-1}xy\in b$. We have
\begin{eqnarray*}
\sigma(x' \Ical_{M'}(x')^{-1})&=&\sigma(x'y^{-1}\tilde \Ical_{M}y(x')^{-1})\\
&=&y^{-1}\sigma(x\tilde \Ical_{M}x^{-1})y\\
&=&\Ical_{M'}
\end{eqnarray*}
and similar translations for $N'$ and $\overline N'$. Hence $x'$ is a $P'$-fundamental alcove in $b$. From Property~(3) of $x$ and the definition of $x'$ we obtain that the Hodge polygon of $x'$ is also minuscule. The theorem follows.

It remains to show that the claim is true. Let $r>0$ be such that $G$ is split over some unramified extension of $\mathcal{O}_F$ of degree $r$. In particular, $\sigma^r$ then acts trivially on the root system of $G$ and on $\widetilde W$. Let $c:=\sigma(x)\sigma^{2}(x)\dotsm \sigma^r (x)\in \widetilde W$. Applying the decomposition $\widetilde W=W\ltimes X_*(T)$ we obtain $c=w_c\mu_c$. Let $n_c$ be the order of $w_c$ in $W$. Replacing $r$ by $n_cr$ and using $\sigma(x)\sigma^{2}(x)\dotsm \sigma^{rn_c}(x)=c^{n_c}\in X_*(T)$ we may assume that we already have $c\in X_*(T)$. From the explicit definition of $x$ we see that $c=r\nu$ as elements of $X_*(T)_{\mathbb{Q}}$. In particular, $c$ is dominant in $G$ and central in $M$. 

We define $\Ical'$ as the image of the standard Iwahori subgroup $\Ical$ under conjugation by $w_{0,M}w_0$. It satisfies $\Ical'_M=\Ical'\cap M=\Ical\cap M$, $\Ical'_N=\Ical'\cap N=K_1\cap N$ and $\Ical'_{\overline N}=\Ical'\cap \overline N=K\cap \overline N$.

The conjugation action of $c\in X_*(T)$ on root subgroups can be described as follows. Let $\alpha$ be a root and $U_{\alpha}$ the corresponding root subgroup. Then $p^cU_{\alpha}(\epsilon)p^{-c}=U_{\alpha}(p^{\langle \alpha,c\rangle}\epsilon)$. Using the fact that $\sigma^r$ acts trivially on $\widetilde W$ we obtain that $\sigma^r(c\Ical'_{M}c^{-1})=c\Ical'_{M}c^{-1}=\Ical'_{M}$ and $\sigma^r(c\Ical'_{N}c^{-1})= c\Ical'_{N}c^{-1}\subseteq \Ical'_{N}$. Hence $c$ itself is a $P$-fundamental alcove for the $\sigma^r$-conjugacy class of $c$ in $G$. Let $\tilde \Ical$ with $\tilde \Ical\cap M=\Ical'_{M}$ be unique Iwahori subgroup of $G(L)$ which has in addition the property that $\tilde \Ical\cap N$ is minimal containing $\Ical'_{N}, \phi_{x}(\Ical'_{N}),\dotsc, \phi_{x}^{r-1} (\Ical'_{N})$, cf. \cite{trunc1}, Lemma 6.8. Then $\phi_{x}(\tilde \Ical)$ is again an Iwahori subgroup. It satisfies $\phi_{x}(\tilde \Ical)\cap M=\phi_{x}(\tilde \Ical\cap M)=\Ical_{M}$ and the analogous minimality property for $\phi_{x}(\Ical'_{N}),\dotsc, \phi_{x}^{r} (\Ical'_{N})$. We have $\phi_{x}^r(\Ical'_{N})=\sigma^r(c\Ical'_{N}c^{-1})\subseteq \Ical'_{N}$. Thus $\phi_{x}(\tilde \Ical'\cap N)\subseteq \tilde \Ical'\cap N$. It remains to show that $\tilde \Ical\subseteq K$. This is equivalent to $\tilde \Ical_{M},\tilde \Ical_N,\tilde \Ical_{\overline N}\subseteq K$. We have $\tilde \Ical_M=\Ical_M\subseteq K$. The other two containments are equivalent to $K_1\cap N\subseteq \tilde \Ical_N\subseteq K\cap N=\Ical\cap N$. As $K_1\cap N=\Ical'_N$, the first of these inclusions follows from the definition of $\tilde \Ical_N$. For the second we have to show that $\Ical'_{N}, \phi_{x}(\Ical'_{N}),\dotsc, \phi_{x}^{r-1} (\Ical'_{N})$ are all contained in $\Ical\cap N$. We decompose $\Ical'_N$ into root subgroups. For each root $\alpha$ of $T$ in $N$ we have to show that $\phi_{x}^n(U_{\alpha}(p\mathcal{O}_L))\subseteq \Ical\cap N$ for $n=0,\dotsc,r-1$. We write $x_n=\sigma(x)\sigma^2(x)\dotsm\sigma^{n}(x)$. By properties 4. and 5. of $x$ it is of the form $w_n\mu_n$ for some $w_n\in W_M$ and $\mu_n$ satisfying $\langle\beta,\mu_n\rangle\geq -1$ for each root $\beta$ of $T$ in $N$. Hence 
\begin{eqnarray*}
\phi_{x}^n(U_{\alpha}(p\mathcal{O}_L))&=&x_n\sigma^n(U_{\alpha}(p\mathcal{O}_L))x_n^{-1}\\
&=&w_n\mu_n(U_{\sigma^n(\alpha)}(p\mathcal{O}_L))\mu_n^{-1}w_n^{-1}\\
&=&w_n U_{\sigma^n(\alpha)}(p^{1+\langle \sigma^n(\alpha),\mu_n\rangle}\mathcal{O}_L)w_n^{-1}\\
&\subset& \Ical\cap N
\end{eqnarray*}
which finishes the proof of the claim and of the theorem. 
\end{proof}

\subsection{Generic Newton strata in Ekedahl-Oort strata in the split case}
We now determine the set of Newton strata which intersect the closure of a given Ekedahl-Oort stratum in a moduli space of abelian varieties with $\Dscr$-structure where the associated group $G$ is split. 

The following lemma shows that the existence of fundamental alcoves implies a combinatorial criterion for non-emptiness of intersections of Iwahori-double cosets and Newton strata. In particular, it can be used to compare these intersections for the two cases $L=\Frac (W(k))$ and $L=k(\!(z)\!)$, compare for example the proof of Theorem~\ref{corEONP}.

\begin{lemma}[\cite{GHKR2}, Proposition 13.3.1]\label{CombLemma}
Let $x_b \in \widetilde{W}$ be a fundamental alcove associated with an element of $B(G)$. Let $b$ be the associated $\sigma$-conjugacy class in $G(L)$ (for $L=k\dlbrack z\drbrack$ or $L=\Frac(W(k))$). Let $\Ical$ be the standard Iwahori subgroup defined as the inverse image of $B(k)$ under the projection $G(O_L)\rightarrow G(k)$. Then
\[
\set{x \in \widetilde W}{\Ical\dot x\Ical \cap b \neq \emptyset} = \{\text{$\dot x\in \Ical\dot y^{-1}\Ical\dot x_b\Ical\sigma(\dot y)\Ical$ for some $\dot y\in \widetilde{W}$}\}.
\]
\end{lemma}

\begin{proof}
To prove the containment $\subseteq$ let $b_0\in b$ and assume that $\dot x=g^{-1}b_0\sigma(g)\in \Ical\dot x\Ical$ for some $g\in G(L)$. Using the Bruhat-Tits decomposition we have $g\in \Ical\dot y\Ical$ for some $y\in \widetilde{W}$. Hence $x$ is contained in the right hand side. For the other inclusion assume that $\Ical\dot x\Ical \subseteq \Ical\dot y^{-1}\Ical\dot x_b\Ical\sigma(\dot y)\Ical$. Then $\Ical\dot x\Ical$ meets $\dot y^{-1}\Ical\dot x_b\Ical\sigma(\dot y)$. By Remark~\ref{remFundAlc} every element of $\Ical\dot x_b\Ical$ can be written as $i^{-1}\dot x_b\sigma(i)$ for some $i \in \Ical$. Hence there is an element in $\Ical\dot x\Ical$ of the form
$\dot y^{-1}i^{-1}\dot x_b\sigma(i\dot y)$. Thus $\Ical\dot x\Ical \cap b \neq \emptyset$.
\end{proof}

\begin{theorem}\label{corEONP}
Let $\Dscr$ be such that the associated group $G$ is split. Let $w \in \leftexp{J}{W}$ and $b \in B(G,\mu)$ such that $\Ascr_0^w \cap \Ncal_b \ne \emptyset$.  Let $\Ascr_0^{w'}$ be the unique Ekedahl-Oort stratum containing fundamental elements associated with $b$. Then $\Ascr_0^{w'} \subseteq \overline{\Ascr_{0}^{w}}$.
\end{theorem}

For the Siegel moduli space this criterion has been conjectured by Oort (\cite{Oort1}, Conjecture 6.9) and has been shown by Harashita in \cite{H3} using different methods.

\begin{proof}
By Theorem~\ref{corclosure} we have to show that there is a $y\in W_J$ with $y^{-1}w'x_{\mu}\sigma(y)x_{\mu}^{-1} \leq w$ with respect to the Bruhat order.

\noindent {\it Claim.} Let $\nu(b)\in X_*(T)_{\QQ,{\rm dom}}$ be the Newton point of $b\in B(G,\mu)$. Then we have $\Ascr_0^w\cap\mathcal{N}_b\neq\emptyset$ if and only if the $\sigma$-conjugacy class in the loop group $LG$ corresponding to $\nu(b)\in X_*(T)_{\QQ,{\rm dom}}$ intersects $\mathcal{S}_{w,\mu}$ non-trivially.

This claim implies the theorem. Indeed, in the context of loop groups, Theorem 1.4, Proposition 5.5, and Lemma 6.11 of \cite{trunc1} together imply the following. Let $w \in \leftexp{J}{W}$ and $b \in B(G,\mu)$ such that $\mathcal{S}_{w,\mu}$ intersects the $\sigma$-conjugacy class $b$. Let $\mathcal{S}_{w',\mu}$ be the truncation stratum associated with a fundamental alcove for $b$. Then $\mathcal{S}_{w',\mu} \subseteq \mathcal{S}_{w,\mu}$. Then the claim together with Corollary~\ref{corloopeo} translate this assertion into the one we want to prove. 

It remains to prove the claim. By Remark~\ref{DZipToPDiv} we see that $\Ascr_0^w\cap\mathcal{N}_b\neq\emptyset$ holds if and only if $K_1\dot w\dot x_{\mu}\mu(p)K_1 \cap b \ne \emptyset$. The same proof as for \cite{trunc1}, Theorem 1.1(2) shows that this condition is equivalent to the same condition with $K_1$ replaced by $\Ical$. Here $\Ical$ denotes the chosen Iwahori subgroup for $O_L = W(k)$, i.e. it is the inverse image of $B(k)$ under $K\rightarrow G(k)$. By Lemma~\ref{CombLemma} this condition can be expressed purely in terms of elements of the affine Weyl group. In particular, it is equivalent to the condition that $\Ical_{LG}\dot w\dot x_{\mu}\mu(z)\Ical_{LG}$ meets the $\sigma$-conjugacy class in $G(k\dlbrack z\drbrack)$ associated with $(\nu(b),\kappa(b))\in X_*(T)_{\mathbb{Q}}^{\Gamma}\times\pi_1(G)_{\Gamma}$. Here $\Ical_{LG}$ denotes the Iwahori subgroup for $O_L = k\dbrack{z}$, i.e. it is the inverse image of $B(k)$ under the projection $G(k[[z]])\rightarrow G(k)$. By \cite{trunc1}, Theorem 1.1(2) this condition is equivalent to the condition that the Newton stratum in the loop group $LG$ corresponding to $(\nu(b),\kappa(b))$ intersects $\mathcal{S}_{w,\mu}$.
\end{proof}

If $G$ is split we attach to the isogeny class of a $p$-divisible group $X$ with $\Dscr$-structure the unique isomorphism class of the $\Dscr$-zip of a minimal $p$-divisible group with $\Dscr$-structure isogenous to $X$. This yields an injective map
\begin{equation}\label{defwb}
B(G,\mu) \mono \leftexp{J}{W}, \qquad b \sends w(b)
\end{equation}
whose image consists of the $w \in \leftexp{J}{W}$ such that the corresponding Ekedahl-Oort stratum $\Ascr^w_0$ is minimal (or, equivalently in the split case, fundamental). We show below that for $b,b' \in B(G,\mu)$ one has $b' \leq b$ if and only if $w(b') \preceq w(b)$.

\begin{remark}\label{MinToNewton}
Theorem~\ref{corEONP} shows that $w(b)$ for $b \in B(G,\mu)$ can also be described as the unique minimal element (with respect to the partial order $\preceq$) in the set
\[
\leftexp{J}{W}_b := \set{w \in \leftexp{J}{W}}{\Ascr^w_0 \cap \Ncal_b \ne \emptyset}.
\]
Equivalently, it is the unique element of minimal length in $\leftexp{J}{W}_b$.
\end{remark}

\begin{corollary}\label{corclosleaf1}\label{NewtonEOSpec}
Let $G$ be split and let $b, b' \in B(G,\mu)$ be two elements with $b' \leq b$. Then
\[
\Ascr^{w(b')}_0 \subseteq \overline{\Ascr^{w(b)}_0}.
\]
In particular, for $b,b' \in B(G,\mu)$ one has $b' \leq b$ if and only if $w(b') \preceq w(b)$.
\end{corollary}

Although the last assertion is purely group-theoretical we do not know of any purely group-theoretic proof of this result.

\begin{proof}
The analogous statement for loop groups is shown in \cite{trunc1}, Proposition 5.7 and Lemma 5.3. The assertion we are interested in thus follows from Corollary~\ref{corloopeo}.
\end{proof}

\begin{proposition}\label{cor812}
Let $G$ be split. Let $w\in \leftexp{J}{W}$ and let
\[
{\rm Min}(w) := \set{w' \in \leftexp{J}{W}}{\text{$\Ascr_0^{w'}$ is minimal and $\Ascr_0^{w'} \subseteq \overline{\Ascr_0^{w}}$}}.
\]
Let $b \in B(G,\mu)$ be a maximal element in $\set{b' \in B(G,\mu)}{w(b') \in {\rm Min}(w)}$ (i.e., in the set of Newton points such that their corresponding Newton strata contain a minimal Ekedahl-Oort stratum $\Ascr_0^{w'}$ with $\Ascr_0^{w'} \subseteq \overline{\Ascr_0^{w}}$). Then the Newton stratum $\Ncal_b$ contains the generic point of some irreducible component of $\Ascr_{0}^{w}$.
\end{proposition}

In the Siegel case this result is shown by Harashita in \cite{H3}. In this case Ekedahl and van der Geer have shown for $p > 2$ that each Ekedahl-Oort stratum which is not contained in the supersingular locus is irreducible (\cite{EkedahlvdGeer}, Theorem 11.5). In particular, there is a unique generic Newton polygon in each Ekedahl-Oort stratum $\Ascr_{0}^{w}$ of $\mathcal{A}_g$. Then the above result determines this Newton polygon.

\begin{proof}
The sets of generic Newton points of irreducible components of $\Ascr_{0}^{w}$ and of $\overline{\Ascr_{0}^{w}}$ coincide. By Theorem~\ref{corEONP}, the set of Newton points of elements of $\overline{\Ascr_{0}^{w}}(\kgbar)$ is equal to the set of Newton points of points $\xi \in \overline{\Ascr_{0}^{w}}$ which are in a minimal Ekedahl-Oort stratum. By the definition of minimality all $p$-divisible groups in one minimal Ekedahl-Oort stratum are isomorphic and in particular in the same Newton stratum. It is equal to the $\sigma$-conjugacy class $[\dot x]$ of one (equivalently, any) representative $\dot x$ of a fundamental element $x\in\widetilde{W}$ corresponding to this minimal Ekedahl-Oort stratum. Let now $b \in B(G,\mu)$ be maximal among these $\sigma$-conjugacy classes $[\dot x]$, and let $Y$ be an irreducible component of $\overline{\Ascr_{0}^{w}}$ containing a point of $\Ncal_b$. By the Grothendieck specialization theorem~\eqref{NewtonSpec} the generic Newton stratum in $Y$ is greater or equal to $b$. By maximality of $b$ they have to be equal.
\end{proof}


\section{Properties of Ekedahl-Oort strata}\label{sec8}

\subsection{Dimensions of Ekedahl-Oort strata}

\begin{theorem}\label{thmEOnonempty}
Every Ekedahl-Oort stratum $\Ascr_0^{w}\subseteq \Ascr_0$ for $w\in {}^{J}W$ is non-empty.
\end{theorem}

\begin{proof}
It is enough to show that the morphism
\[
\zeta \otimes \id_{\kgbar}\colon \Ascr_0 \otimes \kgbar \rightarrow [G_{\kappa} \backslash X_J]  \otimes \kgbar
\]
is surjective. As it is open, this follows as soon as we know that the unique closed point of $[G_{\kappa} \backslash X_J] \otimes \kgbar$ is in the image. The underlying topological space of $[G_{\kappa} \backslash X_J] \otimes \kgbar$ is $\leftexp{J}{W}$ with the topology induced by the partial order $\preceq$ (Proposition~\ref{DescribeZJ}). Its unique closed point is the superspecial element $w=1$. By Proposition~\ref{remminbasic} this is in the image.
\end{proof}

In \cite{We2} it is already shown that all strata are nonempty if $B$ is a totally real field extension of $\QQ$.

\begin{corollary}\label{DimOfEO}
Each Ekedahl-Oort stratum $\Ascr^w_0$ is equi-dimensional of dimension $\ell(w)$.
\end{corollary}

For $p > 2$ this has also been proved by Moonen \cite{Mo2} using the
results of \cite{Mo1} and \cite{We2} and under the assumption that the stratum is non-empty.
In the Siegel case, this corollary has been shown by Oort in \cite{Oo1}, where we use~\eqref{LengthElSeq} to translate his expression for the dimension in terms of elementary sequences into the length of $w$.

\begin{proof}
We know that the $G(\kgbar)$-orbit $O^w$ corresponding to $w \in {}^JW$ has codimension $\dim(\Par_J) - \ell(w)$ in $X_J \otimes \kgbar$ (\cite{MW}). Therefore, the same holds for the substack $[(G \otimes \kgbar)\backslash O^w]$ of $[G_{\kappa} \backslash X_J] \otimes \kgbar$. As $\zeta$ is universally open, it respects codimension and all its fibers $\Ascr_0^w$ are equi-dimensional of dimension
\begin{align*}
\dim(\Ascr_0^w) &= \dim(\Ascr_0) - \codim(\Ascr_0^w,\Ascr_0) =
\dim(\Par_J) - (\dim(\Par_J) - \ell(w))\\
&= \ell(w).\qedhere
\end{align*}
\end{proof}

\subsection{Ekedahl-Oort strata are smooth and quasi-affine} 
For the sake of brevity, we call a level-$n$-truncated $p$-divisible group with \mbox{$\Dscr$-structure} a $\Dscr$-${\rm BT}_n$.

\begin{proposition}\label{EOSmooth}
Each Ekedahl-Oort stratum $\Ascr^w_0$ is smooth.
\end{proposition}

This result has been proved by Vasiu (for arbitrary $p$) in \cite{Va}~Theorem~5.3.1. For the convenience of the reader we include a quick proof for $p > 2$.

\begin{proof}[Proof of Proposition \ref{EOSmooth} for $p > 2$]
Let $\Bscr\Tscr_{\Dscr,1}$ be the algebraic stack of $\Dscr$-${\rm BT}_1$s over $\kgbar$ (see \cite{We2}~(1.7) and~(6.3) for its definition). For $w \in \leftexp{J}{W}$ let $\Mline^w$ by a $\Dscr$-zip over $\kgbar$ whose isomorphism class corresponds to $w$ under the bijection~\eqref{IsoClassesDZip}. Let $X^w$ be the corresponding $\Dscr$-${\rm BT}_1$. Then $X^w$ defines a point $\Spec \kgbar \to \Bscr\Tscr_{\Dscr,1}$. Its residue gerbe (in the sense of \cite{LM}~(11.1)) is the smooth locally closed substack $\Bscr\Tscr^w_{\Dscr,1}$ classifying $\Dscr$-${\rm BT}_1$s that are locally for the fppf-topology isomorphic to $X^w$.

We attach to a point $\APlus$ of $\Ascr_0$ the $p$-torsion of $A$ endowed with the $O_B$-action induced by $\iota$ and the isomorphism $A[p] \iso A[p]\vdual$ induced by $\lambda$. This yields a morphism $\Psi\colon \Ascr_0 \to \Bscr\Tscr_{\Dscr,1}$ which is smooth if $p > 2$ (\cite{We2}~(6.4)). Therefore $\Psi^{-1}(\Bscr\Tscr^w_{\Dscr,1})$ is smooth for $p > 2$. Now $\Ascr_0^w$ and $\Psi^{-1}(\Bscr\Tscr^w_{\Dscr,1})$ have the same underlying topological space because they have the same $k$-valued points for every perfect field extension $k$ of $\kgbar$. As $\Ascr^w_0$ is reduced by definition, we obtain that $\Ascr^w_0 = \Psi^{-1}(\Bscr\Tscr^w_{\Dscr,1})$ is smooth.
\end{proof}

\begin{remark}\label{LinearChar2}
Assume that $(O_B,\star)$ consists only of factors of type (AL) (Remark~\ref{CasesOB}). In this case the formal smoothness of morphisms between deformations of truncated $p$-divisible groups with $\Dscr$-structure in~\cite{We2}, which is used for the smoothness of $\Psi$ in the proof of Proposition~\ref{EOSmooth}, can be reformulated into a statement about formal smoothness of morphisms between deformations of truncated $p$-divisible groups with $O_K$-action ($K$ an unramified extension of $\QQ_p$) but without polarization. In this case the argument in~\cite{We2} does not use $p > 2$ and the morphism $\Psi$ in the proof of Proposition~\ref{EOSmooth} is also smooth for $p = 2$.
\end{remark}

We prove now that every Ekedahl-Oort stratum is quasi-affine. In the Siegel case this has been shown by Oort in~\cite{Oo1}. Our argument in the general case is similar but instead of working with canonical filtrations we use a result of Vasiu. We start with the following lemma.

\begin{lemma}\label{QuasiaffineLocal}
Let $g\colon X \to Y$ be a finite locally free surjective morphism of schemes. Then $X$ is quasi-affine if and only if $Y$ is quasi-affine.
\end{lemma}

\begin{proof}
The condition on $Y$ is clearly sufficient. Now assume that $X$ is quasi-affine. As $X$ is quasi-compact and separated and as $g$ is surjective and universally closed, $Y$ is quasi-compact and separated. Moreover $\Oscr_X = g^*\Oscr_Y$ is ample by hypothesis. Therefore $\Oscr_Y$ is ample because $g$ is finite locally free surjective (\cite{EGA}~II~(6.6.3)). This shows that $Y$ is quasi-affine.
\end{proof}

\begin{theorem}\label{EOQuasiAffine}
Assume that $\Ascr_0$ is a scheme. Then every Ekedahl-Oort stratum $\Ascr^w_0$ ($w \in \leftexp{J}{W}$) is quasi-affine.
\end{theorem}

\begin{proof}
Let $\APlus$ be the restriction of the universal family over $\Ascr_0$ to $\Ascr^w_0$ and let $\Mline = (M,C,D,\varphi_0,\varphi_1)$ be its attached $\Dscr$-zip. By \cite{Va}~Theorem~5.3.1, we find a finite locally free surjective morphism $g\colon S \to \Ascr^w_0$ such that $g^*\Mline$ is isomorphic to a constant $\Dscr$-zip. In particular, $g^*M$ and $g^*C$ are globally free $\Oscr_S$-modules. By the definition of $C$ this shows that $g^*f_*\Omega^1_{A/\Ascr^w_0}$ is a globally free $\Oscr_S$-module, where $f\colon A \to \Ascr^w_0$ denotes the structure morphism. Setting $\omega_A := \det(f_*\Omega^1_{A/\Ascr^w_0})$ we see that $g^*\omega_A$ is a trivial line bundle. On the hand we know by \cite{Lan}~Theorem~7.2.4.1 that $\omega_A$ is ample. Thus $g^*\omega_A$ is ample and $S$ is quasi-affine. Therefore $\Ascr^w_0$ is quasi-affine by Lemma~\ref{QuasiaffineLocal}.
\end{proof}

\begin{remark}
In~\cite{Va}~Theorem~5.3.1 it is shown that the inclusion $\Ascr^w_0 \mono \Ascr_0$ is quasi-affine. This also follows from Theorem~\ref{EOQuasiAffine} or from \cite{PWZ}~Theorem~12.7. In~\cite{NVW}~Theorem~6.3 (for $p \geq 5$) and in~\cite{Yatsyshyn} (in general) the much stronger result is shown that each Ekedahl-Oort stratum $\Ascr^w_0$ is pure in $\Ascr_0$ (i.e., the inclusion $\Ascr^w_0 \mono \Ascr_0$ is an affine morphism). This result neither implies nor is implied by Theorem~\ref{EOQuasiAffine}.
\end{remark}


\section{Non-emptiness of Newton strata}
The previous results allow us to deduce the non-emptiness of all Newton strata which has been conjectured by Fargues (\cite{Fargues} Conjecture 3.1.1) and Rapoport (\cite{RapoportGuide}~Conjecture~7.1). 

\begin{theorem}\label{NewtonNonEmpty}
For all $b \in B(G,\mu)$ the Newton stratum $\Ncal_b$ in $\Ascr_{\Dscr}$ is non-empty.
\end{theorem}

\begin{theorem}\label{IntegralManin}
For every algebraically closed field extension $k$ of $\kappa$ and every $p$-divisible group with $\Dscr$-structure $(X_0,\iota_0,\lambda_0)$ there exists $\APlus \in \Ascr_0(k)$ such that $(X_0,\iota_0,\lambda_0)$ is isomorphic to $(A,\iota,\lambda)[p^{\infty}]$ (as $p$-divisible groups with $\Dscr$-structure).
\end{theorem}

We first show that both theorems are equivalent.

\begin{proof}[Proof of the equivalence of Theorem~\ref{NewtonNonEmpty} and Theorem~\ref{IntegralManin}]
The map \mbox{$C(G) \to B(G)$} induces a surjection $C(G,\mu) \to B(G,\mu)$ (Section~\ref{CGBG}). Thus the non-emptiness of all Newton strata means that for every $p$-divisible group $(X_0,\iota_0,\lambda_0)$ with $\Dscr$-structure over $\kgbar$ there exists $\APlus \in \Ascr_0(\kgbar)$ such that $(A,\iota,\lambda)[p^{\infty}]$ and $(X_0,\iota_0,\lambda_0)$ are isogenous (as $p$-divisible groups with $\Dscr$-structure). Thus Theorem~\ref{IntegralManin} implies Theorem~\ref{NewtonNonEmpty}.

Conversely, let $(X_0,\iota_0,\lambda_0)$ be a $p$-divisible group with $\Dscr$-structure and let $b \in B(G,\mu)$ be its Newton point. If Theorem~\ref{NewtonNonEmpty} holds, there exists $\APlusOne \in \Ascr_0(k)$ and an isogeny of $p$-divisible groups with $\Dscr$-structure $f\colon (X_0,\iota_0,\lambda_0) \to (A_1,\iota_1,\lambda_1)[p^{\infty}]$. Then Theorem~\ref{IntegralManin} follows by Lemma~\ref{IsogAV}.
\end{proof}

For the proof of Theorem~\ref{NewtonNonEmpty} we need some preparations. Let $K$ be a number field and let $H$ be a reductive (and hence connected) group over $K$. Let $\Sigma$ be a finite set of places of $K$ containing all archimedean places. We let $\Sha^{\Sigma}(K,H)$ be the kernel of the canonical map $H^1(K,H) \to \prod_{v \notin\Sigma}H^1(K_v,H)$, where $v$ runs through all places of $K$ which are not in $\Sigma$. Let $\lambda = \lambda_H$ be the restriction of the canonical map $H^1(K,H) \to \prod_{v \mid \infty}H^1(K_v,H)$ to $\Sha^{\Sigma}(K,H)$. Consider the following condition.
\begin{spacedlist}
\item[(S)]
The map
\[
\lambda = \lambda_H\colon \Sha^{\Sigma}(K,H) \to \prod_{v \mid \infty}H^1(K_v,H)
\]
is surjective.
\end{spacedlist}
If $\Sigma$ and $\Sigma'$ are finite sets of places as above with $\Sigma \subseteq \Sigma'$ and $(H,\Sigma)$ satisfies Condition~(S), then $(H,\Sigma')$ satisfies Condition~(S).


\begin{lemma}\label{SProp}
Let $H$ be a reductive group over a number field $K$ and let $\Sigma$ be a finite set of places containing all archimedean places.
\begin{assertionlist}
\item
Let $L$ be a finite extension of $K$, let $H = \Res_{L/K}H_0$ be the restrictions of scalars from a reductive algebraic group $H_0$ over $L$ and let $\Sigma_0$ be the set of all places of $L$ lying over some place in $\Sigma$. Then $(H_0,\Sigma_0)$ satisfies Condition~(S) if and only if $(H,\Sigma)$ satisfies Condition~(S).
\item
If $H$ is simply connected, then Condition~(S) is satisfied for all $\Sigma$.
\item
Let the derived group $H^{\der}$ of $H$ be simply connected and set $D := H/H^{\der}$. If $(D,\Sigma)$ satisfies Condition~(S), then $(H,\Sigma)$ satisfies Condition~(S).
\end{assertionlist}
\end{lemma}

\begin{proof}
The first assertion follows from Shapiro's lemma. Assertion~(2) follows from the fact that for every reductive group $H$ over a number field $K$ the canonical morphism
\begin{equation}\label{InftySurjective}
H^1(K,H) \to \prod_{v \mid \infty}H^1(K_v,H)
\end{equation}
is surjective (\cite{PR_AGNT}~Prop.~6.17) and that for simply connected groups $H^1(K_v,H) = 0$ for all finite places $v$ of $K$  (\cite{PR_AGNT}~Theorem~6.4) which implies $\Sha^{\Sigma}(K,H) = H^1(K,H)$. In fact, as simply connected groups also satisfy the Hasse principle, $\lambda_H$ is bijective if $H$ is simply connected.

To show assertion~(3) we consider the following commutative diagram
\[\xymatrix{
1 \ar[r] & \prod_{v \notin \Sigma}H^1(K_v,H) \ar[r]^{\iota} & \prod_{v \notin \Sigma}H^1(K_v,D) \\
H^1(K,H^{\der}) \ar[u] \ar[r] & H^1(K,H) \ar[u] \ar[r]^{\pi} & H^1(K,D) \ar[u] \\
\Sha^{\Sigma}(K,H^{\der}) \ar@{=}[u] \ar[r] \ar[d]_{\lambda_{H^{\der}}}^{\wr} & \Sha^{\Sigma}(K,H) \ar@^{(->}[u] \ar[r]^{\pi^{\Sigma}} \ar[d]_{\lambda} & \Sha^{\Sigma}(K,D) \ar@^{(->}[u] \ar[d]^{\lambda_D} \\
\prod_{v\mid\infty}H^1(K_v,H^{\der}) \ar[r]^{\rho} & \prod_{v\mid\infty}H^1(K_v,H) \ar[r] & \prod_{v\mid\infty}H^1(K_v,D)
}\]
All rows are exact sequences of pointed sets (for the third row this follows from the exactness of the second row and the equality $\Sha^{\Sigma}(K,H^{\der}) = H^1(K,H^{\der})$). Moreover, for all places $v$ of $K$ it is known that the local cohomology groups $H^1(K_v,H)$ carry abelian group structures functorial in $H$, thus the first and the last row are in fact exact sequences of abelian groups. By \cite{Bo_AbCoh}~Theorem~5.7 the map $\pi$ is surjective. This implies together with the injectivity of $\iota$ that $\pi^{\Sigma}$ is surjective.

Now let $y \in \prod_{v\mid\infty}H^1(K_v,H)$. As $\pi^{\Sigma}$ and $\lambda_D$ are surjective, there exists $x \in \Sha^{\Sigma}(K,H)$ such that $\lambda(x) - y = (\rho \circ \lambda_{H^{\der}})(z_0)$ for some $z_0 \in \Sha^{\Sigma}(K,H^{\der})$. Let $z \in \Sha^{\Sigma}(K,H)$ be the image of $z_0$. It remains to find an element $w \in \Sha^{\Sigma}(K,H)$ such that $\lambda(w) = \lambda(x) - \lambda(z)$. For this we choose a maximal torus $T$ of $H$ such that $x$ and $z$ are both in the image of $\eta\colon H^1(K,T) \to H^1(K,H)$ (this is always possible by \cite{Bo_AbCoh}~Theorem~5.11), say $x = \eta(x')$ and $z = \eta(z')$. Define $w := \eta(x' - z')$. As both compositions
\begin{gather*}
H^1(K,T) \to \prod_{v\notin\Sigma}H^1(K_v,T) \to \prod_{v\notin\Sigma}H^1(K_v,H),\\
H^1(K,T) \to \prod_{v\mid\infty}H^1(K_v,T) \to \prod_{v\mid\infty}H^1(K_v,H)
\end{gather*}
are homomorphisms of groups, one has $w \in \Sha^{\Sigma}(K,H)$ and $\lambda(w) = \lambda(x) - \lambda(z)$.
\end{proof}

Recall that we fixed an integral Shimura-PEL-datum $\Dscr = \bigl(B,\star,V,\lrangle, O_B, \Lambda, h\bigr)$ unramified at $p$. We denoted by $\Gbf = \Gbf(\Dscr)$ the attached $\QQ$-group and we let $[\mu] = [\mu(\Dscr)]$ be the attached $\Gbf(\QQbar_p)$-conjugacy class of minuscule cocharacters of $\Gbf_{\QQbar_p}$. For the proof of Theorem~\ref{NewtonNonEmpty} we may assume and do assume from now on that $B$ is a simple $\QQ$-algebra. Let $F$ be the center of $B$ and let $F_0 = F^{\star=\id}$. Then $F_0$ is a totally real number field and either $F = F_0$ (Case~C) or $F$ is a quadratic imaginary extension of $F_0$ (Case~A). Let $\Sigma$ be the set of places of $\QQ$ consisting of the infinite place and of all finite places $v$ such that there exists a place of $F_0$ over $v$ which is ramified in $F$. As $\Dscr$ is unramified at $p$, we have $p \notin \Sigma$.

\begin{lemma}\label{ChangeMu}
Let $[\mu']$ be any $\Gbf(\QQbar_p)$-conjugacy class of minuscule cocharacters of $\Gbf_{\QQbar_p}$. Then there exists an integral Shimura-PEL-datum $\Dscr'$ unramified at $p$ of the form $\Dscr' = \bigl(B,\star,V,\lrangle', O_B, \Lambda, h'\bigr)$ such that for every place $v \notin \Sigma$ the $(B,\star)$-skew hermitian spaces $(V,\lrangle)$ and $(V,\lrangle')$ are isomorphic over $\QQ_v$ and such that $[\mu(\Dscr')] = [\mu']$.
\end{lemma}

\begin{proof}
We identify the set of $\Gbf(\QQbar_p)$-conjugacy class of cocharacters of $\Gbf_{\QQbar_p}$ with the set of $\Gbf(\CC)$-conjugacy classes of cocharacters of $\Gbf_{\CC}$. The classification of PEL data over $\RR$ (\cite{Ko_ShFin}~\S4) shows that every conjugacy class of a minuscule cocharacter occurs as the conjugacy class attached to a real PEL datum of the form $(B_{\RR},\star,V_{\RR},\lrangle',h')$. Now skew-hermitian $(B_{\RR},\star)$-spaces (resp.~$(B,\star)$-spaces) of the same dimension as $V$ are classified by $H^1(\RR,\Gbf)$ (resp.~by $H^1(\QQ,\Gbf)$). Thus the lemma is shown if we show that the reductive $\QQ$-group $\Gbf$ and $\Sigma$ satisfy the Condition~(S).

By Lemma~\ref{SProp}~(3) it suffices to show that $D := \Gbf/\Gbf^{\rm der}$ and $\Sigma$ satisfy Condition~(S). We use the notation of Remark~\ref{StructureGbf} where we studied $D$. In case (C) one has $D = \GG_{m,\QQ}$ and thus Condition~(S) is trivially satisfied by Hilbert 90. In case (A) let $\Sigma_0$ be the set of places of $F_0$ consisting of the archimedean places and the places which are ramified in $F$. Let us assume for a moment that $(D_0,\Sigma_0)$ satisfies Condition~(S). Then by Lemma~\ref{SProp}~(1) the $\QQ$-torus $D' := \Res_{F_0/\QQ}D_0$ and $\Sigma$ satisfy Condition~(S).

If $n$ is even, then \eqref{DescribeDEven} shows that $(D,\Sigma)$ satisfies Condition~(S). If $n$ is odd, we use the exact sequence~\eqref{DescribeDOdd} and obtain a commutative diagram
\[\xymatrix{
\Sha^{\Sigma}(\QQ,D') \ar[r] \ar[d]_{\lambda_{D'}} & \Sha^{\Sigma}(\QQ,D) \ar[d]^{\lambda_D} \\
H^1(\RR,D') \ar[r] & H^1(\RR,D).
}\]
By our assumption, $\lambda_{D'}$ is surjective and again by Hilbert 90 the lower horizontal map is surjective. Thus $\lambda_D$ is surjective.

It remains to show that $D_0 = \Ker(N_{F/F_0}\colon \Res_{F/F_0}\GG_{m,F} \to \GG_{m,F_0})$ and $\Sigma_0$ satisfy Condition~(S). For every field extension $K$ of $F_0$ one has
\[
H^1(K,D_0) = K^{\times}/N_{F \otimes_{F_0} K/K}((F \otimes_{F_0} K)^{\times}).
\]
We thus have $\prod_{v\mid\infty}H^1((F_0)_v,D_0) = \prod_{\iota}\RR^{\times}/\RR^{>0}$, where $\iota$ runs through the set $I$ of embeddings of $F_0$ into $\RR$. For $x \in F_0^{\times}$ we set $\sgn(x) := (\sgn \iota(x))_{\iota \in I} \in \{\pm 1\}^I$. Now choose for $\eps \in \{\pm 1\}^I$ some unit $x \in O_{F_0}^{\times}$ with $\sgn(x) = \eps$. Then by local class field theory $x$ is a norm for the extension $F_v := F \otimes_{F_0} (F_0)_v$ of $(F_0)_v$ for all unramified places $v$ of $F_0$. 
\end{proof}

\begin{proof}[Proof of Theorem~\ref{NewtonNonEmpty}]
Let $b\in B(G,\mu)$. By Theorem~\ref{thmexfundalc} there is a minuscule $\mu'$ with the following property. Let $\Dscr'$ be a corresponding Shimura-PEL-datum, unramified at $p$, as in Lemma~\ref{ChangeMu} (in particular $[\mu(\Dscr')] = [\mu']$). Then the Newton stratum $\Ncal'_b$ in $\Ascr_{\Dscr'}$ contains a fundamental Ekedahl-Oort stratum. This Ekedahl-Oort stratum is nonempty by Theorem~\ref{thmEOnonempty}. Let $\APlus$ be the abelian variety with $\Dscr'$-structure associated with a $k$-valued point in $\Ncal'_b$. Let $(X,\iota,\lambda)$ be a $p$-divisible group with $\Dscr$-structure in the isogeny class determined by $b$. Then there is a quasi-isogeny between $X$ and $A[p^{\infty}]$ compatible with the $O_B$-action and respecting the polarizations up to a scalar. Lemma \ref{IsogAV} thus implies that there is an abelian variety with $\Dscr$-structure whose $p$-divisible group is $(X,\iota,\lambda)$, which proves the theorem. 
\end{proof}


\appendix
\section{Appendix on Coxeter groups, reductive group schemes and quotient stacks}
In this appendix we fix some conventions and recall results on Coxeter groups, on reductive group schemes and on quotient stacks that are used in the main text.

\subsection{Coset representatives of Coxeter groups}\label{RecallCoxeter}
Let $W$ be a Coxeter group and $I$ its generating set of simple reflections. Let $\ell$ denote the length function on $W$.

Let $J$ be a subset of $I$. We denote by $W_J$ the subgroup of $W$ generated by $J$ and by $W^J$ (respectively $\leftexp{J}{W}$) the set of elements $w$ of $W$ which have minimal length in their coset $wW_J$ (respectively $W_Jw$). Then every $w \in W$ can be written uniquely as $w=w^{J}w_{J}= w'_J \leftexp{J}{w}$ with $w_J, w'_J \in W_J$, $w^J\in W^J$ and $\leftexp{J}{w} \in \leftexp{J}{W}$, and $\ell(w) = \ell(w_J) + \ell(w^J) = \ell(w'_J) + \ell(\leftexp{J}{w})$ (see \cite{DDPW}, Proposition 4.16). In particular, $W^J$ and $\leftexp{J}{W}$ are systems of representatives for $W/W_J$ and $W_J\backslash W$ respectively.

Furthermore, if $K$ is a second subset of $I$, let $\doubleexp{J}{W}{K}$ be the set of $w \in W$ which have minimal length in the double coset $W_JwW_K$. Then $\doubleexp{J}{W}{K}=\leftexp{J}{W}\cap W^K$ and $\doubleexp{J}{W}{K}$ is a system of representatives for $W_J\backslash W/W_K$ (see \cite{DDPW}~(4.3.2)).

\subsection{Bruhat order}\label{BruhatOrder}
We let $\leq$ denote the Bruhat order on~$W$. This natural partial order is characterized by the following property: For $x, w\in W$ we have $x\leq w$ if and only if for some (or, equivalently, any) reduced expression $w=s_{i_{1}}\cdots s_{i_{n}}$ as a product of simple reflections $s_i\in I$, one gets a reduced expression for $x$ by removing certain $s_{i_{j}}$ from this product. More information about the Bruhat order can be found in \cite{BjoernerBrenti}, Chapter 2. The set $\leftexp{J}{W}$ can be described as
\begin{equation}\label{eq:iwchar}
\leftexp{J}{W}=\set{w \in W}{\text{$w < sw$ for all $s \in J$}}
\end{equation}
(see \cite{BjoernerBrenti}, Definition 2.4.2 and Corollary 2.4.5).

\subsection{Reductive group schemes, maximal tori, and Borel subgroups}\label{Reductive}
Let $S$ be a scheme. A \emph{reductive group scheme over $S$} is a smooth affine group scheme $G$ over $S$ such that for every geometric point $s \in S$ the geometric fiber $G_{\sbar}$ is a connected reductive algebraic group over $\kappa(\sbar)$.

Let $G$ be a reductive group scheme over $S$. A \emph{maximal torus of $G$} is a closed subtorus $T$ of $G$ such that $T_{\sbar}$ is a maximal element in the set of subtori of $G_{\sbar}$ for all $s \in S$. By a theorem of Grothendieck (\cite{SGA3}~Exp.~XIV,~3.20) maximal tori of $G$ exist Zariski locally on $S$. A \emph{Borel subgroup} of $G$ is a closed smooth subgroup scheme $B$ of $G$ such that for all $s \in S$ the geometric fiber $B_{\sbar}$ is a Borel subgroup of $G_{\sbar}$ in the usual sense (i.e., a maximal smooth connected solvable subgroup). A reductive group scheme over $S$ is called \emph{split} if there exists a maximal torus $T$ of $G$ such that $T \cong \GG_{m,S}^r$ for some integer $r \geq 0$.

\subsection{Parabolic subgroups and Levi subgroups}\label{Parabolic}
A smooth closed subgroup scheme $P$ of $G$ is called \emph{parabolic subgroup of $G$} if the fppf quotient $G/P$ is representable by a smooth projective scheme or, equivalently by \cite{SGA3}~Exp.~XXII~(5.8.5), if $G_{\sbar}/P_{\sbar}$ is proper for all $s \in S$. Every Borel subgroup of $G$ is a parabolic subgroup. The \emph{unipotent radical} of $P$, denoted by $U_P$, is the largest smooth normal closed subgroup scheme with unipotent and connected fibers. It exists by \cite{SGA3}~Exp.~XXII,~(5.11.4). If $P$ contains a maximal torus $T$ of $G$, there exists a unique reductive closed subgroup scheme $L$ of $P$ containing $T$ such that the canonical homomorphism $L \to P/U_P$ is an isomorphism (loc.~cit.). Any such subgroup $L$ is called a \emph{Levi subgroup of $P$}.

The functor that sends an $S$-scheme $T$ to the set of Borel (resp.~parabolic) subgroups of $G \times_S T$ is representable by a smooth projective $S$-scheme by~\cite{SGA3}~Exp.~XXVI,~3. We call the representing scheme $\Bor_G$ (resp.~$\Par_G$). The functor that attaches to an $S$-scheme $T$ the set of pairs $(P,L)$, where $P$ is a parabolic subgroup of $G \times_S T$ and $L$ is a Levi subgroup of $P$ is representable by a smooth quasi-projective $S$-scheme (loc.~cit.).

If $S$ is local, $G$ is called \emph{quasi-split} if there exists a Borel subgroup of $G$. Every split reductive group scheme is quasi-split. If $S = \Spec R$, where $R$ is a henselian local ring whose residue field $k$ has cohomological dimension $\leq 1$ (e.g., if $k$ is finite), then $G$ is quasi-split. Indeed, the special fiber $G_k$ has a Borel subgroup $B_k$ by \cite{Serre_GalCoh}~III, \S2.3. As the scheme of Borel subgroups $\Bor_G$ is smooth, there exists a lifting of $B_k$ to a Borel subgroup of $G$.

\subsection{Weyl groups and types of parabolic subgroups over connected base schemes}\label{Weylgroups}
Let $G$ be a reductive group over an algebraically closed field, let $B$ be a Borel subgroup of $G$, and let $T$ be a maximal torus of~$B$. Let $W(T) := \Norm_G(T)/T$ denote the associated Weyl group, and let $I(B,T) \subset W(T)$ denote the set of simple reflections defined by~$B$. Then $W(T)$ is a Coxeter group with respect to the subset $I(B,T)$. 

A priori this data depends on the pair $(B,T)$. However, any other such pair $(B',T')$ is obtained by conjugating $(B,T)$ by some element $g\in G$ which is unique up to right multiplication by~$T$. Thus conjugation by $g$ induces isomorphisms $W(T) \stackrel{\sim}{\to} W(T')$ and $I(B,T) \stackrel{\sim}{\to} I(B',T')$ that are independent of~$g$. Moreover, the isomorphisms associated to any three such pairs are compatible with each other. Thus $W := W_G := W(T)$ and $I := I(B,T)$ for any choice of $(B,T)$ can be viewed as instances of ``the'' Weyl group and ``the'' set of simple reflections of $G$, in the sense that up to unique isomorphisms they depend only on~$G$.

Now let $G$ be a quasi-split reductive group scheme over a connected scheme $S$. Then we obtain for any geometric point $\sbar \to S$ the Weyl group and the set of simple reflections $(W_{\sbar},I_{\sbar})$ of $G_{\sbar}$. The algebraic fundamental group $\pi_1(S,\sbar)$ acts naturally on $W_{\sbar}$ preserving $I_{\sbar}$ (because $G$ is quasi-split), and every \'etale path $\gamma$ from $\sbar$ to another geometric point $\sbar'$ of $S$ yields an isomorphism of $(W_{\sbar},I_{\sbar}) \iso (W_{\sbar'},I_{\sbar'})$ that is equivariant with respect to the isomorphism $\pi_1(S,\sbar) \iso \pi_1(S,\sbar')$ induced by $\gamma$ (cf.~\cite{SGA3}~Exp.~XII, 2.1). In particular $(W_{\sbar},I_{\sbar})$ together with its action by $\pi_1(S,\sbar)$ is independent of the choice of $\sbar$ up to isomorphism. We denote it by $(W,I)$ and call it the \emph{Weyl system of $G$}.

If $P$ is a parabolic subgroup of $G$ and $s \in S$, the type $J_{\sbar} \subset I$ of the parabolic subgroup $P_{\sbar}$ of $G_{\sbar}$ is independent of $s \in S$ (\cite{SGA3}~Exp.~XXVI,~3) and we call $J := J_{\sbar}$ the \emph{type of $P$}. For a subset $J$ of $I$ we denote by $\Par_J$ the open and closed subscheme of $\Par$ parameterizing parabolic subgroups of type $J$. If $S$ is a semi-local scheme, $J$ and $\Par_J$ are defined over a finite \'etale covering of $S$ (\cite{SGA3}~Exp.~XXIV, 4.4.1).

For simplicity assume that $S$ is local. Let $J, K \subseteq I$ be subsets and let $S_1 \to S$ be the finite \'etale extension over which $J$ and $K$ are defined. Let $w \in \doubleexp{J}{W}{K}$. For every $S_1$-scheme $S'$ and for every parabolic subgroup $P$ of $G_{S'}$ of type $J$ and every parabolic subgroup $Q$ of $G_{S'}$ of type $K$ we write
\[
\relpos(P,Q) = w
\]
if there exists an fppf-covering on $S'' \to S'$, a Borel subgroup $B$ of $G_{S''}$ and a split maximal torus $T$ of $B$ such that $P_{S''}$ contains $B$ and $Q_{S''}$ contains $\leftexp{\dot w}{B}$, where $\dot w \in \Norm_{G_{S''}}(T)(S'')$ is a representative of $w \in W = \Norm_{G_{S''}}(T)(S'')/T(S'')$.

If $S' = \Spec k$ for an algebraically closed field, then $(P,Q) \sends \relpos(P,Q)$ yields a bijection between $G(k)$-orbits on $\Par_J(k) \times \Par_K(k)$ and the set $\doubleexp JWK$.

\subsection{One-parameter subgroups and their type}\label{one-ps}
Let $G$ be a reductive group scheme over a connected scheme $S$. A \emph{one-parameter subgroup} of $G$ is by definition a homomorphism of group schemes $\lambda\colon \GG_{m,S} \to G$, where $\GG_{m,S}$ denotes the multiplicative group over $S$. We identify the character group of $\GG_{m,S}$ with $\ZZ$. Composing $\lambda$ with the adjoint representation $G \to \GL(\Lie(G))$ yields a decomposition $\Lie(G) = \bigoplus_{n \in \ZZ} \gfr_n$. There is a unique parabolic subgroup $P(\lambda)$ of $G$ and a unique Levi subgroup $L(\lambda)$ of $P(\lambda)$ such that
\[
\Lie P(\lambda) = \bigoplus_{n \geq 0}\gfr_n, \qquad \Lie L(\lambda) = \gfr_0.
\]
Indeed, because of the uniqueness assertion we may show this locally for the \'etale topology and hence can assume that the image of $\lambda$ lies in a split maximal torus. Then the claim follows from~\cite{SGA3}~Exp. XXVI, 1.4 and 4.3.2. By loc.~cit.\ we may define $L(\lambda)$ also as the centralizer of $\lambda$.

The type of $P(\lambda)$ is also called the \emph{type of $\lambda$}. It depends only on the conjugacy class of $\lambda$.

\subsection{The example of the symplectic group}\label{SymplecticExample}
As an example consider a symplectic space $(V,\lrangle)$ of dimension $2g$ over a field $k$ and denote by $G = \GSp(V,\lrangle)$ the group of symplectic similitudes of $(V,\lrangle)$. Let $S_{2g}$ denote the symmetric group of the set $\{1,\dotsc,2g\}$. Then the Weyl system $(W,I)$ of $G$ is given by
\begin{equation}\label{WeylSymplectic}
\begin{aligned}
W &= \set{w \in S_{2g}}{\text{$w(i) + w(2g+1-i) = 2g+1$ for all $i = 1,\dots,g$}},\\
I &= \{s_1,\dots,s_g\},\quad\text{where}\quad
s_i =
\begin{cases} \tau_i\tau_{2g-i}, &\text{for $i = 1,\dots,g-1$;}\\
\tau_g, &\text{for $i = g$.}
\end{cases}
\end{aligned}
\end{equation}
Here $\tau_j$ denotes the transposition of $j$ and $j+1$. The length of an element $w \in W$ can be computed as follows
\begin{align*}
\ell(w) = &\ \#\set{(i,j)}{1 \leq i < j \leq g, w(i) > w(j)} \\
&\ + \#\set{(i,j)}{1 \leq i \leq j \leq g, w(i) + w(j) > 2g+1}.
\end{align*}
For every $d$-tuple $(x_1,\dots,x_d)$ of real numbers we denote by $(x_1,\dots,x_d){\uparrow}$ the $d$-tuple $(x_{\sigma(1)},\dots,x_{\sigma(d)})$ with $\sigma \in S_d$ such that $x_{\sigma(1)} \leq \dots \leq x_{\sigma(d)}$. Then the Bruhat order is given by
\[
w' \leq w \iff \forall\ 1\leq i \leq g: (w'(1),\dots,w'(i)){\uparrow} \leq (w(1),\dots,w(i)){\uparrow},
\]
where we compare tuples componentwisely.

To study the Siegel case we consider the subset $J := \{s_1,\dots,s_{g-1}\}$ of $I$. Then $W_J$ consists of those permutations $w \in W$ such that
$w(\{1,\dots,g\}) = \{1,\dots,g\}$. The map
\[
W_J \to S_g. \qquad w \mapsto w\restricted{\{1,\dots,g\}}
\]
is a group isomorphism. The set $\leftexp{J}{W}$ consists in this case of those elements $w \in W$ such that $w^{-1}(1) < w^{-1}(2) < \dots < w^{-1}(g)$. Of course, this implies $w^{-1}(g+1) < \dots < w^{-1}(2g)$.

It is convenient to give an alternative description of $\leftexp{J}{W}$. To $w \in \leftexp{J}{W}$ we attach $\epsilon = (\epsilon_i)_{1 \leq i \leq g} \in \{0,1\}^g$ by
\[
\epsilon_i := 
\begin{cases}
0,&\text{if $i \in \{w^{-1}(1),\dots,w^{-1}(g)\}$};\\
1,&\text{otherwise},
\end{cases}
\qquad i = 1,\dots,g.
\]
This yields a bijection $\leftexp{J}{W} \bijective \{0,1\}^g$. The length of such an element $\epsilon = (\epsilon_1,\dots,\epsilon_g)$ is equal to
\begin{equation}\label{LengthSymplectic}
\ell(\epsilon) = \sum_{i=1}^g i\epsilon_{g+1-i}.
\end{equation}
The set $\doubleexp JWJ \cong W_J \backslash W/W_J$ corresponds to the set of $W_J$-orbits of $\{0,1\}^g$, where $W_J = S_g$ acts by permuting the entries of $(\epsilon_i)_i \in \{0,1\}^g$. Thus $(\epsilon_i)_i \sends \#\set{i}{\epsilon_i = 1}$ yields a bijection
\begin{equation}\label{SymplecticDouble}
\doubleexp JWJ \bijective \{0,\dots,g\}.
\end{equation}

The set $\leftexp{J}{W}$ can also be described by elementary sequences in the sense of Oort (see~\cite{EkedahlvdGeer}~\S2; note that in loc.~cit.\ the set $W^J$ is considered which we identify with $\leftexp{J}{W}$ via $w \sends w^{-1}$): For $w \in \leftexp{J}{W}$ we define
\[
\varphi_w\colon \{0,\dots,g\} \to \ZZ_{\geq 0}, \qquad i \sends i - \#\set{1 \leq a \leq g}{w^{-1}(a) \leq i}.
\]
Then $\varphi_w$ is an elementary sequence, i.e., a map $\varphi\colon \{0,\dots,g\} \to \ZZ_{\geq 0}$ such that $\varphi(0) = 0$ and such that $\varphi(i-1) \leq \varphi(i) \leq \varphi(i-1) + 1$ for all $i = 1,\dots,g$. The map $w \sends \varphi_w$ yields a bijection between $\leftexp{J}{W}$ and the set $\Scal_g$ of elementary sequences. The tuple $(\epsilon_i)_i \in \{0,1\}^g$ corresponding to an element $w \in \leftexp{J}{W}$ can be described via the elementary sequence $\varphi_w$ by $\epsilon_i = \varphi_w(i) - \varphi_w(i-1)$. Using~\eqref{LengthSymplectic} an easy calculation shows
\begin{equation}\label{LengthElSeq}
\ell(w) = \sum_{i=1}^g \varphi_w(i).
\end{equation}
Finally, for $w,w' \in \leftexp{J}{W}$ we have $w' \leq w$ if and only if $\varphi_{w'}(i) \leq \varphi_w(i)$ for all $i = 1,\dots,g$.

Let us now describe parabolic subgroups of $G$ of type $J$. Let $S$ be any $k$-scheme. Then $V_S = V \otimes_k \Oscr_S$ is a free $\Oscr_S$-module of rank $2g$, and the base change of $\lrangle$ is a perfect alternating form on $V_S$. An $\Oscr_S$-submodule $\Lscr$ of $V_S$ is called \emph{Lagrangian} if $\Lscr$ is locally on $S$ a direct summand of $V_S$ of rank $g$ and if $\Lscr$ is totally isotropic. Attaching to a Lagrangian $\Lscr$ its stabilizer in $G_S$ yields a bijection between the set of Lagrangians in $V_S$ and the set of parabolic subgroups $P$ of $G \times_k S$ of type $J$ (with $J = \{s_1,\dots,s_{g-1}\}$).

Let $\Lscr$ and $\Lscr'$ be Lagrangians in $V_S$ and let $P$ and $P'$ be the corresponding parabolic subgroups of $G_S$ of type $J$. Then for $d \in \{0,\dots,g\} = \doubleexp JWJ$~\eqref{SymplecticDouble} we have
\begin{equation}\label{InterpretePos}
\relpos(P,P') = d \iff \text{$V_S/(\Lscr + \Lscr')$ is locally free of rank $g-d$}.
\end{equation}
Here the condition that $V_S/(\Lscr + \Lscr')$ is locally free (of some rank) is equivalent to the condition that fppf-locally $P$ and $Q$ contain a common maximal torus (\cite{MW}~Section~3.5).

\subsection{The underlying topological space of a quotient stack}\label{QuotStacks}
We recall some -- probably well known -- facts on quotient stacks. For lack of a better reference we refer to~\cite{We2}~(4.2)--(4.4). Let $k$ be a field, $\kbar$ an algebraic closure, $\Gamma := \Aut(\kbar/k)$ the profinite group of $k$-automorphisms of $\kbar$. Let $X$ be a $k$-scheme of finite type and let $H$ be a smooth affine group scheme over $k$ that acts on $X$. We assume that there are only finitely many $H(\kbar)$-orbits in $X(\kbar)$.

Let $[H \backslash X]$ be the algebraic quotient stack. To describe its underlying topological space we first recall that if $\leq$ is a partial order on a set $Z$, we can define a topology on $Z$ by defining a subset $U$ of $Z$ to be open if for all $u \in U$ and $z \in Z$ with $u \leq z$ one has $z \in U$.

Let $\Xi^{\rm alg}$ be the set of $H(\kbar)$-orbits in $X(\kbar)$. For $\Ocal, \Ocal' \in \Xi^{\rm alg}$ we set $\Ocal' \leq \Ocal$ if the closure of $\Ocal$ contains $\Ocal'$. This defines a partial order and hence a topology on $\Xi^{\rm alg}$. The group $\Gamma$ acts on $\Xi^{\rm alg}$ and we denote the set of $\Gamma$-orbits on $\Xi^{\rm alg}$ by $\Xi$. We endow $\Xi$ with the quotient topology. Then $\Xi$ is homeomorphic to the underlying topological space of $[H \backslash X]$.


\end{document}